\documentclass[11pt,leqno]{amsart}
\usepackage{amsmath,amsfonts,amsthm,amscd,amssymb,graphicx}
\usepackage{epstopdf}
\numberwithin{equation}{section}

\newcommand\s{\sigma}

\renewcommand\d{\partial}

\def\t{\tau}

\def\eps{\varepsilon }

\renewcommand\d{\partial}

\newcommand\C{\mathbb C}

\def\t{\tau}

\def\eps{\varepsilon}

\newcommand\br{\begin{remark}}
\newcommand\er{\end{remark}}
\newcommand\bp{\begin{pmatrix}}
\newcommand\ep{\end{pmatrix}}
\newcommand{\be}{\begin{equation}}
\newcommand{\ee}{\end{equation}}
\newcommand\ba{\begin{equation}\begin{aligned}}
\newcommand\ea{\end{aligned}\end{equation}}

\newcommand{\bap}{\begin{app}}
\newcommand{\eap}{\end{app}}
\newcommand{\begs}{\begin{exams}}
\newcommand{\eegs}{\end{exams}}
\newcommand{\beg}{\begin{example}}
\newcommand{\eeg}{\end{exaplem}}
\newcommand{\bpr}{\begin{proposition}}
\newcommand{\epr}{\end{proposition}}
\newcommand{\bt}{\begin{theorem}}
\newcommand{\et}{\end{theorem}}
\newcommand{\bc}{\begin{corollary}}
\newcommand{\ec}{\end{corollary}}
\newcommand{\bl}{\begin{lemma}}
\newcommand{\el}{\end{lemma}}
\newcommand{\bd}{\begin{definition}}
\newcommand{\ed}{\end{definition}}
\newcommand{\brs}{\begin{remarks}}
\newcommand{\ers}{\end{remarks}}

\newtheorem{theorem}{Theorem}[section]
\newtheorem{proposition}[theorem]{Proposition}
\newtheorem{corollary}[theorem]{Corollary}
\newtheorem{lemma}[theorem]{Lemma}

\theoremstyle{remark}
\newtheorem{remark}[theorem]{Remark}
\theoremstyle{definition}
\newtheorem{definition}[theorem]{Definition}

\newtheorem{example}[theorem]{Example}

\newcommand\cA{{\mathcal { A}}}
\newcommand\cB{{\mathcal  B}}

\newcommand\cK{{\mathcal  K}}

\newcommand\cN{{\mathcal  N}}

\newcommand\cS{{\mathcal S}}

\newcommand{\beq}{\begin{equation}}
\newcommand{\eeq}{\end{equation}}
\newcommand{\vp}{\varphi}

\newtheorem{thm}{Theorem}[section]
\newtheorem{prop}[thm]{Proposition}
\newtheorem{cor}[thm]{Corollary}
\newtheorem{lem}[thm]{Lemma}
\newtheorem{defn}[thm]{Definition}
\newtheorem{ass}[thm]{Assumption}

\newtheorem{rem}[thm]{Remark}

\newtheorem{exams}[thm]{Examples}
\newtheorem{notation}[thm]{Notation}
\numberwithin{equation}{section}

\def\({\left(\begin{array}{cccccc}}
\def\){\end{array}\right)}

\newcommand{\pf}{\begin{proof}}
\newcommand{\foorp}{\end{proof}}

\newcommand{\ud}{{\underline d}}
\newcommand{\uh}{{\underline h}}

\newcommand{\uc}{{\underline c}}

\newcommand{\uw}{{\underline w}}

\newcommand{\ua}{{\underline a}}

\newcommand{\bR}{\mathbb{R}}

\newcommand{\bQ}{\mathbb{Q}}
\newcommand{\bW}{\mathbb{W}}

\newcommand{\bC}{\mathbb{C}}

\newcommand{\uzeta}{\underline{\zeta}}

\newcommand{\ub}{\underline{b}}

\newcommand{\cZ}{\mathcal{Z}}
\title
{High-frequency stability of multidimensional
ZND detonations}

\author{Olivier Lafitte}
\address{ Universit\'e de Paris 13, LAGA and CEA Saclay, DM2S}
\email{lafitte@math.univ-paris13.fr}
\author{Mark Williams}
\address{University of North Carolina, Chapel Hill}
\email{ williams@email.unc.edu}
\thanks{Research of M.W. was partially supported by
NSF grants number DMS-0701201 and DMS-1001616}
\author{Kevin Zumbrun}
\address{Indiana University, Bloomington, IN 47405}
\email{kzumbrun@indiana.edu}
\thanks{Research of K.Z. was partially supported
under NSF grants no. DMS-0300487 and DMS-0801745.}

\begin{document}

\begin{abstract}
The rigorous study of spectral stability
of
strong detonations was begun by Erpenbeck in the 1960s.  Working with the Zeldovitch-von Neumann-D\"oring (ZND) model, he identified two fundamental classes of detonation profiles, referred to as those of decreasing (D) and increasing (I) type, which appeared to exhibit very different behavior with respect to high-frequency perturbations.  Using a combination of rigorous and non-rigorous arguments, Erpenbeck concluded that type I detonations were unstable to some oscillatory perturbations for which the (vector) frequency was of arbitrarily large magnitude, while type D detonations were stable provided the frequency magnitude was sufficiently high.  For type D detonations Erpenbeck's methods did not allow him to obtain a cutoff magnitude for stability that was \emph{uniform} with respect to frequency direction.  Thus, he left open the question whether the cutoff magnitude for stability might approach $+\infty$  as certain frequency directions were approached.  In this paper we show by quite different methods that for type D detonations there exists a uniform cutoff magnitude for stability independent of frequency direction.
By reducing the search for unstable frequencies to a bounded frequency set, the uniform cutoff  obtained here is a key step
toward the rigorous validation of a number of results in the computational detonation literature.
%

The detonation profile $P(x)$ is a stationary solution of the ZND system depending on the single spatial variable $x\in[0,+\infty)$, the reaction zone. The spectral stability of the profile is governed by a nonautonomous $5\times 5$ system of linear ODEs in $x$ depending on the perturbation frequency as a vector parameter.  Difficulties in the analysis are caused by the existence of frequency directions $\zeta$ for which two of the eigenvalues of this system cross at a particular point $x=x(\zeta)$ in the reaction zone. Such points $x(\zeta)$ are called \emph{turning points}.  A  necessary step in obtaining a uniform stability cutoff is to obtain explicit representations of the decaying solutions of the system that are uniformly valid for frequencies near turning point frequencies.  The main mathematical difficulty addressed here is to produce such uniform representations for frequencies near the particular turning point frequency $\zeta_\infty$ for which the associated turning point $x(\zeta_\infty)$ is $+\infty$.
Our uniform estimates are potentially
useful for general turning-point problems on unbounded spatial intervals,
both for spectral stability analysis as here and
for resolvent estimates toward linearized and
nonlinear stability or instability.
In particular, they are
needed
for the study of the related
multi-D inviscid limit problem.
\end{abstract}

\date{\today}
\maketitle

\textbf{This paper is dedicated to J. J. Erpenbeck, pioneer in the study of shock and detonation stability.}

\tableofcontents

\part{Introduction}

 The most commonly studied model of combustion is the Zeldovitch-von Neumann-D\"oring (ZND) system \eqref{ZND}, which couples the compressible Euler equations for a reacting gas (in which pressure and internal energy are allowed to depend on the mass fraction $\lambda$ of reactant) to a reaction equation that governs the finite rate at which $\lambda$ changes. In three space dimensions with coordinates $(x,y,z)$ the ZND equations for the unknowns $(v,\mathbf{u},S,\lambda)$ (specific volume, particle velocity $\mathbf{u}=(u_x,u_y,u_z)$, entropy, and mass fraction of reactant) are given by the $6\times 6$ system:

\begin{align}\label{ZND}
\begin{split}
&\partial_t v+\mathbf{u}\cdot \nabla v-v\nabla \cdot \mathbf{u}=0\\
&\partial_t \mathbf{u}+\mathbf{u}\cdot\nabla \mathbf{u}+v\nabla p =0\\
&\partial_t S+\mathbf{u}\cdot\nabla S=-r \Delta F /T:=\Phi\\
&\partial_t\lambda+\mathbf{u}\cdot\nabla \lambda =r,
\end{split}
\end{align}
where $p=p(v,S,\lambda)$ is pressure, $T$ is temperature, $\Delta F$ is the free energy increment, and $r(v,S,\lambda)$ is the reaction rate function.  A steady planar \emph{strong detonation profile} is a weak solution of this system depending only on $x$ with a jump (the stationary von Neumann shock) at $x=0$.
Without loss of generality we study profiles of the form $P(x)=(v,u,0,0,S,\lambda)$, where $u>0$ is the $x$-component of particle velocity.  The solution is constant and supersonic ($u>c_0$, where $c_0$ is the sound speed at $x$) in $x<0$, the quiescent zone, and satisfies a nonlinear system of ODEs in the subsonic reaction zone $x>0$.  In order to be a weak solution of \eqref{ZND} in a neighborhood of $x=0$, $P(x)$ must satisfy an appropriate Rankine-Hugoniot condition at $x=0$.  There is a well-defined limiting state $P_\infty=\lim_{x\to +\infty}P(x)$  with $\lambda(\infty)=0$, and the range of $u$ on $[0,\infty)$ is a compact subinterval of $(0,\infty)$.

The rigorous study of spectral stability for strong detonations was begun by Erpenbeck \cite{E1} in the 1960s.  Working with the ZND model,  in \cite{E2,E3} he identified two fundamental classes of detonation profiles, referred to as those of decreasing (D) and increasing (I) type, which appeared to exhibit very different behavior with respect to high-frequency perturbations.  Using a combination of rigorous and non-rigorous arguments, Erpenbeck concluded that type I detonations were unstable to some oscillatory perturbations for which the (vector) frequency was of arbitrarily large magnitude, while type D detonations were stable provided the frequency magnitude was sufficiently high.  For type D detonations Erpenbeck's methods did not allow him to obtain a cutoff magnitude for stability that was \emph{uniform} with respect to frequency direction.  Thus, he left open the question whether the cutoff magnitude for stability might approach $+\infty$  as certain frequency directions were approached.  In \cite{LWZ1} we identified the mathematical issues left unresolved in \cite{E2,E3} and provided proofs, together with certain simplifications and extensions, of the main conclusions
of \cite{E2,E3}, in particular, high-frequency
instability of type I detonations.
However, the paper \cite{LWZ1} also failed to resolve the above question for type D detonations.
In this paper we show by quite different methods that for type D detonations there exists a uniform cutoff magnitude for stability
independent of frequency direction, thus establishing
{\it high-frequency stability of type D detonations}.

 The spectral stability of
of a ZND profile
is governed by a nonautonomous $5\times 5$ system of linear ODEs in $x$ depending on the perturbation frequency $(\tau,\eps)$ as a parameter
\cite{E2,E3,LWZ1}:
\begin{align}\label{mainsys}
\frac{d\theta}{dx}=-\mathcal{G}^t(x,\tau,\eps)\theta\text{ on }x\geq 0.
\end{align}
Here the system defined by the matrix $\mathcal{G}(x,\tau,\eps)$ is obtained by linearizing the ZND system about the profile $P(x)$, and taking the Laplace transform in time and the Fourier transform in the transverse spatial variables $(y,z)$.  The matrix $\mathcal G^t$ is the transpose of $\mathcal G$, the variable $\tau\in\mathbb{C}$ is dual to time, and $\eps=\sqrt{\alpha^2+\beta^2}$, where $(\alpha,\beta)\in\bR^2$ is dual to $(y,z)$.\footnote{Thus, $\eps/2\pi$ is the transverse wavenumber.}  We have
\begin{align}\label{t1}
\mathcal G(x,\tau,\eps)=-A_x^{-1}(x)[\tau I+i\eps A_y(x)+B(x)],
\end{align}
for matrices $A_x$, $A_y$ and $B$ given in section \ref{coefficients}.  The $x$-dependence of these matrices enters entirely through the profile $P(x)$.\footnote{In fact, we show in section \ref{infinite}  that $P(x)$ can be expressed as $P(x)=\mathcal{P}(\lambda(x))$. } The reduction from a linearized system of dimension 6 to one of dimension 5 and from $(\alpha,\beta)$ to $\eps$ uses the rotational symmetry of $\mathcal{G}$ with respect to the transverse velocity components.\footnote{We refer to \cite{E1} and to the introduction of \cite{LWZ1} for the details of the derivation of \eqref{mainsys} and for additional background on the ZND system.}

 In  [E1] Erpenbeck defined a stability function $V(\tau,\eps)$  whose zeros in the right half plane $\Re\tau>0$ (``unstable zeros") correspond to perturbations of the steady profile $P(x)$ that grow exponentially in time.
 The computation of $V$ requires the evaluation within the reaction zone of the solution $\theta(x,\tau,\eps)$ of the equation \eqref{mainsys}
which satisfies the condition that $\theta$ remains bounded for fixed $(\tau,\eps)$ with $\Re\tau\geq 0$ as $x\to \infty$, and   $\theta$ decays exponentially to zero for $\Re\tau>0$ as $x\to\infty$.  As we will see, this condition determines $\theta$ uniquely up to a constant multiple.   We shall refer to  this solution $\theta$ as  \emph{the decaying solution}, even though it is merely bounded for certain purely imaginary $\tau$ values.

Following the notation of \cite{E2,E3}, we write $\tau$ as
\begin{align}\label{t2}
\tau=\zeta\eps
\end{align}
where $\zeta\in\{z\in\bC:\Re z\geq 0\}$ and $\eps>0$ is large.\footnote{In \cite{E2,E3} Erpenbeck used a decomposition $\tau=\eps\zeta+\nu$, but $\nu$ played no role in his treatment of type-D profiles and can be set equal to zero.}  Thus we can rewrite
equation \eqref{mainsys} as
\begin{align}\label{t3}
\begin{split}
&\frac{d\theta}{dx}=(\eps\Phi_0+\Phi_1)\theta\text{ where }\\
&\Phi_0(x,\zeta)=\{A_x^{-1}(x)\cdot (\zeta I+iA_y(x))\}^t\\
&\Phi_1(x)= \{ A_x^{-1}(x)B(x)\}^t.
\end{split}
\end{align}

The eigenvalues $\mu_j(x,\zeta)$ of the matrix $\Phi_0$, which is given in \eqref{phi0}, play a crucial role in all that follows.  They are
\begin{align}\label{t4}
\mu_1=-\kappa(\kappa\zeta+s)/\eta u,\;\;\;\mu_2=-\kappa(\kappa\zeta-s)/\eta u,\;\;\; \mu_3=\mu_4=\mu_5=\zeta/u,
\end{align}
where with $c_0^2=-v^2p_v(v,S,\lambda)$
\begin{align}\label{t5}
s(x,\zeta)=\sqrt{\zeta^2+c_0^2\eta}, \;\;\;\kappa(x)=\sqrt{1-\eta}=u/c_0.
\end{align}
The square root defining $s$, regarded as a function of $\zeta$, is taken to be the positive branch with branch cut the segment  $[-ic_0\sqrt{\eta},ic_0\sqrt{\eta}]$ on the imaginary axis. Thus, in particular, we have
\begin{align}\label{t5a}
\begin{split}
&s=|s| \text{ when } \zeta^2+c_0^2\eta>0\\
&s=i|s|\text{ when } \zeta^2+c_0^2\eta< 0\text{ and }\zeta=i|\zeta|\\
&s=-i|s|\text{ when } \zeta^2+c_0^2\eta< 0\text{ and }\zeta=-i|\zeta|.
\end{split}
\end{align}
One checks that only $\mu_1$ has, for $\Re\zeta>0$, negative real part; consequently, for a fixed $\zeta$ with $\Re\zeta>0$ the system \eqref{t3} has a one dimensional space of solutions that decay to zero as $x\to +\infty$.    The  eigenvectors corresponding to the $\mu_j$, $j=1,\dots,5$, are the respective columns of the matrix
\begin{align}\label{t6}
T(x,\zeta)=\begin{pmatrix}T_1&T_2&T_3&T_4&T_5\end{pmatrix}=\begin{pmatrix}\frac{ms}{\kappa u}&-\frac{ms}{\kappa u}&-\frac{im}{1-\eta}&0&0\\\frac{\zeta}{u}&\frac{\zeta}{u}&i&0&0\\-i&-i&\frac{\zeta}{u}&0&0\\\frac{-\kappa p_S s}{um}&\frac{\kappa p_S s}{um}&0&1&0\\\frac{-\kappa p_\lambda s}{um}&\frac{\kappa p_\lambda s}{um}&0&0&1\end{pmatrix}
\end{align}
where $m=\frac{u}{v}$ is the mass flux.

From the formulas \eqref{t4}, \eqref{t5} we see that for any fixed value of $\zeta$, the eigenvalues $\mu_1$ and $\mu_2$ are distinct except at values $x=x(\zeta)$ where $s^2(x,\zeta)=\zeta^2+c_0^2\eta(x)=0$; at such values the first and second columns of $T$ are parallel. The eigenvalues $\mu_2$ and $\mu_3$ are distinct except at $x$ values where $\zeta=u$, and then the second and third rows of $T$ are clearly parallel.   For all other values of $x$ the matrix $T(x,\zeta)$ is invertible.\footnote{For certain types of profiles and choices of $\zeta$, there may be more than one $x-$value where $T(x,\zeta)$ is singular.}


A complex number $\zeta$ with $\Re\zeta\geq 0$ is defined in \cite{E2} to be of Class III or Class II respectively, when there exists $x_*\in [0,\infty]$ such that
$s(x_*,\zeta)=0$ or $\zeta=u(x_*)$.
All other $\zeta$  are said to be of Class I.  Thus we have
\begin{align}\label{t7}
\begin{split}
&\text{Class III}  = \{\zeta:\Re\zeta=0\text{ and }\min_x (c_0\eta^{\frac{1}{2}})\leq |\zeta|\leq \max_x (c_0\eta^{\frac{1}{2}})\}\\
&\text{Class II}= \{\zeta:\Im \zeta=0\text{ and }\min_x u\leq \zeta\leq \max_x u\}\\
&\text{Class I}=\{ \text{all remaining }\zeta\in\bC\text{ with }\Re\zeta\geq 0\}.
\end{split}
\end{align}
Class III (resp. II) consists of two (resp. one) bounded closed interval(s), and the minima appearing in \eqref{t7} are positive.  In contrast to \cite{E2,E3} we are able to treat Class I and Class II frequencies by a single argument.  The argument is based on the
observation that for such
$\zeta$ the eigenvalue $\mu_1$ remains well separated from the others.  So for us the important partition of $\zeta-$ space is
\begin{align}
\{\zeta\in\bC:\Re\zeta\geq 0\}=\mathrm{III}\cup \mathrm{III}^c.
\end{align}

\begin{notation}
1) When working with Class III values of $\zeta$ we will usually suppose $\zeta=i|\zeta|$.  The same results hold with the same proofs when $\zeta=-i|\zeta|$, but certain formulas change slightly.  Thus, for some statements it is helpful to define
\begin{align}\label{plus}
\text{Class III}_+=\text{Class III}\cap \{\zeta=i|\zeta|\}.
\end{align}
We let III$^o_+$ denote the interior of the closed interval III$_+$.

2) Henceforth, we shall use the parameter $h=1/\eps$ instead of $\eps$ and, in view of \eqref{t2}, we shall denote the stability function by $V(\zeta,h)$ instead of $V(\tau,\eps)$ and the decaying solution of \eqref{t3} by $\theta(x,\zeta,h)$ instead of $\theta(x,\tau,\eps)$.

\end{notation}

\begin{defn}\label{t7a}
A detonation profile $P(x)$ is said to be of type D (respectively, type I) if the function $c_0^2\eta=c_0^2-u^2$ satisfies
\begin{align}
d_x(c^2_0\eta)
= d_x(c^2_0- u^2)
<0 \text{ (respectively, }>0) \text{ on }[0,+\infty).
\end{align}
\end{defn}
 For profiles of type I it was  shown in \cite{E2,E3,LWZ1}
that the stability function $V(\tau,\eps)$ generally has zeros in $\Re\tau>0$ (unstable zeros) for all transverse wavenumbers $\eps$ above a certain cutoff.\footnote{Theorem 5.2 of \cite{LWZ1} gives a precise statement.}   In this paper we are concerned only with profiles of type D.   In the case of more general profiles, by considering intervals on which $c_0^2\eta$ is increasing or decreasing, stability and instability results  can be proved by combining the results for I and D type profiles (see, for example, \cite{LWZ1}, Theorem 5.2, part e).

We suppose from now on that $P(x)$ is a type D profile.  In this case
\begin{align}\label{t8}
\mathrm{III}_+=\{\zeta=i|\zeta|:c_0\eta^{1/2}(\infty)\leq |\zeta|\leq c_0\eta^{1/2}(0)\}.
\end{align}

\begin{definition}\label{t9}
We refer to class III as the set of \emph{turning point frequencies}.  For each $\zeta\in\mathrm{III}_+$ there is a unique $x=x(\zeta)\in[0,\infty]$ such that $s(x(\zeta),\zeta))=0$.  We refer to $x(\zeta)$ as the \emph{turning point} associated to $\zeta$.  The map $x(\zeta):\mathrm{III}_+\to [0,\infty]$ is bijective.
 We set $\zeta_0= ic_0\eta^{1/2}(0)$ and $\zeta_\infty=ic_0\eta^{1/2}(\infty)$ and note that $x(\zeta_0)=0$ and $x(\zeta_\infty)=+\infty$.
We refer to $x(\zeta_\infty)$ as the \emph{turning point at infinity}.
\end{definition}

For $\zeta\in \mathrm{III}_+$ consider an interval $[0,K]$ that does not contain the turning point $x(\zeta)$. On any such interval the matrix $T(x,\zeta)$ is invertible, and we can use  WKB methods to construct\footnote{See for example Chapters 5 and 6 of Coddington and Levinson \cite{CL}.}approximate solutions of order $h^{m}$
of the system \eqref{t3} associated to each of the eigenvalues $\mu_i$ of the form
\begin{align}\label{t10}
\theta_i(x,\zeta,h)=e^{\frac{1}{h} h_i(x,\zeta)+k_i(x,\zeta)}
[f_{i0}(x,\zeta)+hf_{i1}(x,\zeta)+\dots+h^{(m+1)}f_{i(m+1)}(x,\zeta)].
\end{align}
Here
\begin{align}\label{t11}
h_i(x,\zeta)=\int^x_0\mu_i(x',\zeta)dx', \;k_i(0,\zeta)=0, \text{ and }f_{i,0}=T_i(x,\zeta).
\end{align}
We do not need explicit formulas  for the other quantities appearing in \eqref{t10}.  For our purposes it is only important to specify the leading term of $\theta_i$ uniquely, and the condition 
\eqref{t11} does this.  More generally, for a given $\zeta\in\{\Re\zeta\geq 0\}$ approximate solutions of this form can be constructed on any compact $x-$interval where $T(x,\zeta)$ is invertible.  Classical sufficient conditions for approximate solutions of this type to be
close
in relative error
to true exact solutions
of \eqref{t3} for $h$ small are given, for example,  in \cite{LWZ1}, Theorem 3.1.

Observing that $|\theta_1(0,\zeta,h)-T_1(0,\zeta)|\leq Ch |\theta_1(0,\zeta,h)|$ we give the following definition:
\begin{definition}[Type $\theta_1$]\label{12a}  Consider  $\theta(x,\zeta,h)$, the decaying solution of \eqref{t3}, and suppose $\zeta$ lies in $\mathcal{P}(h)$, a subset of $\Re\zeta\geq 0$ that may depend on $h$.
We say that $\theta$ is  \emph{of type $\theta_1$ at $x=0$ on $\mathcal{P}(h)$} if there is a nonvanishing scalar factor $s(\zeta,h)$ so that, given any $\kappa>0$, there exists an $h_0>0$ such that
\begin{align}\label{t12}
|s(\zeta,h)\theta(0,\zeta,h)-T_1(0,\zeta)|\leq \kappa |T_1(0,\zeta)|\text{ for }\zeta\in\mathcal{P}(h),\;0<h\leq h_0.
\end{align}
\end{definition}

\begin{rem}\label{t13}
Let $K$ be a subset of $\Re\zeta\geq 0$.  Erpenbeck realized in \cite{E2} that in order to show that the stability function $V(\zeta,h)$ is nonvanishing for $\zeta\in K$ for $h$ sufficiently small, it suffices to show that the decaying solution $\theta$ is of type $\theta_1$ at $x=0$ on $K$.   In \cite{E2,E3,LWZ1}  it was shown for type D profiles that $\theta$ is of type $\theta_1$ at $x=0$ on $K$ when $K$ is either a compact subset of $\mathrm{III_+^o}$ or a compact subset of $\{\Re\zeta\geq 0\}\setminus \mathrm{III}$.\footnote{Theorem 5.2, parts (a),(b) of \cite{LWZ1} gives a treatment of such sets $K$.} These papers did not consider sets $K$ containing either of the endpoints $\zeta_0$, $\zeta_\infty$ of $\mathrm{III}_+$; nor did they provide a uniform treatment of the set $|\zeta|\geq M$.  Moreover,  values of $\zeta$ in $\mathrm{III_+^o}$ were treated by  arguments completely different from those used to treat nearby values in $\Re\zeta>0$.  In this paper the main result will be proved by showing that $\theta$ is of type $\theta_1$ at $x=0$ on the full set $\Re\zeta\geq 0$. The proof here gives a uniform analysis near all points in $\mathrm{III}_+$, including the endpoints, and a uniform analysis for $|\zeta|\geq M$.   In section \ref{model} we give a very simple example illustrating how a second-order equation with a  turning point at $x=+\infty$ transforms to Bessel's equation under a similar transformation.

\end{rem}
\section{Assumptions}\label{assumptions}

\begin{ass}\label{thermo}
The thermodynamic functions appearing in the ZND system \eqref{ZND}, $p$ (pressure), $T$ (temperature), $\Delta F$ (free energy increment), and $r$ (reaction rate) are real analytic functions of their arguments $(v,S,\lambda)$.
\end{ass}

\begin{ass}\label{profile}
The steady strong detonation profile $P(x)=(v,u,0,0,S,\lambda)$ is of type D.  It is a real-analytic function of $x$ in the subsonic reaction zone $[0,\infty)$.   There exist constants $C_i$, $i=1,\dots,4$  such that
\begin{align}
0<C_1\leq \kappa=\frac{u}{c_0}\leq C_2 <1 \text{ and } 0<C_3\leq {u}\leq C_4 \text{ for all }x\in [0,\infty).
\end{align}

\end{ass}

\begin{ass}\label{rate}
The rate function satisfies
\begin{align}\label{a22c}
r|_{\lambda=0}=0,\;\;r_\lambda < 0,\; r_v|_{\lambda=0}=0,\; r_S|_{\lambda=0}=0.
\end{align}
\end{ass}

This assumption is satisfied, for example, by rate functions of the form
\begin{align}\label{rate1}
r=-k\rho\phi(T)\lambda,
\end{align}
where $\rho$ is density and $k>0$ is a reaction rate constant, such as the Arrhenius rate law
\begin{align}\label{rate2}
r=-k\lambda \exp(-E/RT)\text{\;\;($E$ is activation energy)}.
\end{align}

 Analogous to $V(\zeta,h)$, one can define a stability function $L_1(\zeta)$ for the von Neumann shock, considered as a purely gas dynamical shock.  This was first done in \cite{E4}, and Erpenbeck's $L_1(\zeta)$ turned out to be a nonvanishing multiple of the Majda stability determinant for shocks defined in \cite{M} twenty years later.   The functions $V(\zeta,h)$ and $L_1(\zeta)$ are described in section \ref{stabilityfn}.

\begin{ass}\label{vN}
The stability function for the von Neumann shock, $L_1(\zeta)$,  has no zeros in $\Re\zeta\geq 0$.
\end{ass}
This means that the equation of state of the unreacted explosive is such that the von Neumann step-shock would be stable if the reactions behind it were somehow suppressed.
This assumption, which is also made in \cite{E2,E3,LWZ1},  allows us to concentrate on effects that arise solely from the reactions; it always holds, for example, for step-shocks in ideal polytropic gases.

\section{Main Result}

Our main result is the following theorem.

\begin{thm}\label{t14}
Consider a strong detonation profile of type D under Assumptions \ref{thermo}, \ref{profile}, \ref{rate}, and \ref{vN}.  There exists an $h_0>0$ such that for all $\Re\zeta\geq 0$, the stability function $V(\zeta,h)\neq 0$ for $0<h\leq h_0$.
\end{thm}

As explained in Remark \ref{t13}, to prove the theorem it suffices to show that the decaying solution $\theta(x,\zeta,h)$ of \eqref{t3} is of type $\theta_1$ at $x=0$ on the set $\{\Re\zeta\geq0\}$.   This property of $\theta$ is a consequence of: Propositions \ref{e4u}, \ref{f9w}, \ref{g19z}, which treat $\zeta_\infty$; Proposition \ref{q31}, which treats points in $\mathrm{III}_+^o$ and Proposition \ref{q34} which treats $\zeta_0$; Corollary \ref{noturn}, which treats compact subsets of $\{\Re\zeta\geq\setminus \mathrm{III}$ and Proposition \ref{k1}, which treats $|\zeta|\geq M$ for $M$ large.

 The main difficulty in proving existence of a uniform stability cutoff is to obtain explicit representations of $\theta$ at $x=0$ that are uniformly valid for frequencies near turning point frequencies.  For the finite turning point frequencies $\zeta\in \mathrm{III_+^o}$, the three-step strategy is to show first that $\theta$ is of type $\theta_1$ to the right of the turning point $x(\zeta)$,\footnote{This is to be expected, since $\theta$ is the decaying solution and $\Re\mu_1(x,\zeta)\leq 0$} then to perform a matching argument involving Airy functions with arguments depending on $(\zeta,h)$ as a parameter to show that $\theta$ is of type $\theta_1$ just to the left of $x(\zeta)$, say at $x(\zeta)-\delta$,  and finally to match using a basis of exact solutions close to the approximate solutions $\{\theta_i\}_{i=1}^5$ \eqref{t10} to conclude that $\theta$ is of type $\theta_1$ on $[0,x(\zeta)-\delta]$.   This strategy encounters special problems when the endpoint frequencies $\zeta_0$ and $\zeta_\infty$ are considered, and those problems are most serious in the case of $\zeta_\infty$, to which all of Part \ref{infinite} is devoted.  For this frequency it is clear that the first step in the strategy does not even make sense since $x(\zeta_\infty)=+\infty$.  The case $\zeta=\zeta_0$ is distinguished by the fact that this is the only case where the point at which $\theta(x,\zeta,h)$ must be explicitly evaluated in order to compute $V(\zeta,h)$, namely $x=0$, is itself a turning point: $x(\zeta_0)=0$.

 In the remainder of this section we discuss our approach to analyzing the turning point at infinity.
The profile $P(x)$ converges at an exponential rate to its endstate $P(+\infty)$,
\begin{align}\label{t15a}
|P(x)-P(+\infty)|\leq Ce^{-\mu x} \text{ for }\mu>0 \text{ as in }\eqref{o6a},
\end{align}
and we use this property in section \ref{extend} to analytically extend $P(x)$ to a half-plane of the form
\begin{align}
\bW(M_0):=\{x\in\bC:\Re x> M_0\}.
\end{align}
This immediately gives an analytic extension of the governing system \eqref{t3} to $\bW(M_0)$.
For any angle $\theta$ such that $0<\theta<\pi/2$, define the infinite wedge
\begin{align}
\bW(M_0,\theta):=\{x\in\bC:|\arg(x-M_0)|< \theta\}\subset \bW(M_0).
\end{align}
For $\zeta\in\{\Re\zeta\geq 0\}$  near $\zeta_\infty$ and $x\in\bW(M_0,\theta)$ the eigenvalues $\{\mu_j(x,\zeta)\}_{j=1,2}$ are well separated from  $\{\mu_j(x,\zeta)\}_{j=3,4,5}$.
In section \ref{conj} we use this fact to construct a $5\times 5$ conjugator $Y(x,\zeta,h)$ such that the map $\theta=Y(x,\zeta,h)\phi$ \emph{exactly} transforms the system \eqref{t3} to block diagonal form on $\bW(M_0,\theta)$:
\begin{align}\label{t15}
h\phi_x=\begin{pmatrix}A_{11}(x,\zeta,h)&0\\0&A_{22}(x,\zeta,h)\end{pmatrix}\phi,
\end{align}
where the blocks $A_{11}$ and $A_{22}$ are $2\times 2$ and $3\times 3$, respectively.    The conjugator has the form $Y=Y_1Y_2$, where $Y_1$ gives an approximate conjugation to block form \eqref{h5} and $Y_2$ solves away the error in the approximate conjugation.  The matrix $Y_1
=(P_0,Q_0,T_3,T_4,T_5)$ for vectors $P_0$, $Q_0$ defined in \eqref{h3} and satisfying 
\begin{align}\label{t15z}
T_1=P_0+sQ_0,\;T_2=P_0-sQ_0\text{ for }T_1,T_2\text{ as in } \eqref{t6},
\end{align}
and  the entries of the matrix $Y_2$ are constructed by solving certain integral equations on $\bW(M_0,\theta)$ by a contraction argument.   Unlike $T(x,\zeta)$, the matrix $Y_1(x,\zeta)$ is always invertible.  The analytic extension of the system \eqref{t3} to $\bW(M_0)$ gives us a freedom to choose integration paths that plays an important role in this and later contraction arguments.

The block $A_{11}$ has eigenvalues close to the crossing eigenvalues $\mu_1(x,\zeta)$, $\mu_2(x,\zeta)$.  Thus, for $\zeta$ near $\zeta_\infty$ we have reduced the problem of constructing the decaying solution of the governing system \eqref{t3} on
$[M_0,+\infty)$ to constructing the decaying solution of the $2\times 2$ system $d_x\phi_1=A_{11}(x,\zeta,h)\phi_1$.
In section \ref{conj} we rewrite this $2\times 2$ system as  equivalent second-order equation \eqref{a1}
\begin{align}\label{t16}
h^2w_{xx}=(C(x,\zeta)+hr(x,\zeta,h))w,\text{ where }C(+\infty,\zeta_\infty)=0,
\end{align}
and we focus on solving this equation on an infinite strip of the form $T_{M,R}:=\{x\in\bC:\Re x\geq M, |\Im x|\leq R\}$. Note that for $M$ large enough, $T_{M,R}\subset \bW(M_0,\theta)$. In section \ref{general} we show that a transformation of the form $t=t(x,\zeta)=f(\zeta)e^{-\mu x/2}$ for $\mu$ as in \eqref{t15a} and some $f(\zeta)$,  transforms \eqref{t16} into an equation that is a perturbation of Bessel's equation:
\begin{align}\label{t17}
\begin{split}
&\;h^2(t^2W_{tt}+tW_t)=(t^2+\tilde\alpha^2)W+\\
&\quad[(t^2+\alpha^2)t^2h_1(t,\zeta)+t^3h_2(t,\zeta)+ht^2h_3(t,\zeta,h)]W \text{ on }\mathcal W,
\end{split}
\end{align}
where $\mathcal{W}$, the image of the strip $T_{M,R}$ under the map $t=t(x,\zeta)$, is a bounded wedge in $\{\Re t\geq 0\}$ with vertex at $t=0$ (see \eqref{hr4}).

Before returning to \eqref{t17} we illustrate how  Bessel's equation arises from a very simple model equation with a turning point at $+\infty$ by a similar transformation.
This model provides motivation for introducing \eqref{t17}, and already indicates the importance of the parameter $\alpha/h$ which appears below in the definition of the frequency regimes I, II, III.

\subsection{Model problem: solution in terms of  Bessel functions}\label{model}

Consider the equation
$$h^2\frac{d^2w}{dx^2}=(e^{-2x}+\alpha^2)w,$$
which becomes under the change of variable $t=e^{-x}$:
$$h^2(t^2w_{tt}+tw_t)=(t^2+\alpha^2)w.$$
Thus, the turning point at $x=+\infty$ for $\alpha=0$ becomes a turning point at $t=0$.\\
Setting $z=h^{-1}t$ and observing the scale-invariance of the left side, we obtain
$$h^2[z^2\frac{d^2w}{dz^2}+z\frac{dw}{dz}]=(h^2z^2+\alpha^2)w,$$
which suggests that the good parameter is $\beta= h^{-1}\alpha$.
In this case, the equation becomes
$$z^2u_{zz}+zu_z-(z^2+\beta^2)u=0$$
We recognize this as the ``modified  Bessel's equation" \cite{O}.\footnote{The standard Bessel's equation is $z^2U_{zz}+zU_z+(z^2-\nu^2)U=0$.}
One may thus deduce that $u(z)= A J_{\beta}(iz)+BY_{\beta}(iz)$; using $z=h^{-1}e^{-x}$ one obtains
$$w(x)= AJ_{ih^{-1}\alpha}(ih^{-1}e^{-x})+BY_{ih^{-1}\alpha}(ih^{-1}e^{-x}),$$
and classical expansions for Bessel functions can now be used to decide which solutions decay or remain bounded as $x\to\infty$ for different choices of the complex parameter $\alpha$.  A connection between  turning points at infinity and Bessel functions is  discussed in \cite{DL}.
\\

In \eqref{t17} the parameter $\alpha=\alpha(\zeta)\in\bC$ satisfies $\alpha^2=i(\zeta-\zeta_\infty)$ and $\tilde \alpha$ is a nonvanishing multiple of $\alpha$. The strategy now is to construct solutions of \eqref{t16} on the strip $T_{M,R}$ that decay as $\Re x\to +\infty$ by constructing solutions of \eqref{t17} on $\mathcal{W}$ that decay as $t\to 0$.

At first it is not at all clear whether and in what sense the perturbation, given by the second line of \eqref{t17}, is ``small enough" for \eqref{t17} to be helpfully regarded as a perturbation of Bessel's equation.  Bessel's equation is a very singular equation, with a regular singular point at $0$, an irregular singular point at $\infty$, and turning points for certain choices of $(\tilde\alpha,h)$; it is a delicate matter to understand perturbations of such a singular object.  If one ignores the perturbation in \eqref{t17} for a moment, it can be seen that the behavior of solutions to \eqref{t17} depends on the both phase of $\tilde\alpha(\zeta)$ and on the relative magnitude of $\tilde{\alpha}$ and $h$.  Accordingly, in section \ref{perturbed} we
identify three different parameter regimes for $(\tilde\alpha(\zeta),h)$.  With $\tilde\beta=\tilde\alpha/h$ these are:\\

I: $|\tilde\beta|\geq K$, $0\leq \mathrm{arg}\;\beta\leq \frac{\pi}{2}-\delta$ \;\;where $K$ is large and $\delta>0$

II: $|\tilde  \beta|\geq K$, $\frac{\pi}{2}-\delta\leq \mathrm{arg}\;\beta\leq \frac{\pi}{2}$

III: $|\tilde\beta|\leq K$.\\

It turns out that the perturbed Bessel problem \eqref{hr4} can be analyzed in Regimes I, II, III by using suitable transformations of dependent and independent variables to reduce \eqref{hr4} to the normal form

\begin{align}\label{t18}
W_{\xi\xi}=(u^2\xi^m+\psi(\xi))W,
\end{align}
where $m=0$, $1$, or $-1$, respectively, $u$ is a large parameter, and $\psi$ depends on the perturbation in \eqref{t17}.   In Regimes I and II the correct choice of large parameter is $u=\tilde \beta=\tilde\alpha/h$ and a basis of solutions of \eqref{t18} can be written in terms of exponentials and Airy functions, respectively (see Propositions \ref{e3} and \ref{f6}).  In Regime III the large parameter is $u=1/h$ and solutions of \eqref{t18} are expressed in terms of the modified Bessel functions $I_{\tilde\beta}$, $K_{\tilde\beta}$ (Proposition \ref{g12}).

Control of the function $\psi$ in \eqref{t18} is gained by careful estimates of the functions $h_i$, $i=1,2,3$, appearing in the perturbation  (see, for example, Proposition \ref{estimates}). The analysis also makes use of some classical methods \cite{O} for constructing  solutions to \eqref{t18} on \emph{large} subdomains of $\bC$.
To complete the analysis of the turning point at $\infty$, one must unravel the many transformations leading from equation \eqref{t16} to the normal form \eqref{t18}, identify the explicit form of the solution that decays as $x$ (the original $x$) goes to $\infty$, and show that this solution is indeed of type $\theta_1$ at $x=M$, the real point on the left boundary of the infinite strip $T_{M,R}$.  Provided $\zeta$ lies close enough to $\zeta_\infty$, the point $x=M$ will lie to the \emph{left} of any of the corresponding turning points $x(\zeta)$.  From here it is then relatively easy to conclude that $\theta$ is of type $\theta_1$ at $x=0$ for $\zeta$ near $\zeta_\infty$.

\br\label{systemrmk}
 We have limited the  exposition here mostly to the case
of an idealized single-step exothermic reaction, for which $\lambda$ is
scalar.
In the
case of a multi-step reaction, the $\lambda$-equation
becomes a system of ODE, with multiple decaying modes $\sim e^{-\mu_j x}v_j$,  where
$\mu_j$ is in general complex with $\Re \mu_j>0$.  One can then apply the analytic stable manifold theorem of \cite{LWZ2} to obtain an analytic extension of the profile to a wedge $\bW(M_0,\theta)$ for {\it some} $\theta>0$.
Likewise (see
\cite{LWZ2}),
we may conjugate the turning point at infinity to a $2\times 2$ block
analytic on the same wedge.
In some circumstances the presence of multiple reaction steps can prevent us from recasting the $2\times 2$ block in the form \eqref{t17}.  Nonetheless, under the condition that $\mu_1$ is real and 
\be\label{cond}
\hbox{\rm
$\Re \mu_j > 2 \mu_1$ for $j\neq 1$}
\ee
the analysis of this paper gives, independently of the results of \cite{LWZ2} just mentioned,  high-frequency stability of type D detonations also in
this more general case.  This extension is discussed in  detail in section \ref{multistep}.
The following example is
a multi-step
 case where condition \eqref{cond} is satisfied.
\er

\begin{example} \label{shorteg}
In Eqs. (4.3)--(4.5), p. 8, \cite{S}, there is described a
model three-step chain-branching reaction given by
F $\to$ Y; F + Y $\to$ 2Y; Y $\to$  P
for a fuel F, radical species Y, and product P,
corresponding to initiation, chain-branching, and chain-termination reactions,
with rates
\be
r_I= fe^{\theta_I(1/T_I-1/T)},
\quad
r_B=
yfe^{\theta_B(1/T_B-1/T)},
\quad
 r_C=y,
\ee
and reaction dynamics
$df/dx= -r_I-r_B$,
$dy/dx= r_I+r_B-r_C$,
where
$f$ and $y$ are mass fractions of $F$ and $Y$;
$T_I>T|_{x=0}>T_B$;
$T_I>T_B>T(\infty)$;
 and $\theta_I>>\theta_B>> 1$.
This been proposed in \cite{K,SD,SKQ} as a realistic model for
hydrogen-oxygen detonations studied experimentally in \cite{AT,St},
wherein ``a small amount of reactant is converted into
chain-carriers, which may be either free radicals or atoms, by means of
the slow chain-initiation reactions, while the rise in concentration of
chain-radicals is retarded by chain-termination steps which occur either
through absorption at the vessel walls or through three-body collisions
in the interior''  \cite{SD}.
%
Setting $\lambda=(f,y)$, the vector of reactants, and linearizing about the equilibrium $\lambda=(0,0)$,
we obtain
\be
d\lambda/dx= \begin{pmatrix}
-e^{\theta_1(1/T_I-1/T(\infty))} & 0\\
* & -1
\end{pmatrix}\lambda,
\ee
verifying the condition $\mu_1=1 << \mu_2:= e^{\theta_1(1/T_I-1/T(\infty))}$
provided
$\theta_I>>1$ and $T(\infty)>T_I$.
\end{example}

\section{Discussion and open problems}
High-frequency instability was established for type I detonations
in \cite{E2,E3,LWZ1}.  Hence,  Theorem \ref{t14} completes the program of
\cite{E2,E3}
of determining
the high-frequency stability behavior of ZND detonations belonging to the two main classes I and D identified by Erpenbeck.
Having a uniform high frequency cutoff for stability is of  more than abstract interest.
As noted in \cite{LS,KaS}, numerical stability computations are
both computationally intensive and delicate, with many features difficult to
resolve or extrapolate in various asymptotic limits.
The use of rigorous analysis
to truncate the relevant parameter regime to a closed,
bounded
region is thus a critical, but up to now missing, part of any numerical stability investigation.
Bounding the set of possibly unstable perturbation frequencies is a step toward the rigorous validation of the existing numerical results on multidimensional stability of type D ZND detonations.

Results in 1D corresponding to the high-frequency results given here
may be found in \cite{Z1}.
However, we emphasize that the multi-D setting is essentially
different from that of 1D, being
more complicated both physically-
in 1D, high-frequency stability holds automatically for all types of
detonations- and mathematically- in 1D, nontrivial turning points do not enter,
so that the analysis can be carried out using the more familiar tools of repeated diagonalization and (a useful modification of) the gap lemma, under the mild
hypothesis of $C^r$ coefficients for the eigenvalue ODE.
Here, by contrast, our arguments use in important ways our assumption of
analytic coefficients.
This is not mathematical convenience, but reflects the inherent difficulty
of the problem; in a companion paper \cite{LWZ2}, we show by explicit counterexamples that the conclusions made here may fail for coefficients that are $C^r$ or even $C^\infty$.

The stability result of Theorem \ref{t14} and the instability results of \cite{LWZ1}
concern the  multidimensional stability (or instability) of ZND detonations with respect to high frequency perturbations.
A fundamental open problem is
to establish full multi-D
stability of ideal gas ZND detonations with one-step Arrhenius reaction rate
in the small-heat release and high-overdrive limits,
generalizing the 1D results
of \cite{Z1} and giving rigorous validation to the formal observations
of Erpenbeck in \cite{E1,E5}.
This would represent the first complete (i.e., covering all frequencies), rigorous result on
multi-D stability of any detonation wave.
Again, given the delicacy of numerical computations on this subject, any such
analytical signposts are invaluable.

We note that the 1D argument of \cite{Z1}, applied
word for word together with
the result of Majda \cite{M} on multi-D stability
of ideal gas shock fronts, gives already by a simple continuity
argument {\it bounded frequency} multi-D
stability of ideal gas ZND detonations with one-step Arrhenius reaction rate
in the small-heat release and high-overdrive limits.
The methods of this paper provide a starting point for treating the remaining high frequency regime.

Another possibility opened up by our analysis
is the treatment of stability in the multidimensional ZND limit of
reactive Navier--Stokes (rNS) detonations.
Establishing a close link between ZND and rNS stability functions for small viscosity/heat conduction/diffusion would instantly give a large number of spectral
stability and bifurcation results for rNS; in 1D such results were proved
in \cite{Z2}.   Solving this problem would involve giving  a multi-parameter extension of the turning-point analysis
carried out here, with viscosity, heat conduction, and species diffusion as
the additional parameters.
As noted in \cite{CJLW,Z2},
the nonlinear implications of spectral stability (``normal modes'')
analysis are far from clear for ZND,
which includes all of the difficulties of the nonreacting Euler equations
and more.    For rNS on the other hand, which incorporates mechanisms for dissipation,  there is a much better chance of translating results on multi-D spectral stability/instability into corresponding nonlinear stability/instability results.   In 1D a  number of results of this type are given in \cite{TZ}.

As a
direction beyond ZND, we mention
the rigorous treatment for Maxwell's equations
of ``hybrid resonance'' or ``X-mode'' heating of fusion plasma
at the ``cutoff'' frequency where light and plasma frequencies collide.  This frequency corresponds to
a finite but singular turning point \cite{DIW}.
For parallel electric and magnetic fields, there is exact
decoupling into ``ordinary'' (O) modes governed by Airy's equation,
and ``extraordinary'' (X) modes governed by a {\it singular}
cousin which is a perturbed  Bessel equation similar to our equation \eqref{t17}.  
Exact conjugation tools like those we have developed here may be useful for completing this singular ODE analysis. 

Finally, the uniform estimates given here are potentially
useful for general turning-point problems on unbounded spatial intervals,
both for spectral stability analysis as here and
for resolvent estimates toward linearized and
nonlinear stability or instability.

\part{The turning point at infinity}\label{infinite}

\section{Analytic extension of the profile to a half-plane}\label{extend}

The analytic extension of the profile $P(x)$ to a half-plane given in this section turns out to be important for the construction in the next section of the conjugator to block form near $\infty$.   For that an extension merely to a strip like $T_{M,R}$ does  not appear to suffice.

In view of Assumptions \ref{thermo}, \ref{profile}, and \ref{rate},
the nonzero components of the detonation profile $P(x)=(v,u,S,\lambda):=(q,\lambda)$ satisfy a $4\times 4$ system of ODEs on $[0,\infty)$ of the form:
\begin{align}\label{o1}
\begin{pmatrix}F(P)\\\lambda\end{pmatrix}_x=\begin{pmatrix}0\\h(P)\lambda\end{pmatrix},\;\;\;\begin{pmatrix}q(0)\\\lambda(0)\end{pmatrix}=\begin{pmatrix}q_0\\\lambda_0\end{pmatrix},
\end{align}
where $F(P)$ and $h(P)$ are real-analytic.   Here the equation $F(P)_x=0$ expresses conservation of mass, momentum, and energy (see \cite{FD}, p. 98), and can be integrated to give $F(q,\lambda)=F(P(+\infty))$,
an equation that determines $q$ as a function  $q=Q(\lambda)$.\footnote{Here $Q(\lambda)$ is actually a branch of a multivalued function.}  With $\bQ(\lambda)=(Q(\lambda),\lambda)$ the system thus reduces to a scalar problem of the form
\begin{align}
\lambda_x=h(\bQ(\lambda))\lambda,\;\;\;\lambda(0)=\lambda_0>0,
\end{align}
where for some constants constants $c_1$, $c_2$
\begin{align}\label{o2}
-c_1<h(\bQ(\lambda))<-c_2<0\text{ on }[0,\lambda_0].
\end{align}
  The condition \eqref{o2} implies
\begin{align}\label{o3}
|\lambda(x)|\leq Ce^{-c_2x}\text{ on }[0,\infty),
\end{align}
and thus $P(x)=(Q(\lambda(x)),\lambda(x))$ satisfies
\begin{align}\label{o4}
|P(x)-P(\infty)|\leq Ce^{-c_2 x},\text{ where }P(\infty)=(Q(0),0).
\end{align}

The next proposition gives more precise information on the profile for $x$ large.
\begin{prop}\label{o5}
For $M_0$ large enough, the profile $P(x)$ extends analytically to a solution of \eqref{o1} on a half-plane $\bW(M_0):=\{x\in\bC:\Re x>M_0\}$.  The extended profile has a convergent expansion
\begin{align}\label{o6}
P(x)=P_0+P_1e^{-\mu x}+P_2e^{-2\mu x}+\dots \text{ on }\mathbb{W}(M_0),
\end{align}
where $\mu=-h(\bQ(0))>0$ and the $P_j$ are constant vectors.  Thus, $P(x)$ satisfies
\begin{align}\label{o6a}
 |P(x)-P(\infty)|\leq Ce^{-\mu x} \text{ on }\bW(M_0).
 \end{align}

\end{prop}

\begin{proof}
\textbf{1. }In view of the above discussion it suffices to show that $\lambda(x)$ has an expansion like \eqref{o6} (with $\lambda_0=0$) on $\bW(M_0)$ for $M_0$ large.  Let us set $H(\lambda):=h(\bQ(\lambda))$, so $\mu=-H(0)$ and
\begin{align}\label{oo7}
\lambda_x=\lambda H(\lambda)
\end{align}
where $H(\lambda)$ is analytic in a neighborhood of $\lambda=0$.

\textbf{2. }We have $H(\lambda)=H(0)+\lambda K(\lambda)$, so
\begin{align}\label{oo8}
\frac{1}{H(\lambda)}-\frac{1}{H(0)}=\lambda R(\lambda),
\end{align}
for some functions $K(\lambda)$,  $R(\lambda)$ analytic near $\lambda=0$. Multiplying \eqref{oo8} by $\lambda_x$ and  using \eqref{oo7}, we obtain
\begin{align}\label{oo9}
\frac{\lambda_x}{\lambda}+\mu=\mu R(\lambda)\lambda_x.
\end{align}
Let $V(\lambda)=\int^\lambda_0 R(s)ds$ and set $\lambda(x)=e^{-\mu x }T(x)$.  Noting that $\ln T$ is a primitive of the left side of \eqref{oo9}, we obtain by integrating \eqref{oo9} from $M_1$ to $x$ for $M_1$ large:
\begin{align}\label{oo10}
\ln T=\mu V(e^{-\mu x}T(x))+C_0,
\end{align}
where $C_0=C_0(M_1)$ is a known constant.   Defining $\cK(T,b)=\ln T-\mu V(bT)$ near the basepoint $(T_0,b_0)=(e^{C_0},0)$, we can solve $\cK(T,b)=C_0$ by the implicit function theorem to obtain
\begin{align}\label{oo11}
T(b)=e^{C_0}+\sum^\infty_{j=1}a_jb^j \text{ for }|b|<\delta
\end{align}
for some  $\delta>0$.   Thus, $\tilde\lambda(b):=bT(b)$ is analytic for $|b|<\delta$, which implies  $\lambda(x) =\tilde\lambda(e^{-\mu x})$ is analytic for $x$ such that $e^{\mu\Re x}< \delta$, that is, $\Re x>-\frac{\ln \delta}{\mu}:=M_0$.  The expansion \eqref{oo11} implies
\begin{align}\label{oo12}
\lambda(x)=e^{C_0}e^{-\mu x}+\sum^\infty_{j=1}a_je^{-(j+1)\mu x} \text{ on }\bW(M_0),
\end{align}
so $P$ depends analytically on $e^{-\mu x}$.\footnote{A similar analysis of profiles was given in \cite{L}.}

\end{proof}


\section{Conjugation to block form near infinity.}\label{conj}
Here we perform a conjugation, based on Proposition \ref{h10z} below, of  Erpenbeck's $5\times 5$ system
\begin{align}\label{f43}
h\theta_x=(\Phi_0(x,\zeta)+h\Phi_1(x))\theta:=G(x,\zeta,h)\theta.
\end{align}
to  block form
\begin{align}\label{f44}
h\phi_x=\begin{pmatrix}A_{11}(x,\zeta,h)&0\\0&A_{22}(x,\zeta,h)\end{pmatrix}\phi
\end{align}
on an infinite wedge
\begin{align}\label{f44r}
\mathbb{W}(M_0,\theta):=\{x=x_r+ix_i\in\mathbb{C}:|\arg(x-M_0)|<\theta\},\quad M_0>>1,
\end{align}
 contained in the half-plane $\bW(M_0)$
 to which the profile $P(x)=(v,u,,S,\lambda)$ has been analytically extended.  On $\bW(M_0)$ we have
 \begin{align}\label{f44s}
 |P(x)-P(\infty|\leq Ce^{-\mu\Re x}
 \end{align}
 for $\mu>0$ as in Proposition \ref{o5}.


Define the $5\times 5$ matrix
\begin{align}\label{h2}
Y_1=\begin{pmatrix}P_0&Q_0&T_3&T_4&T_5\end{pmatrix},
\end{align}
where
\begin{align}\label{h3}
P_0=\begin{pmatrix}0\\\frac{\zeta}{u}\\-i\\0\\0\end{pmatrix}, Q_0=\begin{pmatrix}\frac{m}{\kappa u}\\0\\0\\-\frac{\kappa}{mu}p_S\\-\frac{\kappa}{mu}p_\lambda\end{pmatrix}, T_3=\begin{pmatrix}-\frac{im}{1-\eta}\\i\\\frac{\zeta}{u}\\0\\0\end{pmatrix}, T_4=\begin{pmatrix}0\\0\\0\\1\\0\end{pmatrix}, T_5=\begin{pmatrix}0\\0\\0\\0\\1\end{pmatrix}.
\end{align}
Thus, we have
\begin{align}\label{h4}
T_1=P_0+sQ_0,\;\;T_2=P_0-sQ_0,\;\;s=\sqrt{\zeta^2+c_0^2\eta(x)}
\end{align}
for $T_1$, $T_2$ as in \eqref{t6}.   Setting $\theta=Y_1 \theta^a$, we have
\begin{align}\label{h5}
h\theta^a_x=\begin{pmatrix}A^0_{11}&0\\0&A^0_{22}\end{pmatrix}\theta^a+h\begin{pmatrix}d_{11}&d_{12}\\d_{21}&d_{22}\end{pmatrix}\theta^a,
\end{align}
where
\begin{align}\label{h6}
A^0_{11}=\begin{pmatrix}-\frac{\kappa^2\zeta}{\eta u}&-\frac{\kappa}{\eta u}\\-\frac{s^2\kappa}{\eta u}&-\frac{\kappa^2\zeta}{\eta u}\end{pmatrix},\;A^0_{22}=\begin{pmatrix}\frac{\zeta}{u}&0&0\\0&\frac{\zeta}{u}&0\\0&0&\frac{\zeta}{u}\end{pmatrix},\;\;\begin{pmatrix}d_{11}&d_{12}\\d_{21}&d_{22}\end{pmatrix}=Y_1^{-1}\Phi_1 Y_1-Y_1^{-1}d_xY_1.
\end{align}
Since the eigenvalues of $A^0_{11}$ are separated from those of $A^0_{22}$, we can apply Proposition \ref{h10z} below to find a second conjugator, bounded and analytic in its arguments,
\begin{align}\label{h6a}
Y_2(x,\zeta,h)=\begin{pmatrix}I&h\alpha_{12}\\h\alpha_{21}&I\end{pmatrix},
\end{align}
such that if we set $\theta^a=Y_2\phi$, we have
\begin{align}\label{h7}
h\phi_x=\begin{pmatrix}A^0_{11}+hd_{11}+h^2\beta_{11}&0\\0&A^0_{22}+hd_{22}+h^2\beta_{22}\end{pmatrix}\phi:=\begin{pmatrix}A_{11}(x,\zeta,h)&0\\0&A_{22}(x,\zeta,h)\end{pmatrix}\phi,
\end{align}
where
\begin{align}\label{h8}
\beta_{11}=d_{12}\alpha_{21},\;\;\beta_{22}=d_{21}\alpha_{12}.
\end{align}
Setting $Y=Y_1Y_2$, we conclude that $\theta$ satisfies \eqref{f43} on the wedge $\mathbb{W}(M_0,\theta)$ if and only if $\phi=\begin{pmatrix}\phi_1\\\phi_2\end{pmatrix}$ defined by
$\theta=Y\phi$ satisfies \eqref{h7}, where the $2\times 2$ block
\begin{align}\label{f45}
\begin{split}
&A_{11}(x,\zeta,h)=\begin{pmatrix}a&b\\c&d\end{pmatrix}=\begin{pmatrix}\ua&\ub\\\uc&\ud\end{pmatrix}+O(h)\text{ with } \\
&\ua=-\frac{\kappa^2\zeta}{\eta u},\;\ub=-\frac{\kappa}{\eta u},\;\uc=-\frac{s^2\kappa}{\eta u},\; \ud=-\frac{\kappa^2\zeta}{\eta u}.
\end{split}
\end{align}

\begin{lem}\label{h10}
The functions $d_{11}(x)$ and  $d_{12}(x)$
 decay exponentially to $0$ on $\bW(M_0,\theta)$ as $\Re x\to \infty$ at the same rate as $\lambda(x)$.
\end{lem}

\begin{prop}\label{h10z}
a)  Let $\theta$ be any angle such that $0<\theta<\frac{\pi}{2}$. There exist positive constants $M_0$, $h_0$, and a neighborhood $\omega\ni\zeta_\infty$ such that for $\zeta\in\omega$ and $0<h<h_0$, the conjugator $Y_2(x,\zeta,h)$ as in \eqref{h6a} can be constructed on $\mathbb{W}(M_0,\theta)$ with $\alpha_{12}(x,\zeta,h)$ and $\alpha_{21}(x,\zeta,h)$ bounded and analytic in their arguments.    Moreover, there exists a well-defined endstate $\alpha_{21}(\infty)$ and we have estimates
\begin{align}\label{f45c}
|\partial_x^k\left(\alpha_{21}(x,\zeta,h)-\alpha_{21}(\infty,\zeta,h)\right)|\leq C_kh^{-k}e^{-\mu\Re x}
\end{align}
for $\mu>0$ as in Proposition \ref{o5}.

\end{prop}

\begin{rem}\label{f45z}

1.) The proofs are given in section \ref{conjugation}.  The differential equation satisfied by $\alpha_{21}$ is
\begin{align}\label{f45b}
hd_x\alpha_{21}=A^0_{22}\alpha_{21}-\alpha_{21}A_{11}^0+d_{21}+h(d_{22}\alpha_{21}-\alpha_{21}d_{11})-h^2\alpha_{21}d_{12}\alpha_{21}.
\end{align}
The argument uses the fact that the eigenvalues of $A^0_{11}(\infty,\zeta)$ and $A^0_{22}(\infty,\zeta)$ are separated and  close to the imaginary axis for  $\zeta$ near $\zeta_\infty$.
Using the fact that $d_{21}$ converges exponentially to its endstate $d_{21}(\infty)$,
  \begin{align}
  |d_{21}(x,\zeta)-d_{21}(\infty,\zeta)|\leq Ce^{-\mu \Re x}, \;
  \end{align}
and that $d_{12}$ satisfies a similar estimate but with $d_{12}(\infty)=0$,
  one can show that $\alpha_{21}$ has a well-defined endstate and that the estimates \eqref{f45c} hold.  These estimates are the key to the treatment of the $h_3$ term in the perturbation of Bessel's equation given by \eqref{t17}.

2.) The above lemma and proposition imply that the function $\beta_{11}(x,\zeta,h)$ decays exponentially to $0$ on $\bW(M_0,\theta)$ as $\Re x\to \infty$ at the same rate as $\lambda(x)$.
 Thus, the same holds for the $O(h)$ terms in \eqref{f45}.

\end{rem}

A valuable tool for understanding the behavior of solutions of \eqref{f43} as $x\to \infty$ is the conjugator $M(x,\zeta,h)$ described in the following Proposition.

\begin{prop}[``\cite{MZ} conjugator", \cite{MZ}, Lemma 2.6]\label{MZ}
Consider any $N\times N$ system $d_x\theta=A(x,\zeta,h)\theta$ on $[0,\infty)$, for  $(\zeta,h)$ near a fixed basepoint $(\uzeta,\uh)\in\{\Re\zeta\geq 0\}\times (0,1]$, where $A$ is analytic in its arguments.   Assume there is a corresponding limiting system $d_x\theta_{\infty }=A(\infty,\zeta,h)\theta_\infty$ and that
  \begin{align}
  |A(x,\zeta,h)-A(\infty,\zeta,h)|\leq Ce^{-\beta x}\text{ for some }\beta>0.
  \end{align}
 Then  there exists a neighborhood $\mathcal O\ni(\uzeta,\uh)$ and an $N\times N$ matrix $M(x,\zeta,h)$, analytic in its arguments $x\in [0,\infty]$, $(\zeta,h)\in \mathcal{O}$, and uniformly bounded together with its inverse, such that $\theta(x,\zeta,h)$ is a solution of $d_x\theta=A(x,\zeta,h)\theta$ on $[0,\infty)$ if and only if $\theta_\infty$ defined by
\begin{align}\label{f45a}
\theta=M(x,\zeta,h)\theta_\infty
\end{align}
is a solution of the limiting system.  Moreover, for $k\in\{0,1,2,\dots\}$ $M$ satisfies $|d_x^k(M(x,\zeta,h)-I)|\leq C_ke^{-\delta x}$ for any $0<\delta<\beta$, uniformly for $(\zeta,h)\in\mathcal O$.

\end{prop}

From the matrix formulas given in section \ref{coefficients}, it is not hard to see that the eigenvalues of $G(x,\zeta,h)$ \eqref{f43} are
\begin{align}\label{h9}
\begin{split}
&\mu_j^*:=\mu_j(x,\zeta)+O(he^{-\mu x}), \;j=1,2,3,4 \\
&\mu_5^*=\mu_5(x,\zeta)-h\frac{r_\lambda}{u}+O(he^{-\mu x})
\end{split}
\end{align}
where  $r_\lambda<0$ and $\mu$ is as in Proposition \ref{o5}.\footnote{The details are given in section 2 of \cite{LWZ1}.}
Thus, the eigenvalues of the limiting system $G(\infty,\zeta,h)$ are  $\mu_j(\infty,\zeta)$, $j=1,2,3,4$ and $\mu_5(\infty,\zeta)-h \frac{r_\lambda}{u}(\infty)$.
For $\Re\zeta>0$ only $\mu_1(\infty,\zeta)$ has negative real part, so use of the conjugator $M(x,\zeta,h)$ shows that for $\Re\zeta>0$, the system \eqref{f43} has a one-dimensional space $\mathcal D(\zeta,h)$ of decaying solutions on $[M,\infty)$.

Lemma \ref{h10} implies that $A_{11}(\infty,\zeta,h)=A^0_{11}(\infty,\zeta)$, so
the eigenvalues of $A_{11}(\infty,\zeta,h)$ are $\mu_j(\infty,\zeta)$, $j=1,2$.
   Use of the \cite{MZ} conjugator again implies that for $\Re\zeta>0$ the equation
   \begin{align}\label{h11}
   h\phi_{1x}=A_{11}(x,\zeta,h)\phi_1
    \end{align}
    has a one dimensional space of decaying solutions $\mathcal D_1(\zeta,h)$.  Thus, we must have
\begin{align}\label{f46b}
\mathcal D(\zeta,h)=\{Y(x,\zeta,h)\begin{pmatrix}\phi_1\\0\end{pmatrix},\;\phi_1\in\mathcal D_1(\zeta,h)\}.
\end{align}

Next we reduce \eqref{h11} to an equivalent scalar second-order equation.  Letting $\varphi_0$ for the moment denote any primitive of $\frac{a+d}{2}$, and making the transformation
\begin{align}\label{h12}
\tilde\phi_1=e^{-\frac{\varphi_0}{h}}\phi_1,
\end{align}
we obtain the system
\begin{align}\label{h13}
hd_x\tilde\phi_1=\begin{pmatrix}-\alpha&b\\c&\alpha\end{pmatrix}\tilde\phi_1, \; \alpha=\frac{d-a}{2}.
\end{align}
Setting $\tilde\phi_1=\begin{pmatrix}\tilde u\\\tilde v\end{pmatrix}$, we rewrite the first row of \eqref{h13} as
\begin{align}\label{h14}
\tilde v=b^{-1}(h\tilde u'+\alpha\tilde u),\;\;\text{ where } '=d/dx,
\end{align}
and hence the second row of \eqref{h13} becomes
\begin{align}\label{h15}
h^2(b^{-1}\tilde u')'+h(\frac{\alpha}{b}\tilde u)'=(c+\frac{\alpha^2}{b})\tilde u+\frac{\alpha}{b}h\tilde u'.
\end{align}
Defining $w=b^{-1/2}\tilde u$ and\footnote{It does  not matter which branch of the square root we use here, as long as we always use the same one.} using the identity
\begin{align}\label{h16}
(b^{-1}\tilde u')'=(\frac{1}{2}b^{-\frac{3}{2}}b')'w+b^{-\frac{1}{2}}w'',
\end{align}
we obtain in place of equation \eqref{h15}
\begin{align}\label{h17}
h^2w''=(bc+\alpha^2)w-hb(\frac{\alpha}{b})'w-h^2b^{\frac{1}{2}}(\frac{1}{2}b^{-\frac{3}{2}}b')'w.
\end{align}
With $A^0_{11}=\begin{pmatrix}\ua&\ub\\\uc&\ud\end{pmatrix}$ as in \eqref{h6}, we can rewrite \eqref{h17} as
\begin{align}\label{h18}
h^2w''=(C(x,\zeta)+hr(x,\zeta,h))w,
\end{align}
where
\begin{align}\label{h19}
\begin{split}
&C(x,\zeta)=\ub\uc+(\frac{\ud-\ua}{2})^2=(\zeta^2+c^2_0\eta(x))\ub^2(x) \text{ and }\\
&hr(x,\zeta,h)=(bc+\alpha^2)-\left(\ub\uc+(\frac{\ud-\ua}{2})^2\right)-hb(\frac{\alpha}{b})'-h^2b^{\frac{1}{2}}(\frac{1}{2}b^{-\frac{3}{2}}b')'
\end{split}
\end{align}

\begin{prop}\label{h21}
The function $r$ satisfies $r(\infty,\zeta,h)=0$.

\end{prop}

\begin{proof}
The functions appearing in the expression for $r$ can all be expressed in terms of the components of $A^0_{11}$, $d_{11}$ and $\beta_{11}$, so the Proposition follows directly from Lemma \ref{h10}.
\end{proof}

Making the following choice of $\varphi_0$ such that $\varphi_0'= \frac{a+d}{2}$,
\begin{align}\label{f46a}
\varphi_0(x,\zeta,h)=\frac{a+d}{2}(\infty,\zeta,h)\cdot x+\int^x_\infty \left[\frac{a+d}{2} (s,\zeta,h)-\frac{a+d}{2}(\infty,\zeta,h)\right]ds,
\end{align}
we have shown that solutions of $hd_x\phi_1=A_{11}(x,\zeta,h)\phi_1$ are given by
\begin{align}\label{f46}
\phi_1=e^\frac{\varphi_0}{h}\begin{pmatrix}b^{1/2}&0\\\alpha b^{-1/2}-h(b^{-1/2})_x&b^{-1/2}\end{pmatrix}\begin{pmatrix}w\\hw_x\end{pmatrix}:=K(x,\zeta,h)\begin{pmatrix}w\\hw_x\end{pmatrix},
\end{align}
where $(w,hw_x)$ satisfies
\begin{align}\label{h20}
h\begin{pmatrix}w\\hw_x\end{pmatrix}_x=\begin{pmatrix}0&1\\C(x,\zeta)+hr(x,\zeta,h)&0\end{pmatrix}\begin{pmatrix}w\\hw_x\end{pmatrix}.
\end{align}

\begin{rem}\label{f47a}
Since $\frac{a+d}{2}(\infty,\zeta,h)=\frac{\mu_1+\mu_2}{2}(\infty,\zeta)=-\zeta \frac{\kappa^2}{\eta u}(\infty)$, we see that $\varphi_0$ is the sum of a term with real part $\leq 0$ and, by Lemma \ref{h10}, a term that decays  exponentially to $0$ as $x\to\infty$.
\end{rem}

Using Remark \ref{f47a}, for $\Re\zeta>0$ we  obtain
\begin{align}\label{f47}
\mathcal{D}(\zeta,h)=\mathrm{span}\;\; Y(x,\zeta,h)\begin{pmatrix}K(x,\zeta,h)\begin{pmatrix}w\\hw_x\end{pmatrix}\\0\end{pmatrix},
\end{align}
where $(w,hw_x)$ gives a decaying solution of \eqref{h20}.  Erpenbeck's stability function $V(\zeta,h)$  is expressed in terms of $\theta(0,\zeta,h)$, where $\theta(x,\zeta,h)\in \mathcal{D}(\zeta,h)$.

Our next main task is to construct explicit asymptotic formulas for the exact solutions of \eqref{h20} that decay to zero as $x\to \infty$.

\section{Reduction to a perturbation of Bessel's equation in the general case}\label{general}

In the general case we must consider equation \eqref{h18}
\begin{align}\label{a1}
h^2w''=\left(C(x,\zeta)+hr(x,\zeta,h)\right)w,
\end{align}
where $C(x,\zeta)=(\zeta^2+c_0^2\eta(x))\ub^2(x)$.  Here the $x$-dependence enters only through the detonation profile $P(x)=(v,u,S,\lambda)$, and $\ub=-\kappa/\eta u$.

Since $\zeta^2+c_0^2\eta=(\zeta-ic_0\sqrt{\eta})(\zeta+ic_0\sqrt{\eta})$, we can
write
\begin{align}
C(x,\zeta)=(e(x)+\alpha^2)D(x,\zeta)
\end{align}
where
\begin{align}\label{a1a}
e(x)=c_0\sqrt{\eta}(x)-c_0\sqrt{\eta}(\infty),\;\alpha^2=i(\zeta-\zeta_\infty), \;D(x,\zeta)=(-i\zeta+c_0\sqrt{\eta}(x))\ub^2(x).
\end{align}
We have
 \begin{align}\label{a16b}
D(\infty,\zeta)>0 \text{ for }\zeta=i|\zeta|,
\end{align}
so the function $D(x,\zeta)$ is strictly bounded away from $0$ for $x$ large and $\zeta$ near $\zeta_\infty$ (the latter frequency being the endpoint of $\mathrm{III}_+$ corresponding to the turning point at infinity).
We take $\zeta$ in a small neighborhood of $\zeta_\infty$ in $\Re\zeta\geq 0$.
Note that $e(x)$ is strictly positive on $[0,\infty)$ and decreases to $0$ at an exponential rate (case D).

For the wedge $\bW(M_0,\theta)$  in Proposition \ref{h10z} we now choose $M>M_0$ and $R>0$ such that
the strip
\begin{align}\label{ac1}
T_{M,R}:=\{x=x_r+ix_i:x_r\geq M, |x_i|\leq R\}\subset \bW(M_0,\theta),
\end{align}
 and consider \eqref{a1} on $T_{M,R}$. Proposition \ref{o5} implies that $e(x)$ has an expansion similar to $\lambda(x)$ \eqref{oo12} on $T_{M,R}$, so in particular
 \begin{align}\label{ac2}
 e(x)=ae^{-\mu x}+m(x)e^{-\mu x}\text{ where }a>0,\;|m(x)|\leq Ce^{-\mu\Re x},
 \end{align}
and $m(x)$ is real-valued on $[M,+\infty)$.

Setting $d(x,\zeta)=D(x,\zeta)-D(\infty,\zeta)$,
the problem \eqref{a1} can now be written
\begin{align}\label{a16}
\begin{split}
&h^2w''=\left(ae^{-\mu x}+\alpha^2+m(x)e^{-\mu x}\right)D(x,\zeta)w+hr(x,\zeta,h)w=\\
&(ae^{-\mu x}+\alpha^2)D(\infty,\zeta)w+[(ae^{-\mu x}+\alpha^2)d(x,\zeta)+m(x)e^{-\mu x}D(x,\zeta)]w+hr(x,\zeta,h)w.
\end{split}
\end{align}

Recalling that the $x-$dependence in $d(x,\zeta)$ enters only through the profile, using \eqref{o6}, and setting $t=\frac{2}{\mu}\sqrt{aD(\infty,\zeta)}e^{-\mu x/2}$, we can rewrite $d(x,\zeta)$ using the $t$ variable as
\begin{align}
\begin{split}
&d(x(t),\zeta)=t^2h_1(t,\zeta)\\
\end{split}
\end{align}
where $h_1$ is analytic in $t$ and $O(1)$ on the bounded wedge $\mathcal{W}$ with vertex at $t=0$, which is the image of the strip $T_{M,R}$ under the change of variable $t=t(x)$.   Similarly,
\begin{align}\label{hr1a}
m(x)e^{-\mu x}D(x,\zeta)=t^3h_2(t,\zeta),
\end{align}
where $h_2(t,\zeta)=O(|t|)$  and analytic on $\mathcal W$.     The function $r(x,\zeta,h)$ is more complicated and must be handled carefully.  Equation \eqref{h19} and \eqref{h7} show that the $x-$ dependence in $r$ enters through the components of
\begin{align}\label{ab1}
\begin{split}
&(a)\;P(x),P'(x),P''(x)\\
&(b)\;h\beta_{11},h^2\beta_{11}',h^3\beta_{11}'', \text{ where }\beta_{11}=d_{12}\alpha_{21}.
\end{split}
\end{align}
Thus, Proposition \ref{h21} and Lemma \ref{h10} allow us to write
\begin{align}\label{ab2}
r(x(t),\zeta,h)=t^2h_3(t,\zeta,h)
\end{align}
where $h_3$ is analytic in $t$ and $O(1)$ on $\mathcal W$.

\begin{defn}\label{defW}
It will be convenient to shrink $\mathcal W$ slightly and take it to be symmetric about the horizontal axis.   So we redefine
$\mathcal{W}=\{t\in\bC:|\arg t|<\eps_1,\;0<|t|<\eps_2\}$, where $\eps_1$, $\eps_2$ are small positive constants easily expressed in terms of $M$ and $R$.  Since \eqref{a16b} holds, after shrinking $\omega\ni\zeta_\infty$ if necessary, we can be sure that  $\mathcal W$ still contains the image of $[M,\infty)$ under the map $t=t(x,\zeta)$ for all $\zeta\in\omega$.
\end{defn}




Setting $\tilde\alpha:=\frac{2}{\mu}\alpha\sqrt{D(\infty,\zeta)}$ and
\begin{align}\label{hr3}
z=\frac{t}{h},\;\beta=\frac{\alpha}{h},\;\tilde\beta=\frac{\tilde\alpha}{h},
\end{align}
we obtain the following two equivalent forms for the equation \eqref{a16}
\begin{align}\label{hr4}
\begin{split}
&(a)\;h^2(t^2W_{tt}+tW_t)=(t^2+\tilde\alpha^2)W+\\
&\quad[(t^2+\alpha^2)t^2h_1(t,\zeta)+t^3h_2(t,\zeta)+ht^2h_3(t,\zeta,h)]W \text{ on }\mathcal W\\
&(b)\;(z^2 W_{zz}+zW_z)=(z^2+\tilde\beta^2)W+\\
&\quad [h^2(z^2+\beta^2) z^2h_1(hz,\zeta)+hz^3h_2(hz,\zeta)+hz^2h_3(hz,\zeta,h)]W,
\end{split}
\end{align}
where $z$ lies in the wedge $\mathcal{Z}_h=\mathcal{W}/h$.\footnote{In \eqref{hr4} the functions $h_i$ in \eqref{hr4} are nonvanishing constant multiples of their former selves.}

\begin{rem}\label{hr4z}
For later reference we note that if we set $t=t(x)$ and $W(t(x))=w(x)$, then the right side of \eqref{hr4}(a) is
$\frac{4}{\mu^2}(C(x,\zeta)+hr(x,\zeta,h))w$.

\end{rem}

\begin{prop}\label{estimates}
The functions $h_1(t,\zeta)$, $h_2(t,\zeta)$, $h_3(t,\zeta,h)$ satisfy the following estimates.  Here $h_1(0,\zeta)$, for example, denotes the limiting value of $h_1$ as $t\to 0$, and $k\in\{0,1,2,...\}$.
\begin{align}\label{est}
\begin{split}
&(a)\;|\partial_t^k(h_1(t,\zeta)-h_1(0,\zeta))|\leq C_k|t|^{2-k}, \\
&(b)\;|\partial_t^k h_2(t,\zeta)|\leq C_k |t|^{1-k}\\
&(c)\;|\partial_t^k(h_3(t,\zeta,h)-h_3(0,\zeta,h))|\leq C_k|t|^{2-k}h^{1-k}.
\end{split}
\end{align}
The estimates are uniform for $h\in (0,1]$, $\zeta$ in a small neighborhood of $\zeta_\infty$ in $\Re\zeta\geq 0$, and $t\in\mathcal{W}$ (a wedge in $\Re t\geq 0$ with vertex at $t=0$ such that  $|t|<<1$ for $t\in\mathcal{W}$).
\end{prop}

\begin{proof}

\textbf{1. }Recall $t=\frac{2}{\mu}\sqrt{aD(\infty,\zeta)}e^{-\mu x/2}$. Given a function $f(x)$ analytic on the strip $T_{M,R}$, let
\begin{align}\label{a19}
f^*(t):=f\left(-\frac{2}{\mu}\log\frac{\mu t}{2\sqrt{aD(\infty,\zeta)}}\right)
\end{align}
be the corresponding function on $\mathcal{W}$.  Suppose
\begin{align}\label{a18}
|\partial_x^k f(x)|\leq C_k e^{- |x|\mu},\; x\in T_{M,R}.
\end{align}
Then since $|x|\sim -\frac{2}{\mu} \ln |t|$, \eqref{a19} implies
\begin{align}\label{a20}
|\partial_t f^*(t)|\leq C_1 e^{(\frac{2}{\mu}\ln|t|)\mu}\frac{1}{|t|}\leq C_1|t|,
\end{align}
and by induction
\begin{align}\label{a21}
|\partial_t^k f^*(t)| \leq C_k|t|^{2-k}.
\end{align}

\textbf{2. }To estimate $h_1$ we use the fact that the $x-$dependence in $D(x,\zeta)$ enters only through the profile $P(x)$ and write $D(x,\zeta)=E(P(x),\zeta)$.
We have
\begin{align}
\begin{split}
&D(x,\zeta)-D(\infty,\zeta)=(P(x)-P(\infty))\int^1_0\partial_P E(P(\infty)+s(P(x)-P(\infty)),\zeta)ds\\
&\qquad\quad=t^2(x)h_1(t(x),\zeta),
\end{split}
\end{align}
 where $h_1(t,\zeta)-h_1(0,\zeta)=f^*(t)$ for a function $f(x)$ which satisfies  \eqref{a18} in view of the expansion \eqref{o6}, so we obtain \eqref{est}(a).


\textbf{3. }To estimate $h_2$ we use \eqref{hr1a} and the fact that $m(x)$ satisfies the estimates \eqref{a18}.   Thus,
$f(x):=m(x)D(x,\zeta)$ satisfies the same estimates.  Since $h_2(t)= \frac{f^*(t)}{t}\frac{\mu^2}{4aD(\infty,\zeta)}$, the estimate \eqref{est}(b) follows from \eqref{a21}.

\textbf{4. Estimate of $h_3$. }Terms not involving $\beta_{11}$ in the expression for $h_3$ can be estimated like $h_1$; the worst terms involve $\beta_{11}$ and its derivatives, where $x-$dependence enters not only through the profile $P(x)$ but also through $\alpha_{21}(x,\zeta,h)$.  For example, consider the term
\begin{align}\label{a21z}
h\beta_{11}\ub=hd_{12}\alpha_{21}\ub,
\end{align}
which appears in the expression for $\frac{1}{h} (bc-\ub\uc)$. \footnote{Here and in the rest of step \textbf{5} we use $\beta_{11}$ to denote the appropriate entry of the $2\times 2$ matrix $\beta_{11}$; recall \eqref{h7}. A similar remark applies to $d_{12}$ and $\alpha_{21}$.} Lemma \ref{h10} and the argument giving \eqref{ac2} show that $d_{12}=e^{\mu x}(a+m(x))$ for some (new) $a$ and $m(x)$ satisfying $|m(x)|\leq C^{-\mu\Re x}$.   Thus
\begin{align}\label{a21y}
h\beta_{11}\ub=e^{\mu x}(a+m(x))h\alpha_{21}\ub:=t^2(x)H_3(t(x),\zeta,h).
\end{align}
Since $\alpha_{21}$ satisfies the estimates \eqref{f45c}, using the explicit form of $x=x(t)$ we obtain
\begin{align}\label{a21x}
|\partial_t^k (H_3(t,\zeta,h)-H_3(0,\zeta,h))|\leq C |t|^{2-k}h^{1-k}
\end{align}
by arguing as for $h_1$.
The functions $hd_x\alpha_{21}$ and $h^2d^2_x\alpha_{21}$ also satisfy the estimates \eqref{f45c} (now with endstates $0$), so the terms involving derivatives of $\beta_{11}$ (recall \eqref{ab1}) can be estimated in the same way.

\end{proof}

\section{Differential equations with singularities, turning points, and a large parameter}\label{parameter}
Consider equations of the form
\begin{align}\label{b1}
w_{\sigma\sigma}=(u^2 f(\sigma)+g(\sigma))w
\end{align}
on a domain $D\subset \mathbb{C}$, where $u$ is a large real or complex parameter, and the functions $f$ and $g$ are analytic except at boundary points or isolated interior points of $D$.  Under certain conditions on $f$ and $g$  the problem \eqref{b1} can be usefully transformed by a change of dependent and independent variables into one of the normal forms:
\begin{align}\label{b2}
W_{\xi\xi}=(u^2\xi^m+\psi(\xi))W,
\end{align}
where $m=0$, $1$, or $-1$, and $\psi$ can be expressed explicitly in terms of $f$ and $g$.
The transformation of independent variable in these cases is, respectively,
\begin{align}\label{b3}
\begin{split}
&(a)\; \xi=\int^\sigma_{\sigma_0} f^{1/2}(r)dr\\
&(b)\; \frac{2}{3}\xi^{3/2}=\int_{\sigma_0}^\sigma f^{1/2}(r)dr\\
&(c)\;2\xi^{1/2}=\int_{\sigma_0}^\sigma f^{1/2}(r)dr,
\end{split}
\end{align}
where $\sigma_0$ is  a zero or pole of $f$ in (b), (c), respectively (\cite{O}, Chapter 10).  With $\dot  \sigma=\frac{d\sigma}{d\xi}$ one defines
$W=\dot \sigma^{-1/2}w$, and then finds
\begin{align}\label{b3a}
\psi(\xi)=\dot \sigma^2 g(\sigma)+\dot \sigma^{1/2} \frac{d^2}{d\xi^2}(\dot \sigma^{-1/2}).
\end{align}

The problem \eqref{b2} is easily solved in the elementary case when $\psi$ is identically zero, so it is natural to use variation of constants and integral equations to solve the general case.   This program is carried out in detail in Chapters 10, 11, and 12 of \cite{O}, which treat the respective cases $m=0$, $1$, $-1$.  The elementary solutions  are exponentials $e^{\pm u\xi}$ in the case $m=0$, and Airy functions in the case $m=1$.

 In the case $m=-1$, it is shown in \cite{O}, Chapter 12, that if $g$ has a simple or double pole at $\sigma=\sigma_0$, and we define $\nu$ by
\begin{align}
\frac{\nu^2-1}{4}=(\sigma-\sigma_0)^2g(\sigma)|_{\sigma=\sigma_0},
\end{align}
then under the above transformations \eqref{b1} takes the form
\begin{align}\label{b4}
W_{\xi\xi}=\left(\frac{u^2}{\xi}+\psi(\xi)\right)W=\left(\frac{u^2}{\xi}+\frac{\nu^2-1}{4\xi^2}+\frac{\phi(\xi)}{\xi}\right)W,
\end{align}
where $\phi$ is analytic at $\xi=0$.  We now take the equation obtained by neglecting $\frac{\phi(\xi)}{\xi}$ in \eqref{b4} as the ``elementary equation"; its solutions are the modified Bessel functions $\xi^{1/2}I_\nu(2u\xi^{1/2})$ and $\xi^{1/2}K_\nu(2u\xi^{1/2})$.

\section{Three frequency regimes.}\label{frequencies}
It is not yet clear whether and in what sense the equations \eqref{hr4} are useful perturbations of Bessel's equation.  The answer turns out to depend on both the phase of $\alpha=\sqrt{i(\zeta-\zeta_\infty)}$  and the relative magnitude of $\alpha$ and $h$.\footnote{The square root is positive when its argument is positive.}Here $\zeta$ lies in a small neighborhood of $\zeta_\infty$ in $\Re\zeta\geq 0$.   Let $\beta=\alpha/h$.
For $K>0$ sufficiently large and  a fixed small $\delta>0$  we distinguish the following three regimes, which exhaust the relevant $\alpha$:\\

I: $|\beta|\geq K$, $0\leq \mathrm{arg}\;\beta\leq \frac{\pi}{2}-\delta$, where $\delta>\eps_1$ (for $\eps_1$ as in Definition \ref{defW}).

II: $| \beta|\geq K$, $\frac{\pi}{2}-\delta\leq \mathrm{arg}\;\beta\leq \frac{\pi}{2}$

III: $|\beta|\leq K$.\\

It will turn out that the perturbed Bessel problem \eqref{hr4} can be analyzed in Regimes $\mathrm{I}$, $\mathrm{II}$, $\mathrm{III}$ by reducing to the normal form \eqref{b2} where $m$ is respectively $0$, $1$, $-1$.

\subsection{Regime I}\label{one}
To get an idea of how this works in a simple setting closely related to our perturbed problem, consider the modified Bessel's equation
\begin{align}\label{c1}
w_{zz}+\frac{1}{z}w_z=(1+\frac{\beta^2}{z^2})w,
\end{align}
where first we take $\beta=\alpha/h$ as in case I, and $z=t/h$ for $t\in \mathcal W$ (Definition \ref{defW}).
So $z\in\mathcal{Z}_h=\mathcal{W}/h$.

Setting $w=\hat w z^{-\frac{1}{2}}$ to eliminate the first derivative, we obtain
\begin{align}\label{c3}
\hat w_{zz}=(1+\frac{\beta^2}{z^2})\hat w-\frac{1}{4z^2}\hat w\text{ on }\mathcal{Z}_h.
\end{align}
Next set $v(\sigma)=\hat w(\beta \sigma)$ for $\sigma$ in the rotated large wedge $\mathcal{W}/h\beta=\mathcal{W}/\alpha:=\cZ_\alpha$ to obtain
\begin{align}\label{c4}
v_{\sigma\sigma}=\beta^2(1+\frac{1}{\sigma^2})v-\frac{1}{4\sigma^2}v \text{ on }\cZ_\alpha,
\end{align}
which is a problem of the form \eqref{b1} with
\begin{align}
u=\beta,\; f(\sigma)=1+\frac{1}{\sigma^2},\; g(\sigma)=-\frac{1}{4\sigma^2}.
\end{align}
Note that the condition  $\delta> \eps_1$ in  the definition of Regime I implies that the points $\sigma=\pm i$, where $f(\sigma)=0$, do not lie in $Z_\alpha$ for $\beta$ in Regime I.   As shown in \cite{O}, Chapter 10, the transformations
\begin{align}\label{c5}
\xi=\int f^{1/2}(\sigma)d\sigma,\;\; v=f^{-1/4}(\sigma)W
\end{align}
change \eqref{c4} into a problem satisfied by $W(\xi)$ of the normal form \eqref{b2} with $m=0$ and
\begin{align}\label{c6}
\psi(\xi)=\frac{g(\sigma)}{f(\sigma)}-\frac{1}{f^{3/4}(\sigma)}\frac{d^2}{d\sigma^2}\left(\frac{1}{f^{1/4}(\sigma)}\right).
\end{align}

\begin{rem}\label{c7}
The problem \eqref{c4} has a regular singularity at $0$ and an irregular singularity ``at $\infty$", but no turning points (which are points where $f(\sigma)=0$) in $\mathcal{Z}_\alpha$.   The wedge $\cZ_\alpha$ is bounded for fixed $\alpha$, but since
$\alpha$ can be $O(h)$ for some $\beta$ in regime I, and since we are interested in uniform estimates as $h\to 0$, the domain $\mathcal{Z}_\alpha$ can become unbounded as $h\to 0$.  Thus, we effectively have a singularity at infinity.

In our application to Erpenbeck's stability problem we study \eqref{a1} in the original $x$ variables on the infinite strip $T_{M,R}$ \eqref{ac1}, and we need to know how the solution that decays at $x=\infty$, which corresponds to $\sigma=0$, behaves at $x=M$, which corresponds to $\sigma=e^{-CM}/h$, for some $C>0$.  Obtaining an explicit formula for the exact decaying solution at $x=M$ is the main step before extending the solution to $x=0$, where the stability function can be assessed.
A great advantage of the method presented in Chapter 10 of \cite{O} is that it produces an asymptotic representation of the exact solution at once on the entire (large) domain $\cZ_\alpha$.  If instead one tried, say, to use the theory of regular singularities to construct the decaying solution near $\sigma=0$, and another method to construct a solution near infinity (i.e., for $\sigma=O(1/h)$), there would remain the difficult problem of matching up the two expansions somewhere in between.

\end{rem}

\subsection{Regimes II}\label{two}
Next consider \eqref{c1} again, but with large $\beta$ with argument close to $\pi/2$. So $\beta=i\gamma$ where $\mathrm{arg}\;\gamma$ is close to $0$.  Rewriting \eqref{c3} with $\beta^2=-\gamma^2$ and setting $v(\sigma)=\hat w(\gamma \sigma)$ now for $\sigma\in \mathcal{W}/(-i\alpha)=:\cZ_{-i\alpha}$, we obtain instead of \eqref{c4}
\begin{align}\label{c8}
v_{\sigma\sigma}=\gamma^2(1-\frac{1}{\sigma^2})v-\frac{1}{4\sigma^2}v \text{ on }\cZ_{-i\alpha},
\end{align}
This problem has singularities at zero and infinity as before, but now there is a turning point, namely $\sigma=1$, in the interior of $\cZ_{-i\alpha}$, since $\mathrm{arg}(-i\alpha)$ is near $0$.  Instead of having turning points converging to $z=0$, or running off to infinity in the original $x$ variables, the device of considering $v(\sigma)=\hat w(\gamma \sigma)$ yields a
problem with a single fixed turning point and large parameter $u=\gamma$.  Using the new variables $\xi$ and $W$ defined by
\begin{align}\label{c9}
\left(\frac{d\xi}{d\sigma}\right)^2=\frac{\sigma^2-1}{\xi \sigma^2}=\frac{f}{\xi}:=\hat f,\;\;v=\left(\frac{d\xi}{d\sigma}\right)^{-1/2}W
\end{align}
transforms \eqref{c8} into the normal form \eqref{b2} with $m=1$ and
\begin{align}\label{c10}
\psi(\xi)=\frac{g(\sigma)}{\hat f(\sigma)}-\frac{1}{\hat f^{3/4}(\sigma)}\frac{d^2}{d\sigma^2}\left(\frac{1}{\hat f^{1/4}(\sigma)}\right), \text{ where }g(\sigma)=-\frac{1}{4\sigma^2}.
\end{align}
The method of Chapter 11 of \cite{O} yields an expansion of the exact solution of \eqref{c8} valid on $\cZ_{-i\alpha}$, a large wedge (growing as $h\to 0$) with vertex at $\sigma=0$ and turning point $\sigma=1$ in its interior.

\subsection{Regime III}\label{three}
Now we consider equation \eqref{c1} for $|\beta|\leq K$, which includes the case corresponding to the ``turning point at infinity", $\beta=0$.  Here  it is best to work in the $t$ variables on the bounded ($h-$independent) wedge $\mathcal{W}$. The equation in these variables is
\begin{align}
w_{tt}+\frac{1}{t}w_t=\frac{1}{h^2}\left(1+\frac{\alpha^2}{t^2}\right)w.
\end{align}
Setting $w=v t^{-1/2}$ we obtain
\begin{align}\label{c11}
v_{tt}=\frac{1}{h^2}\left(1+\frac{\alpha^2}{t^2}\right)v-\frac{1}{4t^2}v= \frac{1}{h^2}(1)+\left(\frac{\beta^2-\frac{1}{4}}{t^2}\right) v \text{ on }\mathcal{W},
\end{align}
where we have used the fact that $\beta^2$ is now comparable in size to $1/4$ to group these terms together.

 One might  regard \eqref{c11} as a problem that is already in the normal form \eqref{b2} with $m=0$, $u=\frac{1}{h}$ and $\psi=
\frac{\beta^2-\frac{1}{4}}{t^2}$, and try to apply the method of Chapter 10 of \cite{O}.  This does not work;
the integrals of $|\psi|$ on paths starting at $0$ need to be finite
in order to solve the integral equation arising in the error estimates, but such integrals blow up.   One can see from
\eqref{c6} that $f$ must have a singularity at $t=0$ to balance that of $g$ at $t=0$
in order for such integrals to be finite.   Instead, one might regard \eqref{c11} as a problem of
the form \eqref{b1} with $u=\frac{1}{h}$, $f(t)=1+\frac{\alpha^2}{t^2}$, and $g(t)=-\frac{1}{4t^2}$,
and use transformations like   \eqref{c5} to reduce \eqref{c11} to the normal form \eqref{b2} with $m=0$ and a different $\psi$.
This also fails; the function $\psi$ now depends on $\alpha=O(h)$, and though the integrals described above are now finite for fixed $h$, they blow up as $h\to 0$.

Instead we proceed as follows. Setting $t=2s^{1/2}$ and $\hat v(s)=s^{1/4}v(2s^{1/2})$, we obtain
\begin{align}\label{c12}
\hat v_{ss}=\left(\frac{1}{h^2s}+\frac{\beta^2-1}{4s^2}\right)\hat v \text{ on }\frac{\mathcal{W}^2}{4}.
\end{align}
This problem already has the form of the ``elementary equation" corresponding to the case $m=-1$ of section \ref{parameter}, and has solutions that can be expressed in terms of modified Bessel functions.
There is no singularity at $\infty$ now, since  $\mathcal{W}^2$, which is bounded, is independent of $h$; turning points are absent as well from \eqref{c12}. \footnote{Observe, though, that turning points are present in the first equation of \eqref{c11} for $\arg\alpha=\frac{\pi}{2}$.  They converge to zero as $h\to 0$ since $\alpha=O(h)$.}

When we consider the perturbed Bessel equation in this frequency regime, we will obtain an equation like \eqref{c12} with the same $g(s)$, but with
$\frac{1}{h^2s}$ replaced by $\frac{1}{h^2} f$, where $f=\frac{1}{s}+f_p(s)$, with $f_p$  the perturbation given in \eqref{d6}.

\section{Transformation of the perturbed Bessel's equation.}\label{perturbed}
Next we describe how transformations like those described above can be applied to the perturbed equations given in \eqref{hr4}.   Recalling the definition  of $\tilde\beta=\tilde\beta(\zeta,h)$ from \eqref{hr3} and the formula for $D(\infty,\zeta)$ \eqref{a1a}, we see that corresponding to $\zeta$ in the each of the frequency regimes of section \ref{frequencies}, we have, respectively:\\

I: $|\tilde\beta|\geq K_1$, $-\delta_1 \leq \mathrm{arg}\;\tilde\beta\leq \frac{\pi}{2}-\delta_1$, where $\delta_1>\eps_1$
for $\eps_1$ as in Definition \ref{defW}.\footnote{Since \eqref{a16b} holds, we see that after shrinking $\omega\ni\zeta_\infty$ if necessary,
we can describe Regime I here using a $\delta_1>\eps_1$
provided $\delta>\eps_1$ for $\delta$ as in the definition of Regime I in section \ref{frequencies}.}

II: $| \tilde \beta|\geq K_1$, $\frac{\pi}{2}-\delta_2\leq \mathrm{arg}\;\tilde \beta\leq \frac{\pi}{2}$

III: $|\tilde \beta|\leq K_2$.\\

Here $0<K_1<K_2$, $\delta_j>0$ is small, and $K_1$ can be made arbitrarily large  by taking $K$ in section \ref{frequencies} large.



\subsection{Regime I}\label{onep}

Applying the same transformations as in section \eqref{one} to the perturbed equation \eqref{hr4}(b), but with $\tilde\beta$ now playing the role of $\beta$, we obtain instead of \eqref{c4} the equation
\begin{align}\label{d1}
v_{\sigma\sigma}=(\tilde\beta^2 f(\sigma)+g(\sigma))v \text{ on }\mathcal{W}/\tilde\alpha :=\mathcal{Z}_{\tilde\alpha},
\end{align}
where $g(\sigma)=-\frac{1}{4\sigma^2}$ as before and
\begin{align}\label{d2}
\begin{split}
&f(\sigma)=f_0(\sigma)+f_p(\sigma),\text{ where }\\
&f_0(\sigma)=1+\frac{1}{\sigma^2}\text{ and }f_p(\sigma)=(\tilde\alpha^2 \sigma^2+\alpha^2)h_1(\tilde\alpha \sigma,\zeta)+\tilde\alpha \sigma h_2(\tilde\alpha \sigma,\zeta)+hh_3(\tilde\alpha \sigma,\zeta,h).
\end{split}
\end{align}

\begin{rem}\label{d3a}
It will be important later to take the perturbation $f_p(\sigma)$ sufficiently small on the relevant domain (e.g., $\mathcal{Z}_{\tilde\alpha}$ or $\mathcal{Z}_{-i\tilde\alpha}$).   This will be the case provided $\tilde\alpha \sigma$, $\alpha$, and $h$ are small.  Since $\tilde\alpha \sigma\in\mathcal{W}$, $\tilde \alpha \sigma$ is small when $\mathcal{W}$, a wedge with vertex at $0$, is small; more precisely, $|\tilde\alpha\sigma|\leq \eps_2$ for $\eps_2$ as in  Definition \ref{defW}. The parameter $\eps_2$ is small when the strip $T_{M,R}$ is a sufficiently small neighborhood of $\infty$, that is, when $M$ is sufficiently large. One makes $\alpha$ small by restricting $\zeta$ to a sufficiently small neighborhood $\omega\ni\zeta_\infty$.
\end{rem}

\subsection{Regime II}\label{twop}
Writing $\tilde\beta^2=-\tilde\gamma^2$, where $\arg \tilde \gamma$ is close to zero,  and applying the same transformations as in section \ref{two} to the perturbed equation \eqref{hr4}(b), but with $\tilde\gamma$ now playing the role of $\gamma$, we obtain instead of \eqref{c8} the equation
\begin{align}\label{d3}
v_{\sigma\sigma}=(\tilde\gamma^2 f(\sigma)+g(\sigma))v \text{ on }\mathcal{W}/(-i\tilde\alpha) :=\mathcal{Z}_{-i\tilde\alpha},
\end{align}
where $g(\sigma)=-\frac{1}{4\sigma^2}$ as before and (since $h\tilde\gamma=-i\tilde\alpha$)
\begin{align}\label{d4}
\begin{split}
&f(\sigma)=f_0(\sigma)+f_p(\sigma),\text{ where }\\
&f_0(\sigma)=1-\frac{1}{\sigma^2}\text{ and }f_p(\sigma)=(\alpha^2-\tilde\alpha^2 \sigma^2)h_1(-i\tilde\alpha \sigma,\zeta)-i\tilde\alpha \sigma h_2(-i\tilde\alpha \sigma,\zeta)+hh_3(-i\tilde\alpha \sigma,\zeta,h).
\end{split}
\end{align}
Clearly the function $f_p$ will be small under the same conditions as described in Remark \ref{d3a}.


\subsection{Regime III}
Starting now with the perturbed equation in the $t$ form \eqref{hr4}(a) and making the same transformations as in section \ref{three}, in place of \eqref{c12} we obtain
\begin{align}\label{d5}
\hat v_{ss}=\left(\frac{1}{h^2}f(s)+g(s)\right)\hat v \text{ on } \mathcal{W}^2/4,
\end{align}
where $g(s)=\frac{\tilde\beta^2-1}{4s^2}$ and
\begin{align}\label{d6}
\begin{split}
&f(s)=f_0(s)+f_p(s)\text{ with }\\
&f_0(s)=\frac{1}{s}\text{ and }f_p(s)=\frac{1}{s}\left[(4s+\alpha^2)h_1(2s^{1/2},\zeta)+2s^{1/2}h_2(2s^{1/2},\zeta)+hh_3(2s^{1/2},\zeta,h)\right].
\end{split}
\end{align}

Note that each of the perturbations $f_p$  as in \eqref{d2}, \eqref{d4}, or \eqref{d6} is determined once $\mathcal{W}$, $\zeta$, and $h$ are specified, where $\mathcal{W}$ is the wedge defined by the choice of constants $\eps_1$, $\eps_2$ as in Definition \ref{defW}.   Thus, we can write $f_p(\cdot)=f_p(\cdot,\eps_1,\eps_2,\zeta,h)$.  Remark \eqref{d3a} and the estimates of $h_j$, $j=1,2,3$ of Proposition \ref{estimates} directly imply:

\begin{prop}\label{d7}
Let $f_p$ be a perturbation as above.  Set $N(p)=\eps_2+|\zeta-\zeta_\infty|+h$, where $\eps_2$ appears in the definition of $\mathcal{W}$.   Then given $\delta_1>0$ there exists $\delta_2>0$ such that
\begin{align}
\begin{split}
&N(p)<\delta_2\Rightarrow |f_p|_{L^\infty(\mathcal{Z}_{\tilde \alpha})}<\delta_1\text{ for Regime I},\\
&N(p)<\delta_2\Rightarrow |f_p|_{L^\infty(\mathcal{Z}_{-i\tilde \alpha})}<\delta_1\text{ for Regime II},\\
&N(p)<\delta_2\Rightarrow |sf_p|_{L^\infty(\mathcal{W}^2/4)}<\delta_1\text{ for Regime III},
\end{split}
\end{align}
\end{prop}
\begin{rem}
In later arguments arguments we will reduce the perturbation $f_p$ by reducing $N_p$.  We do not include $\eps_1$ (as in Definition \ref{defW}) as one of the summands in the definition of $N_p$, since in that case shrinking $N_p$ could produce a wedge $\mathcal{W}$ that no longer contains the image of $[M,\infty)$ under the map $t=t(x,\zeta)$.   Although there are restrictions on the size of $\eps_1$ (for example, in the definition of Regime I), Proposition \ref{d7} implies that for  almost all purposes it suffices to shrink $N_p$ as defined above.

\end{rem}

\section{Leading term expansions}
In this section we describe the form of the leading term expansions for exact solutions to the perturbed problem in each of the frequency regimes described in section \ref{perturbed}.  In each case the $\xi$ variable is defined  as in \eqref{b3} for appropriately chosen lower limits $\sigma_0$, where $f$ is given by \eqref{d2}, \eqref{d4}, or \eqref{d6}.  In each case one achieves the normal form \eqref{b2} by defining $W(\xi)$ and $\psi(\xi)$ as described in section \ref{parameter}.

The main things to check are that ``progressive paths" of integration can be chosen as required by the contraction arguments
and that the integrals involving $\psi(\xi)$ that arise in the error estimate converge at the singularity at zero and (in the cases of regimes I and II) at the singularity at infinity.
These points are explained in the following discussion and in the proofs.

\subsection{Regime I}\label{onelead}

 Recall the definitions of the variables
\begin{align}\label{e4s}
t=\frac{2}{\mu}\sqrt{aD(\infty,\zeta)}e^{-\mu x/2},\;z=\frac{t}{h},\;\tilde\alpha=\frac{2}{\mu}\alpha\sqrt{D(\infty,\zeta)},\;\tilde\beta=\frac{\tilde\alpha}{h},\;\sigma=\frac{z}{\tilde\beta}\in\mathcal{W}/\tilde\alpha=\mathcal Z_{\tilde\alpha}.
\end{align}

When $f_p(\sigma)$ in \eqref{d2} is neglected, the integral defining $\xi(\sigma)$ is easily evaluated by trigonometric substitution and yields
\begin{align}\label{e1}
\xi=\int^\sigma_{\sigma_0}\frac{(1+s^2)^{1/2}}{s}ds=(1+\sigma^2)^{1/2}+\log\frac{\sigma}{1+(1+\sigma^2)^{1/2}},
\end{align}
where the branches of square root and logarithm are the principal ones.   Here $\sigma_0$ is the point on the positive real axis at which the right side of \eqref{e1} vanishes.  Using, for example, the fact that for small $|\sigma|$, $\xi=\log(\frac{1}{2}\sigma)+1+o(1)$, while for large $|\sigma|$, $\xi=\sigma+o(1)$, it is not hard to draw a picture of the domain in the $\xi-$plane, $\mathcal{Z}_\xi$, that corresponds to the large wedge $\mathcal{Z}_{\tilde\alpha}$ under the map \eqref{e1} (Figure 7.2, \cite{O}, Chapter 10).  Progressive paths are of two types: those along which $\Re(\tilde\beta\xi)$ is nondecreasing, and those along which $\Re(\tilde\beta\xi)$ is nonincreasing. The choice of such paths is obvious in the $\xi-$plane and using
\begin{align}
\psi(\xi)=\frac{1}{4}\sigma^2(4-\sigma^2)/(1+\sigma^2)^3,
\end{align}
one sees that in this case ($f_p$ neglected) the integrals
\begin{align}\label{e2}
\int^\xi_{\xi(\sigma_j)}|\psi(s)|\;d|s|, \text{ where }\sigma_1=0,\;\sigma_2=\infty
\end{align}
are finite along such paths.\footnote{Here ``$\infty$" should be interpreted as a point at the far right extreme of the large wedge $\cZ_{\tilde\alpha}$.}

When $f_p$ is included in the definition of $f$, the domain in the $\xi-$plane corresponding to $\mathcal{Z}_{\tilde\alpha}$ under the map
$\xi=\int_{\sigma_0}^\sigma f^{1/2}(s)ds$ is a small perturbation of $\cZ_\xi$ when $f_p$ is small, and progressive paths are again easy to choose on an appropriate subdomain.   Using the estimates of Proposition \ref{estimates} for the functions $h_j$, $j=1,2,3$ appearing in the definition of $f_p$, one can show that integrals \eqref{e2} involving the redefined $\psi(\xi)$ are again finite.

To get started it is necessary to show that the map $\sigma\to \xi$ defines a good, global change of variables:

\begin{prop}\label{e4z}
For $f_0$ and $f_p$ as in \eqref{d2} let
\begin{align}\label{e4zz}
\xi_f(\sigma):=\int_{\sigma_0}^\sigma (f_0+f_p)^{1/2}(s)ds,
\end{align}
where $\sigma_0$ is (as before) the point on the positive real axis where the right side of \eqref{e1} vanishes.
For perturbations $f_p$ with $N_p$ sufficiently small (recall Proposition \ref{d7}) the function $\xi=\xi_{f}(\sigma)$, is a globally one-to-one analytic map of $\cZ_{\tilde\alpha}$ onto the open set which is its range.

\end{prop}

In the next Proposition $\mathcal{Z}_{\tilde\alpha,s}$ is an open subdomain of $\mathcal{Z}_{\tilde\alpha}$ containing the image of the segment of the $x-$axis, $[M,\infty)$, under the map $x\to \sigma$.   The domain $\mathcal{Z}_{\tilde\alpha,s}$ is defined as $\xi_f^{-1}(\Delta_\xi)$, where $\Delta_\xi$ (described precisely in the proof given in section \ref{pfinfiniteII}) is an open domain in $\xi-$space on which progressive paths  can be chosen.

\begin{prop}\label{e3}
Suppose $\tilde\beta$  as  defined in \eqref{hr3} lies in Regime I.
For  $f_p$ as in \eqref{d2} taken sufficiently small (by the choices explained in Remark \eqref{d3a}), the perturbed Bessel problem \eqref{d1}  has exact solutions
\begin{align}
\begin{split}\label{e3q}
&v_j(\sigma)=\xi^{-1/2}_\sigma(\sigma)\left(e^{\tilde\beta\xi(\sigma)}+\eta_1(\tilde\beta,\xi(\sigma))\right),\;j=1,2\\
&v_2(\sigma)=\xi^{-1/2}_\sigma(\sigma) \left(e^{-\tilde\beta\xi(\sigma)}+\eta_2(\tilde\beta,\xi(\sigma))\right)
\end{split}
\end{align}
on $\mathcal{Z}_{\tilde\alpha,s}$, where the error terms satisfy
\begin{align}\label{e3z}
|\eta_j(\tilde\beta,\xi)|,|\partial_\xi\eta_j(\tilde\beta,\xi)|\leq \frac{C}{|\tilde\beta|} |e^{(-1)^{j-1}\tilde\beta\xi}|.
\end{align}

\end{prop}

\begin{rem}
The proof, given in section \ref{pfinfiniteII}, is based on Theorem 3.1 of Chapter 10 of \cite{O} and the estimates of Proposition \ref{estimates}.   The result of \cite{O} constructs solutions
\begin{align}
W_j(\xi)=e^{(-1)^{j-1}\tilde\beta\xi}+\eta_j(\tilde\beta,\xi),\;j=1,2
\end{align}
of
\begin{align}\label{e3ww}
W_{\xi\xi}=(\tilde\beta^2+\psi(\xi))W,\;\text{ where }
\psi(\xi)=\frac{g(\sigma)}{f(\sigma)}-\frac{1}{f^{3/4}(\sigma)}\frac{d^2}{d\sigma^2}\left(\frac{1}{f^{1/4}(\sigma)}\right),
\end{align}
by solving the  integral equation satisfied by $\eta_j$ (obtained by variation of parameters).  When $j=1$ the equation is
 \begin{align}
 \eta_1(\tilde\beta,\xi)=\int^\xi_{\alpha_j}K(\xi,v)\left[\frac{\psi(v)e^{\tilde\beta v}}{\tilde\beta}+\frac{\psi(v)\eta_1(\tilde\beta,v)}{\tilde \beta}\right]dv,\text{ where }K(\xi,v)=\frac{1}{2}\left(e^{\tilde\beta(\xi-v)}-e^{\tilde\beta(v-\xi)}\right),
 \end{align}
 and the integral is taken on progressive paths.
 This equation is solved on the  domain $\Delta_\xi=\xi(\mathcal{Z}_{\tilde\alpha,s})$ by  iteration.
  The progressive path property of $\Delta_\xi$ gives a useful pointwise estimate of $|K(\xi,v)|$.   Together with a uniform bound on the integrals \eqref{e2},  this yields convergence of the sequence of iterates.\footnote{It is not necessary to take $|\tilde\beta|$ large to obtain convergence.  See Theorem 10.1 of Chapter 6, \cite{O}.}  Convergence of the sequence of differentiated iterates follows from a similar pointwise estimate of $|\partial_\xi K(\xi,v)|$   and the property  $K(\xi,\xi)=0$.

\end{rem}

Consider the first-order system corresponding to \eqref{a1}:
\begin{align}\label{e4yy}
h\begin{pmatrix}w\\hw_x\end{pmatrix}_x=\begin{pmatrix}0&1\\C(x,\zeta)+hr(x,\zeta,h)&0\end{pmatrix}\begin{pmatrix}w\\hw_x\end{pmatrix}.
\end{align}
The next Proposition describes the solutions of \eqref{e4yy} that are bounded for $\Re\zeta=0$ and decaying for $\Re\zeta>0$ as $x\to \infty$ in $[M,\infty)$, when $\tilde\beta(\zeta,h)$ lies in Regime I.
Recall that Proposition \ref{e3} is valid for a small enough choice of wedge $\mathcal W$, neighborhood $\omega \ni \zeta_\infty$, and $h_0$ such that $0<h\leq h_0$.

\begin{prop}[Choice of decaying solution]\label{e4y}
After shrinking $\omega$ and reducing $h_0$ if necessary, we have, for $\zeta\in \omega$ and $0<h<h_0$ such that $\tilde\beta(\zeta,h)$ lies in Regime I,  that the bounded (resp. decaying) solution of \eqref{e4yy} on $[M,\infty)$ for $\Re\zeta=0$ (resp. $\Re\zeta>0$) is given by
\begin{align}\label{e4x}
w(x)=z(x)^{-1/2}v_1(\sigma(x))
\end{align}
Here $v_1$ is defined in  \eqref{e3q} and the maps  $x\to z(x)$ and $x\to \sigma(x)$ are defined  by \eqref{e4s}.   The corresponding decaying solution of Erpenbeck's $5\times 5$ system \eqref{f43} is thus given by the formula in \eqref{f47} for this choice of $w(x)$.

\end{prop}

Having identified the exact decaying solution of Erpenbeck's system for $\tilde\beta$ in Regime I, the next step is to show that this solution is of type $\theta_1$ at $x=M$ (recall \eqref{t10} and Definition \ref{12a}).  Since $x=M$ is to the left of any turning point, it will then be rather easy to conclude that the exact decaying solution is of type $\theta_1$ at $x=0$.  This will allow us to deduce that the stability function $V(\zeta,h)$ is nonvanishing for $\tilde\beta$ in Regime I.


\begin{prop}[Decaying solution is of type $\theta_1$ at $x=M$]\label{e4w}
Let $\theta(x,\zeta,h)$ be the exact decaying solution of \eqref{f43} identified in Proposition \ref{e4y}, and let $\theta_1(x,\zeta,h)$ be the approximate solution given by \eqref{t10}.      There exist $h_0>0$, a  neighborhood $\omega\ni\zeta_\infty$, and a nonvanishing scalar function $H(x,\zeta,h)$ defined for $x$ near $M$ such that
\begin{align}\label{e4v}
|H(x,\zeta,h)\theta(x,\zeta,h)-\theta_1(x,\zeta,h)|\leq\frac{C}{|\tilde\beta(\zeta,h)|}|\theta_1(x,\zeta,h)|\text{ for }x\text{ near }M,
\end{align}
where $C$ is independent of $x$ near $M$ and of $\zeta\in\omega$, $0<h\leq h_0$ such that $\tilde\beta$ lies in Regime I.\footnote{Recall that for $\tilde\beta(\zeta,h)$ in Regime I we have $h<\frac{1}{|\tilde\beta|}\leq\frac{1}{K_1}<<1$.}
\end{prop}

The final step in the treatment of Regime I is to show that the exact decaying solution of \eqref{f43} is of type $\theta_1$ at $x=0$.  In fact, we show next that a multiple of $\theta$ is of type $\theta_1$ on all of $[0,M]$.
The explicit formula \eqref{u1} for the stability function $V(\zeta,h)$ in terms of $\theta(0,\zeta,h)$ shows then that $V(\zeta,h)$ is nonvanishing for $(\zeta,h)$ in Regime I when Assumption \ref{vN} holds.

\begin{prop}[Decaying solution is of type $\theta_1$ at $x=0$]\label{e4u}
Let $\theta(x,\zeta,h)$ be the exact decaying solution of \eqref{f43} identified in Proposition \ref{e4y},  and let $H(x,\zeta,h)$ be the function referred to in Proposition \ref{e4w}.    There exist $h_0>0$ and a  neighborhood $\omega\ni\zeta_\infty$ such that
\begin{align}\label{e4t}
|H(M,\zeta,h)\theta(x,\zeta,h)-\theta_1(x,\zeta,h)|\leq\frac{C}{|\tilde\beta(\zeta,h)|}|\theta_1(x,\zeta,h)|\text{ on }[0,M],
\end{align}
where $C$ is independent of $x\in [0,M]$ and of $\zeta\in\omega$, $0<h\leq h_0$ such that $\tilde\beta$ lies in Regime I.
\end{prop}

\begin{proof}
Let $\overline\theta_j(x,\zeta,h)$, $j=1,\dots,5$ be exact solutions of \eqref{f43} on $[0,M]$  (constructed as in \cite{LWZ1}, Theorem 3.1, for example) such that
\begin{align}\label{f9v}
|\overline\theta_j-\theta_j|\leq Ch|\theta_j|\text{ on }[0,M],
\end{align}
where the approximate solutions $\theta_j$ are defined on $[0,M]$, an interval with no turning points.
The formulas \eqref{t4} for the $\mu_j(x,\zeta)$ and the fact that $0<\kappa(x)<1$ imply that for $x\in [0,M]$ and $\zeta\in\omega_\infty$, we have
\begin{align}\label{f9u}
\Re \mu_1<0,\;\Re\mu_2>0,\;\Re\mu_j\geq 0, \;j=3,4,5,
\end{align}
and thus for $h_j(x,\zeta)=\int^x_0\mu_j(x',\zeta)dx'$ we have
\begin{align}\label{f9t}
-\Re h_1(M,\zeta):=a>0,\;\Re h_2(M,\zeta):=b>0,\;\Re h_j(M,\zeta)=c\geq 0, \;j=3,4,5.
\end{align}
Expanding the exact solution solution $H(M,\zeta,h)\theta(x,\zeta,M)$ in the given basis,
\begin{align}
H(M,\zeta,h)\theta(x,\zeta,h)=c_1(\zeta,h)\overline\theta_1(x,\zeta,h)+\dots+c_5(\zeta,h)\overline\theta_5 \text{ on }[0,M],
\end{align}
evaluating at $x=M$, and then using \eqref{f9za}, \eqref{f9v}, and Cramer's rule, we obtain
\begin{align}\label{f9s}
c_1(\zeta,h)=1+O(1/|\tilde\beta(\zeta,h)|),\;c_2=O(e^{-\frac{a+b}{h}}/|\tilde\beta(\zeta,h)|),\;c_j=O(e^{-\frac{a+c}{h}}/|\tilde\beta(\zeta,h)|),\;j=3,4,5.
\end{align}
In view of \eqref{f9v} the behavior of the $\overline\theta_j$ on $[0,M]$ is given by the explicit formulas for the $\theta_j$.
Thus, it follows from these formulas and \eqref{f9s} that the $\overline\theta_1$ term dominates on $[0,M]$, or more precisely,  that \eqref{e4t} holds.

\end{proof}

\subsection{Regime II}\label{twolead}

 Recall the definitions of the variables
\begin{align}\label{f9zz}
t=\frac{2}{\mu}\sqrt{aD(\infty,\zeta)}e^{-\mu x/2},\;z=\frac{t}{h},\;\tilde\gamma=-i\tilde\beta,\;\sigma=\frac{z}{\tilde\gamma}\in\mathcal Z_{-i\tilde\alpha}.
\end{align}
With $[0,1]$ denoting the line segment joining $0$ to $1$, we define $\mathcal{Z}_{cut}(1)$ to be the simply connected subregion of $\Re \sigma>0$ given by
\begin{align}\label{f1}
\mathcal{Z}_{cut}(1)= \cZ_{-i\tilde\alpha}\setminus [0,1].
\end{align}
Set $\Xi:=\frac{2}{3}\xi^{3/2}$.  When $f_p(\sigma)$ in \eqref{d4} is neglected, we define $\Xi(\sigma)$ by
\begin{align}\label{f3}
\Xi(\sigma)=\int^\sigma_1 \frac{(s^2-1)^{1/2}}{s}ds=(\sigma^2-1)^{1/2}+i\log\left(\frac{1+i(\sigma^2-1)^{1/2}}{\sigma}\right),
\end{align}
where the branch of $(\sigma^2-1)^{1/2}$ on $\cZ_{cut}(1)$ is positive for $\sigma>1$, and the branch of $\log\left(\frac{1+i(\sigma^2-1)^{1/2}}{\sigma}\right)$  takes negative (resp. positive) values in the limit as $\sigma\to a_\pm$, $0<a<1$, from the upper (resp. lower) half-plane.\footnote{This branch takes values $ib$, $0<b<\frac{\pi}{2}$, for $\sigma>1$.} The definition of $\Xi(\sigma)$ is extended by continuity to $a_\pm$ for $0<a<1$.  Observe that
\begin{align}\label{f4}
\Xi(a_\pm)=\mp ib,\text{ for some }b=b(a)>0, \text{ for }0<a<1,
\end{align}
and that $\mp ib(a)$ are mapped to the same point on the negative $\xi$ axis under the map $\Xi\to\xi$. Morever, the map
$\sigma\to\xi(\sigma)$ turns out to extend analytically to a map of a full neighborhood of $\sigma=1$ onto a full neighborhood of $\xi=0$.

In this case ($f_p$ neglected) one can draw the domains in the $\Xi$ and $\xi$ planes corresponding to $\cZ_{cut}(1)$ under \eqref{f3}.\footnote{Drawings for the rather different case where $f(\sigma)=\frac{1}{\sigma^2}-1$ are given in Figures 10.1-10.4 of \cite{O}, Chapter 10.} Progressive paths in the $\xi$ plane now have the property that on corresponding paths in the $\Xi$ plane $\Re(\tilde\gamma \Xi)$ is monotonic, and it is easy to identify such paths in the $\Xi$-plane.

When $f_p$ is included in the definition of $f$, $\Xi$ and $\xi$ are now defined by \eqref{b3}(b) with $f$ as in \eqref{d4}, where $\sigma_0$ is the point (close to $1$) where $f(\sigma_0)=0$.  In the definition of $\Xi$ we now take $\sigma\in \cZ_{cut}(\sigma_0)$ defined by
\begin{align}\label{f4a}
\mathcal{Z}_{cut}(\sigma_0)= \cZ_{-i\tilde\alpha}\setminus [0,\sigma_0],
\end{align}
where $[0,\sigma_0]$ is the line segment joining $0$ to $\sigma_0$.  We show below that, unlike $\Xi(\sigma)$, the function $\xi(\sigma)$ extends across the cut to be analytic on all of $\cZ_{-i\tilde \alpha}$.

The domain in $\xi$ space corresponding to $\mathcal{Z}_{-i\tilde \alpha}$ under the map
$\sigma\to \xi$ is a small perturbation, when $f_p$ is small, of the  domain in the case $f_p=0$, and progressive paths are again not hard to choose.   Using the estimates of Proposition \ref{estimates} for the functions $h_j$, $j=1,2,3$ appearing in the definition of $f_p$, one can show that integrals arising in the error analysis,
\begin{align}\label{f5}
\int_{\alpha_j}^\xi |\psi(s)s^{-1/2}|\;d|s| \text{ (on progressive paths)}
\end{align}
are finite.
Here $\psi(\xi)$ is given by \eqref{c10} with $f$ as in \eqref{d4} and $g=-\frac{1}{4\sigma^2}$.

The next Proposition, proved in section \ref{pfinfiniteII}, is more difficult for Regime II than its analogue for Regime I, since $f=f_0+f_p$ vanishes at $\sigma_0\in\mathcal{Z}_{-i\tilde\alpha}$.

\begin{prop}\label{f12}
For perturbations $f_p$ with $N_p$ sufficiently small (recall Proposition \ref{d7}) the function $\xi=\xi_{f}(\sigma)$, which is initially defined on $\cZ_{cut}(\sigma_0)$, extends across the cut as a globally one-to-one analytic map of $\cZ_{-i\tilde\alpha}$ onto the open set which is its range.

\end{prop}

In the next Proposition we use the notation
\begin{align}\label{f6a}
Ai_0(z)=Ai(z),\;Ai_1(z)=Ai(ze^{-2\pi i/3}), \;Ai_{-1}(z)=Ai(ze^{2\pi i/3}).
\end{align}
We denote by $\mathcal{Z}_{-i\tilde\alpha,s}\subset \mathcal{Z}_{-i\tilde\alpha}$ an open subdomain, chosen as explained in the proof, and containing the image of the segment of the $x-$axis, $[M,\infty)$, under the map $x\to \sigma$.  We denote by $\Delta_\xi$ the image of $\mathcal{Z}_{-i\tilde\alpha,s}$ under the map $\sigma\to\xi(\sigma)$.

\begin{prop}\label{f6}
Suppose $\tilde\beta$  as  defined in \eqref{hr3} lies in Regime II, and set $\tilde\beta^2=-\tilde\gamma^2$, where $\arg\tilde\gamma$ is near $0$.
For  $f_p$ as in \eqref{d4} taken sufficiently small (by the choices explained in Remark \eqref{d3a}),  the perturbed Bessel problem \eqref{d3}  has exact solutions
\begin{align}\label{f7}
\begin{split}
&v_j(\sigma)=\xi^{-1/2}_\sigma(\sigma)\left(Ai_j(\tilde\gamma^{2/3}\xi(\sigma))+\eta_j(\tilde\gamma,\xi(\sigma))\right), \;j=0,1,-1
\end{split}
\end{align}
on $\mathcal{Z}_{-i\tilde\alpha,s}$, where the error term $\eta_j$ satisfies
\begin{align}\label{f6z}
\begin{split}
&|\eta_j(\tilde\gamma,\xi)|\leq \frac{C}{|\tilde\gamma|} |Ai_j(\tilde\gamma^{2/3}\xi)|\\
&|\partial_\xi\eta_j(\tilde\gamma,\xi)|\leq \frac{C}{|\tilde\gamma|} |\partial_\xi\left(Ai_j(\tilde\gamma^{2/3}\xi)\right)| \end{split}
\end{align}
 for $\xi\in\Delta_\xi$  with $|\xi|>>1$ and $\Re\xi>0$.

\end{prop}

\begin{rem}\label{f6za}
Information about the error terms $\eta_j$ for $\xi$ near negative infinity is more complicated to state, but is implicit in Proposition \ref{f9} in the case of $\eta_1$.  For explicit estimates of the $\eta_j$ we refer to Theorem 9.1 of Chapter 11 of \cite{O}.

\end{rem}



The next three propositions are analogues of the last three propositions in section \ref{onelead}.

\begin{prop}[Choice of decaying solution]\label{f9}
After shrinking $\omega$ and reducing $h_0$ if necessary, we have, for $\zeta\in \omega$ and $0<h<h_0$ such that $\tilde\beta(\zeta,h)$ lies in Regime II,  that the bounded (resp. decaying) solution of \eqref{e4yy} on $[M,\infty)$ for $\Re\zeta=0$ (resp. $\Re\zeta>0$) is given by
\begin{align}\label{f9zy}
w(x)=z(x)^{-1/2}v_1(\sigma(x))
\end{align}
Here $v_1$ is defined in  \eqref{f7} and the maps  $x\to z(x)$ and $x\to \sigma(x)$ are defined  by \eqref{f9zz}.   The corresponding decaying solution of Erpenbeck's $5\times 5$ system \eqref{f43} is thus given by the formula in \eqref{f47}.

\end{prop}

\begin{prop}[Decaying solution is of type $\theta_1$ at $x=M$]\label{f9z}
Let $\theta(x,\zeta,h)$ be the exact decaying solution of \eqref{f43} identified in Proposition \ref{f9}, and let $\theta_1(x,\zeta,h)$ be the approximate solution defined in \eqref{t10}.     There exist $h_0>0$, a  neighborhood $\omega\ni\zeta_\infty$, and a nonvanishing scalar function $H(x,\zeta,h)$ defined for $x$ near $M$ such that
\begin{align}\label{f9za}
|H(x,\zeta,h)\theta(x,\zeta,h)-\theta_1(x,\zeta,h)|\leq\frac{C}{|\tilde\beta(\zeta,h)|}|\theta_1(x,\zeta,h)|\text{ for }x\text{ near }M,
\end{align}
where $C$ is independent of $x$ near $M$ and of $\zeta\in\omega$, $0<h\leq h_0$ such that $\tilde\beta$ lies in Regime II.
\end{prop}


\begin{prop}[Decaying solution is of type $\theta_1$ at $x=0$]\label{f9w}
Let $\theta(x,\zeta,h)$ be the exact decaying solution of \eqref{f43} identified in Proposition \ref{f9},  and let $H(x,\zeta,h)$ be the function referred to in Proposition \ref{f9z}.    There exist $h_0>0$ and a  neighborhood $\omega\ni\zeta_\infty$ such that
\begin{align}\label{f9wa}
|H(M,\zeta,h)\theta(x,\zeta,h)-\theta_1(x,\zeta,h)|\leq\frac{C}{|\tilde\beta(\zeta,h)|}|\theta_1(x,\zeta,h)|\text{ on }[0,M],
\end{align}
where $C$ is independent of $x\in [0,M]$ and of $\zeta\in\omega$, $0<h\leq h_0$ such that $\tilde\beta$ lies in Regime II.
\end{prop}

\subsection{Regime III}\label{threelead}
Recall the definitions of the variables
\begin{align}\label{g8z}
t=\frac{2}{\mu}\sqrt{aD(\infty,\zeta)}e^{-\mu x/2}, \;s=t^2/4.
\end{align}
When $f_p(s)$ in \eqref{d6} is neglected, the integral defining $\xi(s)$ is
\begin{align}\label{g8}
\xi^{1/2}=\int^s_0 \frac{1}{2t^{1/2}}dt=s^{1/2}, \text{ so }\xi=s,
\end{align}
and the relevant domain in the $\xi$-plane is the bounded wedge $\mathcal{W}^2/4$.  Progressive paths in the $\xi$-plane are now either those along which both $\Re \xi^{1/2}$ and $|\xi|$ are nondecreasing, or those along which both $\Re\xi^{1/2}$ and $|\xi|$ are nonincreasing.   The image of $\mathcal{W}^2/4$ under the map $s\to\xi^{1/2}$ is just $\mathcal{W}/2$, and progressive paths are easy to choose in the $\xi^{1/2}$-plane.

When $f_p$ as in \eqref{d6} is included in the integral defining $\xi$
\begin{align}\label{g9}
2\xi^{1/2}=\int^s_0 f^{1/2}(t)dt,
\end{align}
the image of $\mathcal{W}^2/4$ under the map $s\to\xi^{1/2}$ is a small perturbation of $\mathcal{W}/2$ when $f_p$ is small, and  progressive paths satisfying the above conditions  are again not hard to choose.

\begin{prop}\label{g8t}
For perturbations $f_p$ with $N_p$ sufficiently small (recall Proposition \ref{d7}) the function $\xi=\xi_f(s)$ is a globally one-to-one analytic map of $\mathcal{W}^2/4$  onto its image.

\end{prop}

In the next Proposition we denote by $\mathcal{W}_s\subset \mathcal{W}^2/4$ an open subdomain,  chosen as explained in the proof, and  containing the image of the segment of the $x-$axis, $[M,\infty)$ under the map $x\to s$ given by \eqref{g8z}.  We let $\Delta_\xi$ denote the image of $\mathcal{W}_s$ under the map $s\to \xi(s)$.
With $\hat v(s)$ as in \eqref{d5} and $W(\xi)$ defined by $\hat v=(\frac{d\xi}{ds})^{-1/2}W$, the problem satisfied by $W$ has the form
\begin{align}\label{g10}
W_{\xi\xi}=\left(\frac{1}{h^2\xi}+\psi(\xi)\right)W=\left(\frac{1}{h^2\xi}+\frac{\tilde\beta^2-1}{4\xi^2}+\frac{\phi(\xi)}{\xi}\right)W\text{ on }\bW_\xi,
\end{align}
where (with $g(s)$ as in \eqref{d5})\footnote{Observe that when $f(s)=1/s$ and $\xi=s$, we have $\phi(\xi)=0$.}
\begin{align}
\phi(\xi)=\frac{1-4\tilde\beta^2}{16\xi}+\frac{g(s)}{f(s)}+\frac{4f(s)f''(s)-5f^{'2}(s)}{16 f^3(s)}.
\end{align}
Using the estimates of Proposition \ref{estimates}, one checks the finiteness of the integrals required for the error analysis of Theorem 9.1 of \cite{O}, Chapter 12:
\begin{align}\label{g11}
\int_{\alpha_j}^\xi |\phi(t)t^{-1/2}|\;d|t| \text{ (on progressive paths)}.
\end{align}
In this case there is no singularity at infinity, since $\Delta_\xi$ is bounded independent of $h$.

\begin{prop}\label{g12}
Suppose $\tilde\beta$  as  defined in \eqref{hr3} lies in Regime III.
For  $sf_p(s)$ as in \eqref{d6} taken sufficiently small (by the choices explained in Remark \eqref{d3a}),  the perturbed Bessel problem \eqref{d5} has exact solutions on $\mathcal{W}_s$ given by
\begin{align}\label{g12z}
\hat v_j(s)=\xi_s^{-1/2}(s)W_j(\xi(s)),\;j=1,2,
\end{align}
where the $W_j(\xi)$ are exact solutions of \eqref{g10} of the form
\begin{align}\label{g12a}
\begin{split}
&(a) W_1(\xi)=\xi^{1/2}I_{\tilde\beta}(2\xi^{1/2}/h)+\eta_1(h,\xi)\\
&(b) W_2(\xi)=\xi^{1/2}K_{\tilde\beta}(2\xi^{1/2}/h)+\eta_2(h,\xi).
\end{split}
\end{align}
Here the error term $\eta_1$ satisfies
\begin{align}\label{g12b}
\begin{split}
&|\eta_1(h,\xi)|\leq Ch |\xi^{1/2}I_{\tilde\beta}(2\xi^{1/2}/h)|\\
&|\partial_\xi\eta_1(h,\xi)|\leq Ch \left|\partial_\xi\left(\xi^{1/2}I_{\tilde\beta}(2\xi^{1/2}/h)\right)\right|
\end{split}
\end{align}
for $\xi\in\Delta_\xi$  with $|\xi^{1/2}/h|$ large.   The error $\eta_2$ satisfies analogous estimates.

\end{prop}

\begin{rem}\label{g12c}
Information about the error terms $\eta_j$ for $\xi$ near $0$ is more complicated to state, but is implicit in Proposition \ref{g13z} in the case of $\eta_1$.  For explicit estimates of the $\eta_j$ we refer to Theorem 9.1 of Chapter 12 of \cite{O}.    That theorem  deals only with real $\tilde\beta$, but we show how the result can be extended to $\Re\tilde\beta\geq 0$ in section \ref{pfinfiniteII}.

\end{rem}

\begin{prop}[Choice of decaying solution]\label{g13z}
After shrinking $\omega$ and reducing $h_0$ if necessary, we have, for $\zeta\in \omega$ and $0<h<h_0$ such that $\tilde\beta(\zeta,h)$ lies in Regime III,  that the bounded (resp. decaying) solution of \eqref{e4yy} on $[M,\infty)$ for $\Re\zeta=0$ (resp. $\Re\zeta>0$) is given by
\begin{align}\label{g14z}
w(x)=\frac{\sqrt{2}}{t(x)}\hat v_1(s(x)),
\end{align}
where $\hat v_1$ is defined in  \eqref{g12z} and the maps  $x\to t(x)$ and $x\to s(x)$ are defined  by \eqref{g8z}.   The corresponding decaying solution $\theta(x,\zeta,h)$ of Erpenbeck's $5\times 5$ system \eqref{f43} is thus given by the formula in \eqref{f47}.

\end{prop}

The next step is to show that this solution is of type $\theta_1$ at $x=M$.


\begin{prop}[Decaying solution is of type $\theta_1$ at $x=M$]\label{g17z}
Let $\theta(x,\zeta,h)$ be the exact decaying solution of \eqref{f43} identified in Proposition \ref{g13z}, and let $\theta_1(x,\zeta,h)$ be the approximate solution defined in \eqref{t10}.      There exist $h_0>0$, a  neighborhood $\omega\ni\zeta_\infty$, and a nonvanishing scalar function $H(x,\zeta,h)$ defined for $x$ near $M$ such that
\begin{align}\label{g18z}
|H(x,\zeta,h)\theta(x,\zeta,h)-\theta_1(x,\zeta,h)|\leq Ch|\theta_1(x,\zeta,h)|\text{ for }x\text{ near }M,
\end{align}
where $C$ is independent of $x$ near $M$ and of $\zeta\in\omega$, $0<h\leq h_0$ such that $\tilde\beta$ lies in Regime III.
\end{prop}

The proof of the next Proposition is exactly like that of Proposition \ref{f9w}.


\begin{prop}[Decaying solution is of type $\theta_1$
 at $x=0$]\label{g19z}
Let $\theta(x,\zeta,h)$ be the exact decaying solution of \eqref{f43} identified in Proposition \ref{g13z}, and let $\theta_1(x,\zeta,h)$ and $H(x,\zeta,h)$ be as in Proposition \ref{g17z}.    There exist $h_0>0$ and a  neighborhood $\omega\ni\zeta_\infty$ such that
\begin{align}\label{g20z}
|H(M,\zeta,h)\theta(x,\zeta,h)-\theta_1(x,\zeta,h)|\leq Ch|\theta_1(x,\zeta,h)|\text{ on }[0,M],
\end{align}
where $C$ is independent of $x\in [0,M]$ and of $\zeta\in\omega$, $0<h\leq h_0$ such that $\tilde\beta$ lies in Regime III.
\end{prop}

\section{Multi-step reactions}\label{multistep}

The treatment of the turning point at infinity in the case of a scalar reaction equation works verbatim for type D multi-step reactions provided the reactant $k$-vector $\lambda(x)$ is analytic and  has the structure
\begin{align}\label{structure}
\lambda(x)=Ae^{-\mu x}+ m(x)e^{-\mu x}, \;\text{where }A \text{ is constant},\; \mu>0, \text{ and }|m(x)|\leq Ce^{-\mu\Re x}
\end{align}
on a wedge $\bW(M_0,\theta)$ for some $\theta>0$.  With \eqref{structure} the function $e(x)$ again has the structure in \eqref{ac2} and the proof of Proposition \ref{estimates}, giving the estimates of the functions $h_j$, $j=1,2,3$ that appear in the perturbation of Bessel's equation \eqref{hr4}, goes through unchanged.
We now show that the eigenvalue separation condition \eqref{cond} implies \eqref{structure}; thus,  Example \ref{shorteg} satisfies \eqref{structure}.

We write the equation satisfied by $\lambda$ as 
\begin{align}\label{prob}
\lambda_x=f(\lambda)=B\lambda+N(\lambda), \text{ where }f(0)=0\text{ and }B=df(0),
\end{align}
and denote by $\Pi_{ws}$, $\Pi_{ss}$ the projections of $\bC^k$ onto, respectively, the weakly stable subspace corresponding to the  eigenvalue $-\mu_1$ of $B$,  and the complementary strongly stable subspace.  We can suppose $\lambda$ is already given as an $\bR^k$-valued decaying solution of \eqref{prob} on $\bR$.
For $M_0$ sufficiently large and $\theta$ small enough, the problem \ref{prob} with initial condition $\lambda|_{x=M_0}=\lambda(M_0)$ can be solved on the wedge $\bW(M_0,\theta)$
by a classical contraction argument applied to the integral equation
\begin{align}\label{ms1}
\lambda(x)= e^{Bx}\lambda(0) + \int_{0}^x e^{B(x-s)} N( \lambda(s))\,ds.
\end{align}
Here by a translation we have replaced $M_0$ by $0$.
By \eqref{cond}  the weakly stable subspace is simple with eigenvalue $-\mu:=-\mu_1$, so
we can rearrange \eqref{ms1} as
$$
\lambda(x)= e^{-\mu x}\left(\lambda_{ws}(0) + \int_{0}^x e^{\mu s} \Pi_{ws}N( \lambda(s))\,ds\right)
+\left(e^{Bx} \lambda_{ss}(0) + \int_{0}^x e^{B(x-s)} \Pi_{ss} N( \lambda(s))\,ds\right)
=:I + II.
$$

Using $|\lambda(x)|\leq C|\lambda(0)|e^{-\mu \Re x}$,  $|N|\leq C_2|\lambda(x)|^2$, and the estimate
\begin{align}
|e^{B(x-s)}\Pi_{ss}|\leq Ce^{-2\tilde\mu(x-s)}, \text{ where }\tilde\mu>2\mu,
\end{align}
which follows from the separation condition \eqref{cond}, we find
that $II$ is bounded in modulus by $C_3 |\lambda(0)|^2 e^{-2\mu \Re x}$, and so can be viewed as part of the second term on the right in \eqref{structure}.
Splitting $I$ now as
$$
I= e^{-\mu x}\left(\lambda_{ws}(0) + \int_{0}^\infty e^{\mu s} \Pi_{ws}N( \lambda(s))\,ds\right)
- \int_{x}^\infty e^{-\mu(x- s)} \Pi_{ws}N( \lambda(s))\,ds=: I_1 + I_2,
$$
we see  that $|I_2|\leq C_4 \int e^{-\mu \Re(x-s)}e^{-2\mu \Re s}d(\Re s)
\leq C_5 e^{-2\mu \Re x}$, and so $I_2$ can be treated like $II$ above.
Setting
$$
A:=
\lambda_{ws}(0) + \int_{0}^\infty e^{\mu s} \Pi_{ws}N( \lambda(s))\,ds,
$$
we obtain \eqref{structure}.

The treatment of  frequencies $\zeta\neq \zeta_\infty$ with $\Re\zeta\geq 0$ does not require \eqref{structure}, so for such frequencies the proofs in the scalar case work for multistep reactions as long as the assumptions of section \ref{assumptions} hold.   Thus, our main result, Theorem \ref{t14}, holds for also for type D multi-step reactions under the additional separation condition \eqref{cond}.


\part{Finite turning point and non-turning point frequencies}\label{finite}

In this part we treat non-turning point frequencies as well as frequencies $\zeta\in \mathrm{III}_+^o$, for each of which there exists a turning point $x(\zeta)\in (0,\infty)$.  We also study the upper endpoint frequency $\zeta_0$ for which the corresponding turning point is the endpoint $x(\zeta_0)=0$ of the reaction zone $[0,\infty)$.

First we give a lemma that  extends the map $\zeta\to x(\zeta)$ to a neighborhood of a turning point frequency.

\begin{lem}\label{q11}Fix a basepoint  $\uzeta\in \mathrm{III}_+^0$.  There exist neighborhoods $\omega\ni\uzeta$ and $\mathcal O\ni x(\uzeta)$ and  an analytic homeomorphism
 $x:\omega\to \mathcal{O}$, where $x(\zeta)$ is defined to be the unique root of
 \begin{align}\label{q11a}
 f(x,\zeta):=\zeta^2+c_0^2\eta(x)=0.
 \end{align}
 Moreover,
 \begin{align}\label{q12}
 \Im x(\zeta)\geq 0 \text{ for }\Re\zeta\geq 0\text{ and }\Im x(\zeta)= 0 \Leftrightarrow \Re\zeta= 0.
 \end{align}

\end{lem}

\begin{proof}
The profile $P(x)$ is of type D, so $f_x(x(\uzeta),\uzeta)< 0$.  The fact that $x(\zeta)$ is analytic thus follows from the implicit function theorem. We have
\begin{align}\label{q13}
f_x(x(\zeta),\zeta)x_\zeta(\zeta)+2\zeta=0,
\end{align}
so $x_\zeta(\uzeta)\neq 0$ since $\uzeta\neq 0$. Hence we have an analytic homeomorphism of some neighborhoods $\omega$ and $\mathcal O$.  Since $\Im\uzeta>0$, \eqref{q13} implies $\Im x_\zeta(\uzeta)>0$, which yields \eqref{q12}.

\end{proof}

The frequencies $\zeta\in \{\Re\zeta\geq 0\}\setminus \mathrm{III}:=\cN$, for which are no associated turning points, are divided into two sets:
\begin{align}
\cN=(\cN\cap\{|\zeta|\geq M\})\cup(\cN\cap\{|\zeta|\leq M\}),
\end{align}
for some sufficiently large $M$.  The unbounded set is studied in section \ref{unbounded}.  The bounded set was treated in \cite{LWZ1} using the following theorem, which we reproduce here since it is needed for the analysis of finite turning points.  In this Theorem the $\mu_j(x,\zeta)$, $j=1,\dots,5$  are the eigenvalues of $\Phi_0(x,\zeta)$ given in \eqref{t4}, and $\mu>0$ is the constant determining the rate of profile decay in \eqref{o6a}.

\begin{thm}[\cite{LWZ1}, Theorem 2.1]\label{conjugation2}
\textbf{1. } Consider the system \eqref{f43}
\begin{align}\label{q1}
\theta'=\frac{1}{h}\left[\Phi_0(x,\zeta)+h\Phi_1(x)\right]\theta
\end{align}
on an interval $[a,\infty)$, $a\geq 0$, and for values of $\zeta$ such that
\begin{align}\label{q2}
|\mu_1(x,\zeta)-\mu_j(x,\zeta)|\geq C_\zeta>0,\;j=2,\dots,5 \text{ for }0<h\leq h(\zeta) \text{ small enough}.
\end{align}
Then there exists an exact solution $\theta(x,\zeta,h)$ such that for any $0<\delta_*<\mu$
\begin{align}\label{q3}
\left|\theta-e^{\frac{1}{h}\int^{ x}_0 \mu_1^\sharp(s,\zeta,h)ds}\left[T_1(x,\zeta)+O(h)\right]\right|\leq C_\zeta h e^{-\delta_* x} |e^{\frac{1}{h}\int^{x}_0 \mu_1^\sharp( s,\zeta,h)d s}|\;\;\text{ on }[a,\infty),
\end{align}
where $T_1(x,\zeta)$ as in \eqref{t6}, and
\begin{align}\label{q4}
\mu_1^\sharp=\mu_1(x,\zeta)+O(he^{-\mu x}).
\end{align}

\textbf{2. }Let $K\subset \{\Re\zeta\geq 0\}\setminus \mathrm{III}$ be compact.  Then \eqref{q2} and \eqref{q3} hold on $[0,\infty)$ with constants $h(\zeta)$, $C_\zeta$ that can be taken independent of $\zeta\in K$.

\textbf{3. }Let $\uzeta\in \mathrm{III}_+^o$ and $\delta>0$.  There exists a neighborhood $\omega_1\ni\uzeta$ in $\Re\zeta\geq 0$ such that
\begin{align}\label{q4a}
x(\uzeta)-\delta<\Re x(\zeta)<x(\uzeta)+\delta\text{ for all }\zeta\in\omega_1
\end{align}
and such that
\eqref{q2} and \eqref{q3} hold on $[x(\uzeta)+\delta,\infty)$ with constants $h(\zeta)$, $C_\zeta$ that can be taken independent of $\zeta\in \omega_1$.
\end{thm}

As an immediate corollary of part \textbf{2} we have
\begin{cor}[Non-turning point frequencies]\label{noturn}
Let $K\subset \{\Re\zeta\geq 0\}\setminus \mathrm{III}$ be compact. The exact bounded solution $\theta(x,\zeta)$ of \eqref{q1} given by Theorem \ref{conjugation2} satisfies
\begin{align}\label{q5}
|\theta(0,\zeta,h)-T_1(0,\zeta)|\leq C_Kh \text{ for }0\leq h\leq h_K,
\end{align}
where $C_K$ and $h_K$ can be taken independent of $\zeta\in K$.
\end{cor}

The Corollary implies that for $\zeta\in K$, $0<h\leq h_K$, the solution  $\theta$ is of type $\theta_1$ at $x=0$ and thus the stability function $V(\zeta,h)$ is nonvanishing.  In view of \eqref{q3} and \eqref{q4}, part \textbf{3} of Theorem \ref{conjugation2} yields:
\begin{cor}\label{q5a}
Let $\omega_1\ni\uzeta$ and $\delta>0$ be as in \eqref{q4a}.  For $\zeta\in\omega_1$ there is a bounded, nonvanishing function $H(x,\zeta,h)$ and an $h_0>0$ such that
\begin{align}\label{q5aa}
|H(x,\zeta,h)\theta(x,\zeta,h)-\theta_1(x,\zeta,h)|\leq Ch|\theta_1(x,\zeta,h)|\text{ on }[x(\uzeta)+\delta,\infty)
\end{align}
for $0<h\leq h_0$.
\end{cor}

In the next section we show  that there is a nonvanishing scalar function $s(\zeta,h)$ such that
$s(\zeta,h)\theta(x,\zeta,h)$ is of type $\theta_1$ at $x(\uzeta)-2\delta$ for $\zeta\in\omega_1$.  This is done by matching arguments that use Airy functions to represent  exact solutions on a full neighborhood of the turning points.
It then follows as in Proposition \ref{f9w} that $s(\zeta,h)\theta(x,\zeta,h)$ is of type $\theta_1$ at $x=0$.

\section{Turning points in $(0,\infty)$.}\label{finitetp}

We now fix a basepoint $\uzeta\in \mathrm{III}_+^o$ with associated turning point $x(\uzeta)\in (0,\infty)$.  The goal is to find a neighborhood $\omega\ni\uzeta$ and a constant $h_0>0$ such that the stability function $V(\zeta,h)\neq 0$ for
$\zeta\in\omega$ and $0<h\leq h_0$.  The first step is to conjugate the system \eqref{q1} to the block form \eqref{h7}:

\begin{align}\label{q6}
h\phi_x=\begin{pmatrix}A^0_{11}+hd_{11}+h^2\beta_{11}&0\\0&A^0_{22}+hd_{22}+h^2\beta_{22}\end{pmatrix}\phi:=\begin{pmatrix}A_{11}(x,\zeta,h)&0\\0&A_{22}(x,\zeta,h)\end{pmatrix}\phi,
\end{align}
for $\zeta$ near $\uzeta$ and $x$ in a complex neighborhood of $x(\uzeta)$.

\begin{prop}\label{q7}
 Let $\uzeta\in \mathrm{III}_+^o$ and let $x(\uzeta)\in (0,\infty)$ be the corresponding turning point.  There exists a constant $h_0>0$ and simply connected open neighborhoods  $\omega\ni\uzeta$, $\mathcal O\ni x(\uzeta)$   such that for $\zeta\in\omega$ and $0<h\leq h_0$ a conjugator $Y(x,\zeta,h)$ can be constructed on $\mathcal O$ with the property that
$\theta(x,\zeta,h)$ satisfies the Erpenbeck system \eqref{q1} on $\mathcal O$ if and only if $\phi$ defined by $\theta=Y\phi$
satisfies \eqref{q6}.  The conjugator $Y(x,\zeta,h)$ is bounded and analytic in its arguments.  The entries of the $2\times 2$ block $A_{11}(x,\zeta,h)$ in \eqref{q6} again have the form given in \eqref{f45}.

\end{prop}

\begin{proof}
The conjugator is constructed as $Y=Y_1Y_2$, where $Y_1(x,\zeta)$ is  given by \eqref{h2} and $Y_2(x,\zeta,h)$  has the form \eqref{h6a}.  The entries of $Y_{2}$ satisfy equations like \eqref{f45b} and are constructed by  a classical contraction argument;  see, for example, Theorem 6.1-1 of \cite{Wa}.   The analyticity of the $Y_j$ in $x$ is a consequence of the fact that the profile $P(x)$ extends analytically to a complex neighborhood of $x(\uzeta)$.
As in Proposition \ref{h10z}, the argument uses the fact that the blocks $A^0_{11}(x,\zeta)$ and $A^0_{22}(x,\zeta)$ in \eqref{h5} have no eigenvalues in common for $(x,\zeta)\in \mathcal O\times \omega$ for small enough $\mathcal O$, $\omega$.

\end{proof}

Writing $\phi=(\phi_1,\phi_2)$ and letting $\varphi_0(x,\zeta,h)$ denote any function such that $d_x\varphi_0=\frac{a+d}{2}$, we obtain by the same calculations that produced \eqref{f46}  that
solutions of $hd_x\phi_1=A_{11}(x,\zeta,h)\phi_1$ in $\mathcal O$ are given by

\begin{align}\label{q8}
\phi_1=e^\frac{\varphi_0}{h}\begin{pmatrix}b^{1/2}&0\\\alpha b^{-1/2}-h(b^{-1/2})_x&b^{-1/2}\end{pmatrix}\begin{pmatrix}w\\hw_x\end{pmatrix}:=K(x,\zeta,h)\begin{pmatrix}w\\hw_x\end{pmatrix},
\end{align}
where $(w,hw_x)$ satisfies
\begin{align}\label{q9}
h\begin{pmatrix}w\\hw_x\end{pmatrix}_x=\begin{pmatrix}0&1\\C(x,\zeta)+hr(x,\zeta,h)&0\end{pmatrix}\begin{pmatrix}w\\hw_x\end{pmatrix}.
\end{align}
Thus, we can construct two independent solutions on $\mathcal{O}$ of the original system \eqref{q1} of the form
\begin{align}\label{q10}
\theta=Y(x,\zeta,h)\begin{pmatrix}K(x,\zeta,h)\begin{pmatrix}w\\hw_x\end{pmatrix}\\0\end{pmatrix},
\end{align}
using independent solutions of \eqref{q9}.

\begin{rem}\label{q15}
In the remainder of this section the simply connected neighborhoods $\omega\ni \uzeta$ and $\mathcal O\ni x(\uzeta)$ may need to be reduced in size a finite number of times.  These reductions will often be performed without comment.
\end{rem}

 The following propositions will allow us to construct solutions of \eqref{q9} in terms of Airy functions.
 Recall that for $\zeta\in \mathrm{III}_+^0$ the number $x(\zeta)\in (0,\infty)$ is the unique root of $\zeta^2+c_0^2\eta(x)=0$.

Let us write the function $C(x,\zeta)$ in \eqref{q9} as
\begin{align}\label{q14}
\begin{split}
&C(x,\zeta)=(\zeta^2+c_0^2\eta(x))\ub^2(x)=(x-x(\zeta))d(x,\zeta),\\
 &\qquad \text{ where }d(x,\zeta)=\int^1_0 C_x(x(\zeta)+t(x-x(\zeta)),\zeta)dt.
\end{split}
\end{align}

\begin{prop}\label{q16}
The equation
\begin{align}\label{q17}
\rho^2_x\rho=C(x,\zeta)
\end{align}
has a solution for $(x,\zeta)\in\mathcal O\times \omega$ given by
\begin{align}\label{q18}
\rho(x,\zeta)=(x(\zeta)-x)\left(\int^1_0 3u^2\sqrt{-d(x(\zeta)+u^2(x-x(\zeta)), \zeta)} \;du\right)^{2/3},
\end{align}
where the expression inside the square root and the square root itself are positive when $x$ is real and $\Re\zeta=0$.   The function $\rho$ is analytic in both $x$ and $\zeta$ and satisfies:
\begin{align}\label{q19}
\begin{split}
&a)\;\rho(x(\zeta),\zeta)=0 \text{ for }\zeta\in\omega;\\
&b)\;\text{for }x\text{ real and }\Re\zeta=0, \;\rho(x,\zeta)\text{ is real  and }\rho_x(x,\zeta)<0;\\
&c)\;\text{for each }\zeta\in\omega,\; \rho(\cdot,\zeta) \text{ is an analytic homeomorphism of }\mathcal{O}\text{ onto a neighborhood }\mathcal{O}_\zeta\ni 0;\\
&d)\;\text{ for }\Re\zeta=0, \text{ we have }(-i\rho_\zeta)(x(\zeta),\zeta)>0.
\end{split}
\end{align}
\end{prop}

\begin{rem}\label{q19a}
The inequality \eqref{q19}(d) implies that the map $\zeta\to\rho(x(\uzeta),\zeta)$ is an analytic homeomorphism of a neighborhood of $\uzeta$ onto a neighborhood of $0$.  From \eqref{q19}(b) and (d) we see that for $\zeta$ near $\uzeta$, when $\Re\zeta>0$ we have $\Im \rho(x(\uzeta),\zeta)>0$.   After  shrinking $\omega$ if necessary, we conclude that for real $x$  near $x(\uzeta)$ and $\zeta\in\omega$ with $\Re\zeta>0$, we have $\Im\rho(x,\zeta)>0$.

\end{rem}

The system \eqref{q9} is equivalent to the scalar equation
\begin{align}\label{q20}
h^2w_{xx}=(C(x,\zeta)+hr(x,\zeta,h))w.
\end{align}
Using \eqref{q17}, the  property \eqref{q19}(c), and  for each $\zeta\in\omega$ making the changes of variables
\begin{align}\label{q20a}
y=\rho(x,\zeta),\; W(y,\zeta):=w(x(y,\zeta)),
\end{align}
we find that \eqref{q20} takes the form (suppressing some $\zeta$ arguments)
\begin{align}\label{q21}
h^2\rho_x(x(y))d_y\left(\rho_x(x(y))W_y\right)=[y\rho_x^2(x(y))+hr(x(y),\zeta,h)]W.
\end{align}
The transformation
\begin{align}
f(y)=(\rho_x(x(y))^{1/2}W(y)
\end{align}
leads to
\begin{align}\label{q22}
h^2f_{yy}=(y+hq(y,h))f, \text{ where }q(y,h)=r\rho_x^{-2}+h\rho_x^{-1/2}d^2_y(\rho_x^{1/2}).
\end{align}
This is a perturbation of Airy's equation that we can rewrite as
\begin{align}\label{q23}
h\begin{pmatrix}f\\hf_y\end{pmatrix}_y=\begin{pmatrix}0&1\\y+hq(y,h)&0\end{pmatrix}\begin{pmatrix}f\\hf_y\end{pmatrix}.
\end{align}

The following Proposition is a classical result of turning point theory.  A reference for the proof is \cite{Wa}, Theorem 6.5-1.
\begin{prop}[Exact conjugation to Airy's equation]\label{q23a}
There exists $h_0>0$ and a conjugator $P(y,\zeta,h)=I+hQ(y,\zeta,h)$, with $Q$ bounded and analytic in its arguments $x\in \mathcal{O}$, $\zeta\in\omega$, $0<h\leq h_0$,\footnote{Recall Remark \ref{q15}.}such that the transformation $(f,hf_y)=PZ$ takes \eqref{q23} into the equation
\begin{align}\label{q24}
hZ_y=\begin{pmatrix}0&1\\y&0\end{pmatrix}Z.
\end{align}
\end{prop}

For any $\zeta\in\omega$ two independent solutions of \eqref{q24} on $\mathcal O_\zeta$ (as in \eqref{q19}(c)) are given by
\begin{align}\label{q25}
Z_\pm(y)=\begin{pmatrix}Ai(h^{-2/3}e^{\pm 2\pi i/3}y)\\h^{1/3}e^{\pm 2\pi i/3}Ai(h^{-2/3}e^{\pm 2\pi i/3}y)\end{pmatrix}.
\end{align}

We recall that the phase function $\varphi_0$ in \eqref{q8} is required to be a primitive of $\frac{a+d}{2}$,
where $a=\ua+O(h)$, $d=\ud+O(h)$ (see \eqref{h7} and \eqref{f45}). Since $\ua=\ud$ is defined for all $x\geq 0$ and for all $\zeta$, and morever extends analytically to a complex neighborhood of the positive real axis, we may (and do)  henceforth take $\varphi_0$ of the form
\begin{align}\label{q25a}
\varphi_0(x,\zeta,h)=\int^x_0\ua(s,\zeta)ds+O(h).
\end{align}
Using the formula \eqref{q10} for $\theta$ and retracing through the changes of variables, we obtain

\begin{prop}[Exact solutions of \eqref{q1}] \label{q26}
For each $\zeta\in\omega$ two exact independent solutions $\theta_\pm(x,\zeta,h)$ on $\mathcal{O}$ of the original $5\times 5$ system \eqref{q1} are given by formula \eqref{q10} with
\begin{align}\label{q27}
\begin{pmatrix}w\\hw_x\end{pmatrix}_\pm=\begin{pmatrix}\rho_x^{-1/2}&0\\h\rho_x d_y(\rho_x^{-1/2})&\rho_x^{1/2}\end{pmatrix}\begin{pmatrix}f\\hf_y\end{pmatrix}_\pm=\begin{pmatrix}\rho_x^{-1/2}&0\\h\rho_x d_y(\rho_x^{-1/2})&\rho_x^{1/2}\end{pmatrix}P(\rho,\zeta,h)Z_\pm(\rho)
\end{align}
\end{prop}

Ignoring relative $O(h)$ errors in \eqref{q10}, we obtain by these substitutions $\theta_\pm(x,\zeta,h)\sim$
\begin{align}\label{q28}
\begin{split}
\quad e^{\varphi_0/h}\left[b^{1/2}(\rho_x)^{-1/2} Ai(h^{-2/3}\rho e^{\pm2\pi i/3})P_0+b^{-1/2}h^{1/3}(\rho_x)^{1/2}e^{\pm2\pi i/3} Ai'(h^{-2/3}\rho e^{\pm 2\pi i/3})Q_0\right].
\end{split}
\end{align}

We now choose $\delta>0$ and a neighborhood $\omega_1\ni \uzeta$ satisfying \eqref{q4a} as in Corollary \ref{q5a}, and so that for $\mathcal{O}\ni x(\uzeta)$ as  Proposition \ref{q26} we have
\begin{align}\label{q29}
x(\omega_1)\cup [x(\uzeta)-2\delta,x(\uzeta)+2\delta]\subset \mathcal{O}.
\end{align}
The next Proposition shows that for $H$ as in  Corollary \ref{q5a}, a nonvanishing multiple of the exact decaying solution $H(x(\uzeta)+2\delta,\zeta,h)\theta$  is of type $\theta_1$ at $x(\uzeta)-2\delta$.

\begin{prop}\label{q30}
Let $x_L=x(\uzeta)-2\delta$ and $x_R=x(\uzeta)+2\delta$. For $\zeta\in\omega_1$ there is an $h_0>0$ and a nonvanishing scalar function $\alpha(\zeta,h)$ such that
\begin{align}\label{q30a}
|\alpha(\zeta,h)H(x_R,\zeta,h)\theta(x_L,\zeta,h)-\theta_1(x_L,\zeta,h)|\leq Ch|\theta_1(x_L,\zeta,h)|
\end{align}
for $0<h\leq h_0$.
\end{prop}

The proof, given in section \ref{pffinite}, is based on expanding $H(x_R,\zeta,h)\theta(x,\zeta,h)$ in a basis of local exact solutions of \eqref{q1}, $\cB=\{\theta_-,\theta_+,\overline\theta_3,\overline\theta_4,\overline\theta_5\}$, where the $\overline\theta_j$, $j=3,4,5$ are of type $\theta_j$ for approximate solutions $\theta_j$ as in \eqref{t10}.
Using the expansions \eqref{f57} for the Airy function, we show first that appropriate multiples of $\theta_-$ and $\theta_+$ are, respectively, of type $\theta_1$ and $\theta_2$ at $x_R$.  Corollary \ref{q5a} and the explicitly known asymptotic behavior (in $h$) of the elements of $\cB$ at both $x_R$ and $x_L$ then allow us to conclude that \eqref{q30a} holds.

Finally, the proof of Proposition \ref{e4u} yields

\begin{prop}\label{q31}
There is a neighborhood $\omega_1\ni\uzeta$ and an $h_0>0$ such that for $\zeta\in\omega_1$ and $\alpha$ as in Proposition \ref{q30}:
\begin{align}\label{q32}
|\alpha(\zeta,h)H(x_R,\zeta,h)\theta(x,\zeta,h)-\theta_1(x,\zeta,h)|\leq Ch|\theta_1(x,\zeta,h)|\text{ on }[0,x_L]
\end{align}
for $0<h\leq h_0$.
\end{prop}

\section{The turning point at $0$.}\label{tpzero}

 For the boundary point frequency $\zeta_0\in \mathrm{III}_+$ the point $x=0$, where we need explicit information about the exact decaying solution $\theta$ in order to evaluate the stability function $V(\zeta,h)$, coincides with the turning point.   This fact presents some new difficulties that we sketch after stating Proposition \ref{q34}.

The first step in treating the turning point at $x=0$ is to extend the detonation profile $p(x)$ analytically to a complex neighborhood of $x=0$; this allows us to study the Erpenbeck system \eqref{q1} on a neighborhood of $x=0$.  We can then immediately extend Lemma \ref{q11} to obtain an analytic homeomorphism $x:\omega\to \mathcal{O}$, where now $x(\zeta_0)=0$, $\omega\ni\zeta_0$ and $\mathcal{O}\ni 0$.  Similarly, with no changes in the proofs we obtain extensions of Theorem \ref{conjugation2} part \textbf{3}, Corollary \ref{q5a}, and Propositions \ref{q7}, \ref{q16}, \ref{q23a}, and \ref{q26} to the case where $\uzeta$ is now $\zeta_0$.  In particular, we obtain exact solutions
$\theta_\pm$ on $\mathcal{O}$ satisfying \eqref{q28}.

We now choose $\delta>0$ and a neighborhood $\omega_1\ni \zeta_0$ satisfying \eqref{q4a} as in Corollary \ref{q5a}, and so that for $\mathcal{O}\ni 0$ as  Proposition \ref{q26} we have
\begin{align}\label{q33}
x(\omega_1)\cup [-2\delta,2\delta]\subset \mathcal{O}.
\end{align}
The next Proposition shows that for $H$ as in  Corollary \ref{q5a}, a nonvanishing multiple of the exact decaying solution $H(2\delta,\zeta,h)\theta$  is of type $\theta_1$ at $0$.

\begin{prop}\label{q34}
Fix $\kappa>0$. There exists a neighborhood $\omega_2\ni\zeta_0$ with $\omega_2\subset\omega_1$, an $h_0>0$, and a nonvanishing scalar function $\alpha(\zeta,h)$ such that
\begin{align}\label{q35}
|\alpha(\zeta,h)H(2\delta,\zeta,h)\theta(0,\zeta,h)-\theta_1(0,\zeta,h)|\leq \kappa|\theta_1(0,\zeta,h)|
\end{align}
for $\zeta\in\omega_2$ and $0<h\leq h_0$.  Both $\omega_2$ and $h_0$ depend on $\kappa$.
\end{prop}

As with Proposition \ref{q30} the proof involves working with a local basis of exact solutions
$\cB=\{\theta_-,\theta_+,\overline\theta_3,\overline\theta_4,\overline\theta_5\}$, and again we make use of the expressions \eqref{q28} for $\theta_\pm$ in terms of Airy functions.
However, since $\rho(0,\zeta_0)=0$, we can not use the  Airy function expansions (\eqref{f57}, for example) when $\zeta$ is too close to $\zeta_0$.  Since the arguments of the  Airy functions in \eqref{q28} are
$h^{-2/3}\rho(x,\zeta)e^{\pm 2\pi i/3}$, we see that there are two natural frequency regimes to consider:
\begin{align}\label{q36}
\begin{split}
&\text{Regime A}=\{(\zeta,h):|\rho(0,\zeta)h^{-2/3}|\leq M\\
&\text{Regime B}=\{(\zeta,h):|\rho(0,\zeta)h^{-2/3}|\geq M.
\end{split}
\end{align}
Here $\zeta\in\omega_1$ and $M$ is chosen large enough so that standard expansions of $Ai(z)$ apply in  $|z|\geq M$; thus, we are able to use those expansions in the analysis of Regime B.

 In the   proof of Proposition \ref{q31} for turning points in $(0,\infty)$, it was helpful that the arguments of $\rho(x_R,\zeta)$ and $\rho(x_L,\zeta)$ were always close to $\pi$ and $0$, respectively, for $\zeta$ near $\uzeta$, and thus $e^{\pm 2\pi i/3}\rho(x_{R,L},\zeta)$ stayed away from the negative real axis, where the zeros of $Ai$ and $Ai'$ are located.
  Now, however, the argument of $\rho(0,\zeta)$ can take on all values in $[0,\pi]$ for $\zeta$ near $\zeta_0$.  The analysis of $\theta_+$  is complicated by the fact that for $\arg(\rho(0,\zeta))\sim \pi/3$, we have $\arg  \left(e^{2 \pi i/3}\rho(0,\zeta)\right)\sim \pi$.

The formula \eqref{t10} for the approximate solution $\theta_1$ shows that
\begin{align}\label{q37}
\theta_1(0,\zeta,h)=P_0(0,\zeta)+s(0,\zeta)Q_0+O(h).
\end{align}
The proof of Proposition \ref{q34} makes use of the fact that for $\zeta$ near $\zeta_0$, the functions $s(0,\zeta)$ and $\rho(0,\zeta)$ are both close to zero.  This implies that the terms involving $Q_0$ in both \eqref{q37} and the expression \eqref{q28} for $\theta_-$ are small compared to the terms involving $P_0$.

\section{The case $|\zeta|\geq M$.}\label{unbounded}
We conclude Part \ref{finite} by treating the case $|\zeta|\geq M>>1$.  We must show that there exists $h_0>0$ and $M$ such that  for all $0<h\leq h_0$ and $|\zeta|\geq M$ with $\Re\zeta\geq 0$, the decaying (or bounded when $\Re\zeta=0$) solution $\theta(x,\zeta,h)$ of Erpenbeck's $5\times 5$ system,
\begin{align}\label{k0}
h\theta_x=(\Phi_0(x,\zeta)+h\Phi_1(x))\theta,
\end{align}
is of type $\theta_1$ at $x=0$.  As noted above, this implies nonvanishing of the stability function $V(\zeta,h)$.
Here there are no turning points, but the difficulty is to give a \emph{uniform} treatment of the noncompact set of parameters $\zeta$.   This case was studied on p. 610 of
\cite{E3}, but the choice of $h_0$ there was not uniform with respect to large $\zeta$, and so we are not able to use this result.

\begin{prop}\label{k1}
Let $\theta(x,\zeta,h)$ be as just described and let $h_0=1$.  There exists $M>0$ such that for $|\zeta|\geq M$ and $0<h\leq h_0$ we have
\begin{align}\label{k3}
|\theta(0,\zeta,h)-T_1(0,\zeta)|\leq Ch/|\zeta|,
\end{align}
where $C>0$ is independent of $(\zeta,h)$.

\end{prop}

\begin{proof}
\textbf{1. }First we rewrite equation \ref{k0} as
\begin{align}\label{k4}
\theta_x=\frac{|\zeta}{h}|\left(\tilde\Phi_0(x,\zeta)+\frac{h}{|\zeta|}\Phi_1(x)\right),
\end{align}
where $\Phi_1(x)=(A_x^{-1}B(x))^t$ as before, and
\begin{align}\label{k5}
\tilde\Phi_0(x,\zeta)=\frac{1}{|\zeta|}\Phi_0=\frac{\zeta}{|\zeta|}(A_x^{-1})^t+\frac{i}{|\zeta|}(A_x^{-1}A_y)^t.
\end{align}
The eigenvalues of $\tilde\Phi_0(x,\zeta)$ are $\tilde\mu_j(x,\zeta):=\frac{1}{|\zeta|}\mu_j(x,\zeta)$ for $\mu_j$ as in \eqref{t4}.

\textbf{2. }As in section 2 of \cite{LWZ1}, direct computation and the use of Assumption \ref{rate} shows that for $\mu>0$ as in \eqref{f44s}
\begin{align}\label{k6}
(A_x^{-1}B)^t(x)=O(e^{-\mu x})+\begin{pmatrix}0\\\mathrm{row\; 5}\end{pmatrix}, \text{ where row 5 }=(*,*,*,*,-r_\lambda/u).
\end{align}
This implies that the eigenvalues of $\tilde\Phi_0(x,\zeta)+\frac{h}{|\zeta|}\Phi_1(x)$ are
\begin{align}\label{k7}
\begin{split}
&\mu_j^*:=\frac{1}{|\zeta|}\mu_j(x,\zeta)+O(\tilde h e^{-\mu x}), \;j=1,2,3,4\;\;\;\tilde h:=\frac{h}{|\zeta|}\\
&\mu_5^*=\frac{1}{|\zeta|}\mu_3(x,\zeta)-\tilde h \frac{r_\lambda}{u}+O(\tilde h e^{-\mu x}),\text{ where }r_\lambda<0.
\end{split}
\end{align}

\textbf{3. Uniform separation of eigenvalues. }Using the fact that
\begin{align}\label{k8}
\begin{split}
&\mu_2(x,\zeta)-\mu_1(x,\zeta)=\frac{2\kappa s}{\eta u},\;\;s(x,\zeta)=\sqrt{\zeta^2+c_0^2\eta(x)}\\
&\mu_3(x,\zeta)-\mu_1(x,\zeta)=\frac{\zeta+\kappa s}{\eta u},
\end{split}
\end{align}
and noting that $s(x,\zeta)\sim \zeta$ for $|\zeta|>>1$, we obtain for large enough $M$
\begin{align}\label{k9}
|\tilde\mu_1(x,\zeta)-\tilde \mu_j(x,\zeta)|\geq C>0, \;j=2,\dots,5,
\end{align}
where $C$ is independent of $x\in [0,\infty)$ and $|\zeta|\geq M$.  Moreover, since $\Re\mu_j\geq \Re\mu_1$, $j=2,\dots,5$, we find from \eqref{k7} that
\begin{align}\label{k10}
\Re\mu_j^*(x,\zeta,\tilde h)-\Re\mu_1^*(x,\zeta,\tilde h)\geq O(\tilde he^{-\mu  x}),\;j=2,\dots,5
\end{align}
uniformly for $x\in [0,\infty)$ and $|\zeta|\geq M$.

\textbf{4. Conclusion. }As a consequence of the separation inequalities \eqref{k9} and \eqref{k10}, we are in a position to apply (verbatim) the proof of Theorem 2.1 of \cite{LWZ1} to the system
\begin{align}\label{k11}
\theta_x=\frac{1}{\tilde h}\left(\tilde\Phi_0(x,\zeta)+\tilde h\Phi_1(x)\right),
\end{align}
where $\tilde\Phi_0$ and $\tilde h$ play the roles of, respectively, $\Phi_0$ and $h$ in the earlier proof.  For a possibly larger choice of $M$, the argument there\footnote{This argument is based on the ``Variable Coefficient Gap Lemma" stated in Appendix A of \cite{LWZ1}, and first proved in \cite{Z1}.} shows that $\theta$ satisfies
\begin{align}\label{k12}
\left|\theta(x,\zeta,h)-e^{\frac{1}{\tilde h}\int^{ x}_0 \mu_1^\sharp(s,\zeta,\tilde h)ds}\left[T_1(x,\zeta)+O(\tilde h)\right]\right|\leq C \tilde h e^{-\delta x} |e^{\frac{1}{\tilde h}\int^{x}_0 \mu_1^\sharp( s,\zeta,\tilde h)d s}|\;\;\text{ on }[0,\infty),
\end{align}
where $0<\delta<\mu$, $C$ is independent of $|\zeta|\geq M$, and
\begin{align}
\mu_1^\sharp(s,\zeta,\tilde h)=\mu_1^*(s,\zeta,\tilde h)+O(\tilde h e^{-\mu x}).
\end{align}
Evaluating \eqref{k12} at $x=0$ we obtain \eqref{k3}.

\end{proof}

\part{Proofs for part \ref{infinite}}\label{pfinfinite}

\section{Conjugation to block form}\label{conjugation}

In this section we prove Lemma \ref{h10} and Proposition \ref{h10z}.

\begin{proof}[Proof of Lemma \ref{h10}]
 We give the proof for $d_{12}$; the proof for $d_{11}$ is quite similar.  The assertion for $\beta_{11}$ then follows from $\beta_{11}=d_{12}\alpha_{21}$ and the boundedness of $\alpha_{21}$.

Since the $x-$dependence of $d_{12}$ enters only through the profile, it suffices to show $d_{12}(\infty)=0$.
  We have
\begin{align}
D:=\begin{pmatrix}d_{11}&d_{12}\\d_{21}&d_{22}\end{pmatrix}=Y_1^{-1}\Phi_1Y_1-Y_1^{-1}d_xY_1,
\end{align}
so we need only show that that $(1,2)$ entry of $Y_1^{-1}\Phi_1Y_1$ is $0$ at $x=\infty$.  Here $d_{12}$ is a $2\times 3$ submatrix of the $5\times 5$ matrix $D$.

On pages 116-117 of \cite{E2} Erpenbeck writes $\Phi_1=W_{10}+W_{11}$, where $W_{10}(\infty)=0$ and $W_{11}$ is a matrix whose first four rows vanish at $\infty$\footnote{Here we use Assumption \ref{rate}.}.  So it suffices to consider $Y_1^{-1}W_{11}Y_1(\infty)$.    Suppressing  evaluations at $\infty$, we write
\begin{align}
W_{11}=\begin{pmatrix}w_a&w_b\\w_c&w_d\end{pmatrix}
\end{align}
where $w_a=0$ ($w_a$ is $2\times 2$), $w_b=0$ and
\begin{align}
w_c=\begin{pmatrix}0&0\\0&0\\ *&*\end{pmatrix},\;w_d=\begin{pmatrix}0&0&0\\0&0&0\\ *&*&*\end{pmatrix}.
\end{align}
Writing
\begin{align}
Y_1=\begin{pmatrix}t_a&t_b\\t_c&t_d\end{pmatrix}\text{ and }Y_1^{-1}=\begin{pmatrix}s_a&s_b\\s_c&s_d\end{pmatrix}
\end{align}
The $(1,2)$ submatrix of $Y_1^{-1}W_{11}Y_1$ is then $s_b(w_ct_b+w_dt_d)$.  From \eqref{h3} we see that
\begin{align}
t_b=\begin{pmatrix}*&0&0\\ *&0&0\end{pmatrix},\;t_d=\begin{pmatrix}*&0&0\\0&*&0\\0&0&*\end{pmatrix},\;s_b=\begin{pmatrix} *&0&0\\ *&0&0\end{pmatrix}.
\end{align}
Computing $s_b(w_ct_b+w_dt_d)$ we find that all entries vanish.

\end{proof}

\begin{proof}[Proof of Proposition \ref{h10z}]

\textbf{1. Integral equation for $\alpha_{21}$. }In order for $Y_2$ to conjugate solutions of \eqref{h5} to solutions of \eqref{h6a}, we must have
\begin{align}\label{p1}
hY_2'=\mathbb AY_2-Y_2\mathbb B\text{ on the wedge }\mathbb W=\mathbb{W}(M_0,\theta),
\end{align}
where $\mathbb A$ and $\mathbb B$ are the coefficient matrices of \eqref{h5} and \eqref{h6a} respectively. Direct computation shows that the functions $\alpha_{12}$, and $\alpha_{21}$ must therefore satisfy the equations
\begin{align}\label{p2}
\begin{split}
&hd_x\alpha_{12}=A^0_{11}\alpha_{12}-\alpha_{12}A_{22}^0+h(d_{11}\alpha_{12}-\alpha_{12}d_{22})+d_{12}-h^2\alpha_{12}d_{21}\alpha_{12}\\
&hd_x\alpha_{21}=A^0_{22}\alpha_{21}-\alpha_{21}A_{11}^0+h(d_{22}\alpha_{21}-\alpha_{21}d_{11})+d_{21}-h^2\alpha_{21}d_{12}\alpha_{21}.
\end{split}
\end{align}
Thinking of  the $3\times 2$ matrix $\alpha_{21}$ as an element of $\mathbb{C}^6$ and using obvious notation, we rewrite the second equation as
\begin{align}\label{p3}
hd_x\alpha_{21}=\left(\mathcal A(\zeta,h)+O(e^{-\mu \Re x})\right)\alpha_{21}+\left(d_{21}(\infty,\zeta)+O(e^{-\mu \Re x})\right)+O(h^2e^{-\mu \Re x})(\alpha_{21},\alpha_{21}),
\end{align}
where (with slight abuse)
\begin{align}\label{p4}
\mathcal A(\zeta,h)\alpha_{21}=A^0_{22}(\infty,\zeta)\alpha_{21}-\alpha_{21}A_{11}^0(\infty,\zeta)+h\left(d_{22}(\infty,\zeta)\alpha_{21}-\alpha_{21}d_{11}(\infty,\zeta)\right)
\end{align}
Here we have used \eqref{f44s} and the fact that $d_{12}(\infty,\zeta)=0$.

The eigenvalues of $A^0_{11}(\infty,\zeta)$ (resp., $A^0_{22}(\infty,\zeta)$) are $\mu_j(\infty,\zeta)$, $j=1,2$ (resp. $\mu_j(\infty,\zeta)$, $j=3,4,5$). The six eigenvalues $a_j(\zeta,h)$ of $\mathcal{A}(\zeta,h)$ are differences $\lambda_2-\lambda_1$, where $\lambda_j(\zeta,h)$ is an eigenvalue of $A^0_{jj}(\infty,\zeta)+hd_{jj}(\infty,\zeta)$.
From the formulas \eqref{t4} for the $\mu_j$ we see that there exist  constants $\ua$ and $h_0$ and a neighborhood
$\omega\ni\zeta_\infty$ such that all the eigenvalues $a_j(\zeta,h)$ satisfy
\begin{align}\label{p5}
\Im a_j(\zeta,h)>\ua>0\text{ for }0<h\leq h_0,\;\zeta\in\omega.
\end{align}
Given any $\eps_0>0$, after reducing $h_0$ and shrinking $\omega$ if necessary, we will also have for all $j$:
\begin{align}\label{p6}
|\Re a_j(\zeta,h)|<\eps_0 \text{ for }0<h\leq h_0,\;\zeta\in\omega.
\end{align}
In view of \eqref{p3}, we will construct $\alpha_{21}$ as a fixed point of the map (analyzed below)
\begin{align}\label{p7}
\begin{split}
&T\alpha_{21}(x)=\\
&\quad h^{-1}\int^x_{\infty_-}e^{h^{-1}\mathcal A(\zeta,h)(x-y)}[O(e^{-\mu \Re y})\alpha_{21}+\left(d_{21}(\infty,\zeta)+O(e^{-\mu \Re y})\right)+O(h^2e^{-\mu \Re y})(\alpha_{21},\alpha_{21})]\;dy,
\end{split}
\end{align}
where $x\in\mathbb{W}(M_0,\theta)$, $\infty_-$ is the point at $\infty$ on the lower boundary of the wedge $\mathbb{W}$, and the path of integration is a straight segment.

\textbf{2. Estimate of $e^{h^{-1}\mathcal A(\zeta,h)(x-y)}$. }Write $x=x_r+ix_i$ and let $y(s)=s+iy_i(s)$ be a parametrization of the segment from $\infty_-$ to $x$.  If $a(\zeta,h)=a_r+ia_{i}$ denotes any eigenvalue of $\mathcal A(\zeta,h)$, we have
\begin{align}\label{p8}
\Re (a(x-y(s)))=a_r(x_r-s)-a_i(x_i-y_i(s)):=a_r\Delta_r-a_i\Delta_i.
\end{align}
Choosing $\eps_0$ in \eqref{p6} such that $0<\eps_0<\ua\tan \theta$ and noting that $\left|\frac{\Delta_i}{\Delta_r}\right|\geq \tan\theta$, we estimate
\begin{align}\label{p9}
a_r\Delta_r-a_i\Delta_i\leq\eps_0|\Delta_r|-\ua|\Delta_i|\leq \eps_0|\Delta_r|-\ua |\Delta_r|\tan\theta=-\kappa|\Delta_r|, \text{ where }\kappa=\ua\tan\theta-\eps_0>0.
\end{align}
The Jordan form of the matrix $\cA(\zeta,h)$ can have nontrivial blocks, but the semisimplicity of the eigenvalues $\mu_j(\infty)$  $j=3,4,5$, of  $A^0_{22}(\infty,\zeta)$ implies that such a  block can be at most of size $4\times 4$.\footnote{This is because $\cA(\zeta,h)$   has at least 3 independent eigenvectors.}
Thus, \eqref{p8} and \eqref{p9} yield the estimate
\begin{align}\label{p10}
\left|e^{h^{-1}\mathcal A(\zeta,h)(x-y)}\right|\leq C\left(1+\frac{|x_r-s|^3}{h}\right)e^{-\kappa \frac{|x_r-s|}{h}}
\end{align}
on the path $y(s)$.

\textbf{3. Contraction. }Using the estimate \eqref{p10} we can choose $K>0$ such that
\begin{align}\label{p11}
\left|\quad h^{-1}\int^x_{\infty_-}e^{h^{-1}\mathcal A(\zeta,h)(x-y)}d_{21}(\infty,\zeta)dy\right|\leq K \text{ for }0<h\leq h_0,\;\zeta\in\omega,\; x\in\mathbb{W}(M_0,\theta).
\end{align}
In fact the integral in \eqref{p11} is independent of $x$, so we call it $D(\zeta,h)$.   For later use we note that for $x=x_r+ix_i\in \mathbb{W}(M_0,\theta)$,
\begin{align}\label{p12}
h^{-1}\int_\infty^{x_r}\left(1+\frac{|x_r-s|^3}{h}\right)e^{-\kappa \frac{|x_r-s|}{h}}e^{-\mu s}ds\leq Ce^{-\mu x_r}.
\end{align}
Denoting the set of analytic functions on $\mathbb{W}$ by $H(\mathbb{W})$, we let
\begin{align}\label{p13}
B=\{\alpha_{21}\in H(\mathbb{W}):|\alpha_{21}|_{L^\infty(\mathbb{W})}\leq K+1\}.
\end{align}
After increasing $M_0$ if necessary, we see that \eqref{p7}, \eqref{p10}, and \eqref{p12} imply that $T:B\to B$.  Using the same facts and again increasing $M_0$ if necessary, we see that $T$ gives a contraction on $B$.  So we now have a solution
$\alpha_{21}$ to \eqref{p3} satisfying
\begin{align}
|\alpha_{21}|_{L^\infty(\mathbb{W})}\leq K+1.
\end{align}
The contraction argument for $\alpha_{12}$ is  similar, and we leave it to the reader.

\textbf{4. Decay of $\alpha_{21}$ to its endstate. }Recall that $D(\zeta,h)$ is the ($x$-independent) integral in \eqref{p11}.  From \eqref{p7} and \eqref{p12} we see that
\begin{align}
|T\alpha_{21}(x)-D(\zeta,h)|=|\alpha_{21}(x)-D(\zeta,h)|\leq Ce^{-\mu x_r}\text{ for }x\in\mathbb{W}(M_0,\theta).
\end{align}
so $D(\zeta,h)=\alpha_{21}(\infty,\zeta,h)$.

\textbf{5. Derivative estimates. }The estimates \eqref{f45c} are obtained by differentiating \eqref{p7} and again applying \eqref{p12}.

\end{proof}

\section{Regimes I and II}\label{pfinfiniteII}

Regime II is the most difficult of the regimes to treat.  We give the proofs for this regime first; the proofs for Regime I are generally similar, but much simpler.
\subsection{Proofs for Regime II}

After establishing some notation, we give the proofs of Propositions \ref{f12}, \ref{f6}, \ref{f9}, and \ref{f9z}.

\subsubsection{The change of variable $\sigma\to \xi(\sigma)$.}An application of Rouch\'e's Theorem shows that for $|f_p|_{L^\infty(\cZ_{-i\tilde\alpha})}$ sufficiently small, the function $f(\sigma)=(1-\frac{1}{\sigma^2})+f_p(\sigma)$ has a unique simple zero $\sigma_0$ on $\cZ_{-i\tilde \alpha}$, and that $\sigma_0\to 1$ as $|f_p|_{L^\infty}\to 0$.  We set $f_0(\sigma):=1-\frac{1}{\sigma^2}$ and  introduce subscripts to distinguish
\begin{align}\label{f8}
\begin{split}
&\Xi_{f}(\sigma)=\frac{2}{3}\xi^{\frac{3}{2}}_{f}(\sigma)=\int^\sigma_{\sigma_0}\sqrt{f_0+f_p}\;(\sigma)d\sigma,\;\sigma\in\cZ_{cut}(\sigma_0)\text{ and }\\
&\Xi_{f_0}(\sigma)=\frac{2}{3}\xi^{\frac{3}{2}}_{f_0}(\sigma)=\int^\sigma_1 \sqrt{f_0}\;(\sigma)d\sigma, \;\;\sigma\in\cZ_{cut}(1).
\end{split}
\end{align}
Our analysis of $\xi_{f}(\sigma)$ is based on comparison with $\xi_{f_0}(\frac{\sigma}{\sigma_0})$, and for this we must carefully choose the branches of the square roots in \eqref{f8}.   Write
\begin{align}\label{f9a}
\begin{split}
& f_0(\sigma)=(\sigma-1)d_1(\sigma), \text{ where }d_1(\sigma)=\frac{\sigma+1}{\sigma^2},\text{ and }\\
& f(\sigma)=(\sigma-\sigma_0)d_{\sigma_0}(\sigma) \text{ where } d_{\sigma_0}(\sigma)=\frac{\sigma^2-1+\sigma^2f_p(\sigma)}{(\sigma-\sigma_0)\sigma^2}.
\end{split}
\end{align}
Now define
\begin{align}\label{f10}
\sqrt{f_0(\sigma)}=\sqrt{\sigma-1}\sqrt{d_1(\sigma)} \text{ on }\cZ_{cut}(1)
\end{align}
where $\sqrt{\sigma-1}$ is the branch on $\cZ_{cut}(1)$ that is positive for $\sigma>1$, and $\sqrt{d_1(\sigma)}=\frac{\sqrt{\sigma+1}}{\sigma}$, defined on $\Re \sigma>0$, is positive for $\sigma>0$\footnote{This definition yields the branch of $\sqrt{\sigma^2-1}$ used in \eqref{f3}.} .  Similarly, define
\begin{align}\label{f11}
\sqrt{f}(\sigma)=\sqrt{\sigma-\sigma_0}\sqrt{d_{\sigma_0}(\sigma)}\text{ on }\cZ_{cut}(\sigma_0),
\end{align}
where $\sqrt{\sigma-\sigma_0}$ on $\cZ_{cut}(\sigma_0)$ is given by
\begin{align}
\sqrt{\sigma-\sigma_0}=\sqrt{\sigma_0}\sqrt{\frac{\sigma}{\sigma_0}-1}\text{ with }\sqrt{\sigma_0}\sim 1\text{ and }\sqrt{\cdot -1}\text{ as above},
\end{align}
and $\sqrt{d_{\sigma_0}(\sigma)}$ is close to $\sqrt{d_1(\sigma)}$ for $f_p$ small.   Other powers of $\sigma-\sigma_0$ on
$\mathcal Z_{cut}(\sigma_0)$ are defined similarly.

\begin{proof}[Proof of Proposition \ref{f12}.]
\textbf{1. Analyticity on $\cZ_{-i\tilde\alpha}$.}  In the integral
\begin{align}\label{f13}
\frac{3}{2}\Xi_{f}(\sigma)=\frac{3}{2}\int^\sigma_{\sigma_0}\sqrt{s-\sigma_0}\sqrt{d_{\sigma_0}(s)}ds
\end{align}
we make the changes of variable $t=\sqrt{s-\sigma_0}$ and then $t=\sqrt{\sigma-\sigma_0}\;u$ to obtain
\begin{align}\label{f14}
\frac{3}{2}\Xi_{f}(\sigma)=\int_0^{\sqrt{\sigma-\sigma_0}}3t^2\sqrt{d_{\sigma_0}(t^2+\sigma_0)}\; dt=(\sigma-\sigma_0)^{\frac{3}{2}}\int^1_03u^2\sqrt{d_{\sigma_0}((\sigma-\sigma_0)u^2+\sigma_0)} \;du
\end{align}
Estimates given below (see step \textbf{3}) imply that the second integral in \eqref{f14} is nonvanishing on $\cZ_{-i\tilde\alpha}$ for $f_p$ sufficiently small.  Thus, this integral has a well-defined analytic logarithm, which we use to define roots of the integral.\footnote{The logarithm is chosen so that its argument is close to zero for $z$ large and positive.}  In particular, we obtain\footnote{Of course, we have a formula for $\xi_{f_0}(\sigma)$ similar to \eqref{f15} in which $\sigma_0$ is replaced by $1$.}
\begin{align}\label{f15}
\xi_{f}(\sigma)=(\sigma-\sigma_0) \left(\int^1_0 3u^2 \sqrt{d_{\sigma_0}((\sigma-\sigma_0)u^2+\sigma_0)} \;du\right)^{2/3}.
\end{align}
Using the above logarithm we define $\sqrt{\xi_{f}}$ on $\mathcal{Z}_{cut}(\sigma_0)$ and we have
\begin{align}\label{f15z}
\sqrt{\xi_{f}}\;\xi_{f}'=\sqrt{f}\text{ on }\mathcal{Z}_{cut}(\sigma_0).
\end{align}
From \eqref{f15} it is clear that $\xi_{f}$ is analytic  on $\cZ_{-i\tilde\alpha}$.  With \eqref{f15z} it follows that $\xi_{f}'$ is nonvanishing on $\cZ_{-i\tilde\alpha}$.  Thus $\xi_{f}$ is a locally one-to-one, conformal map of $\cZ_{-i\tilde\alpha}$ onto its range.

\textbf{2. Global injectivity. }For some sufficiently small $\delta>0$ and sufficiently large $K>0$ to be chosen, we divide $\cZ_{-i\tilde\alpha}$ into regions $A$, $B$, and $C$ where, respectively, $|\sigma|<\delta$, $\delta\leq |\sigma|\leq K$, and $|\sigma|> K$.  The first and main step is to prove injectivity of $\xi_{f}$ restricted to  each of these subsets.  The proof relies on the global injectivity of $\xi_{f_0}$, which follows from direct analysis of \eqref{f3}.

The next Lemma is essential for proving the injectivity and mapping properties of $\xi_{f}$.

\begin{lem}\label{f11a}
There exist constants $\eps_j=\eps_j(f_p)>0$, $j=A,B,C$, which approach zero as  
$N_p\to 0$,\footnote{$N_p$ occurs in Prop. \ref{d7}.}such that
\begin{align}\label{f11b}
\begin{split}
&(a)\;|\xi_{f}(\sigma)-\xi_{f_0}(\frac{\sigma}{\sigma_0})|\leq \eps_A/|\xi_{f_0}(\frac{\sigma}{\sigma_0})|^{\frac{1}{2}}, \text{ for }\sigma\in A\\
&(b)\;|\xi_{f}(\sigma)-\xi_{f_0}(\frac{\sigma}{\sigma_0})|\leq \eps_B, \text{ for }\sigma\in B\\
&(c)|\Xi_{f}(\sigma)-\Xi_{f_0}(\frac{\sigma}{\sigma_0})|\leq \eps_C|\Xi_{f_0}(\frac{\sigma}{\sigma_0})|, \text{ for }\sigma\in C.
\end{split}
\end{align}

\end{lem}

\begin{proof}

\textbf{. Estimates a and b. }We set $\uw=\uw(\sigma,u)=(\frac{\sigma}{\sigma_0}-1)u^2+1$ and write
\begin{align}\label{f11cc}
\xi_{f}(\sigma)-\xi_{f_0}(\frac{\sigma}{\sigma_0})=(\frac{\sigma}{\sigma_0}-1)\left[ \left(\sigma_0^{\frac{3}{2}}\int^1_0 3u^2\sqrt{d_{\sigma_0}(\sigma_0\uw)} du\right)^{2/3}-\left( \int^1_0 3u^2\sqrt{d_1(\uw)} du\right)^{2/3}\right].
\end{align}
Thus,  estimates \textbf{a} and \textbf{b} follow from the fact that given $\eps>0$, we have (for small $f_p$)
\begin{align}\label{f11c}
|\sigma_0^{3/2}\sqrt{d_{\sigma_0}(\sigma_0\uw)}-\sqrt{d_1(\uw)}|\leq\eps |\uw| \text{ for all }(u,\sigma)\in [0,1]\times (A\cup B).
\end{align}
To see this one computes (observing  important cancellations) the difference in \eqref{f11c} to be
\begin{align}\label{f11d}
\frac{\uw\left(\sigma_0^2-1+\sigma_0^2f_p(\sigma_0\uw)\right)}{\sqrt{\uw-1}\left(\sqrt{\sigma_0^2\uw-1+\sigma_0^2\uw^2f_p(\sigma_0\uw)}+\sqrt{\uw^2-1}\right)}.
\end{align}
 For $\uw$ bounded away from $0$ we can factor $\uw-1$ out of numerator and denominator (since $(f_0+f_p)(\sigma_0)=0$) to obtain \eqref{f11c} when $(f_p,f_p')$ is small.
 When $\uw$ is near $0$, we obtain \eqref{f11c} since
 \begin{align}
|\sigma_0^2-1+\sigma_0^2f_p(\sigma_0\uw)|\leq \eps\text{ when }f_p\text{ is small}.
\end{align}

We remark that the individual integrals in \eqref{f11cc} do blow up since $\uw\to 0$ as $(\sigma,u)\to (0,1)$. In fact it is clear from \eqref{f3} that
\begin{align}
|\xi_{f_0}(\sigma)|\sim C (|\ln |\sigma||)^{2/3} \text{ as }\sigma\to 0.
\end{align}

\textbf{Estimate c. }We use again the formula \eqref{f11d}.  Since $|\uw|\sim |\sigma|$ for $|\sigma|$ large and
$|(\sigma_0^2-1+\sigma_0^2f_p(\sigma_0\uw)|\leq \eps$ for $f_p$ small, we have $|\eqref{f11d}|\leq \eps/|\sigma|^{\frac{1}{2}}$ for $|\sigma|$ large. The estimate follows since $|\Xi_{f_0}(\frac{\sigma}{\sigma_0})|\sim |\sigma|$ for $|\sigma|$ large.
\end{proof}

\textbf{3. }The nonvanishing of the second integral in \eqref{f14} for small $f_p$ can be deduced from the above estimates of \eqref{f11d} together with the nonvanishing of the (computable) integral
\begin{align}
\int^1_0 3u^2\sqrt{d_1(\uw)}du.
\end{align}

\textbf{4. Region $B$. }Writing
\begin{align}
\xi_{f}(\sigma)-\xi_{f}(a)=(\sigma-a)\left(\xi'_{f}(a)+(\sigma-a)h(\sigma,a)\right)  \text{ for }\sigma,a\in B,
\end{align}
and noting that there exist positive constants $C_1$, $C_2$ such that $|\xi_{f}'(\sigma)|\geq C_1$ on $B$ and $|h(\sigma,a)|\leq C_2$ on $B\times B$, we see that there exists $\kappa>0$ such that
\begin{align}\label{f16}
\xi_{f}(\sigma)\neq \xi_{f}(a) \text{ for }(\sigma,a)\in B\times B, \sigma\neq a, |\sigma-a|\leq \kappa.
\end{align}

Since region $B$ is compact, $\xi_{f_0}$ is injective on $B$, and $\xi_{f_0}'$ is everywhere nonvanishing, there exists a constant $C>0$ such that
   \begin{align}
   |\xi_{f_0}(\frac{\sigma_1}{\sigma_0})-\xi_{f_0}(\frac{\sigma_2}{\sigma_0})|\geq C |\sigma_1-\sigma_2| \text{ for all }\sigma_1,\sigma_2\in B.
\end{align}
Suppose now that for all $\mu>0$ injectivity of $\xi_{f}|_B$ fails for some perturbation $f_p$ with $|f_p,f'_p|_{L^\infty(B)}<\mu$. Then estimate \eqref{f11b}(b) implies that there exists a sequence of perturbations $f_{p,k}$, sequences of points $\sigma_{1,k}$, $\sigma_{2,k}$ in $B$, and a sequence of positive constants $\eps_{B,k}\to 0$ such that
\begin{align}
0=|\xi_{f_0+f_{p,k}}(\sigma_{1,k})-\xi_{f_0+f_{p,k}}(\sigma_{2,k})|\geq |\xi_{f_0}(\frac{\sigma_{1,k}}{\sigma_0})-\xi_{f_0}(\frac{\sigma_{2,k}}{\sigma_0})|-\eps_{B,k}\geq C|\sigma_{1,k}-\sigma_{2,k}|-\eps_{B,k}.
\end{align}
So $|\sigma_{1,k}-\sigma_{2,k}|\to 0$, which contradicts \eqref{f16}.

\textbf{5. Region A. }For $\sigma_1, \sigma_2\in A$ we write
\begin{align}\label{f17}
\xi_{f}(\sigma_1)-\xi_{f}(\sigma_2)=(\sigma_1-\sigma_2)\int^1_0\xi'_{f}(\sigma_2+s(\sigma_1-\sigma_2))ds.
\end{align}
We claim that for $\delta$ as in step \textbf{2} small enough, the integral on the right in \eqref{f17} is nonvanishing (and in fact very large) for small $f_p$.   Using the explicit form of $f_0$ one computes directly (e.g., using partial fractions) that the dominant contribution to
$\xi_{f_0}'$ for $\sigma\in A$ is a term of the form
\begin{align}\label{f18a}
\frac{C}{\sigma (\log  \sigma)^{1/3}}.
\end{align}
Since $\arg \sigma\sim 0$ for $\sigma\in A$, we deduce\footnote{This can also be derived from the explicit formula \eqref{f3}.}
\begin{align}\label{f18}
\left|\int^1_0 \xi_{f_0}'\left(\frac{\sigma_2+s(\sigma_1-\sigma_2))}{\sigma_0}\right)ds\right|\geq C \frac{1}{|\sigma_1||\ln|\sigma_1||^{1/3}}, \text{ where }|\sigma_1|\geq|\sigma_2|, \;\sigma_1,\sigma_2\in A.
\end{align}
To see that \eqref{f18} holds with $\xi_{f}$ in place of $\xi_{f_0}$, we use the estimate
\begin{align}\label{f18ak}
\left|\xi'_{f}(\sigma)-\frac{1}{\sigma_0}\xi_{f_0}'(\frac{\sigma}{\sigma_0})\right| \leq \frac{\eps}{|\sigma||\ln|\sigma||^{1/3}},
\end{align}
where $\eps\to 0$ as $N_p\to 0$.
To prove this we write out the derivatives in \eqref{f18a} explicitly, forming two differences $A_1-A_2$ and $B_1-B_2$ in the obvious way, and use Lemma \ref{f11a}(a) to estimate
\begin{align}\label{f18aa}
|A_1-A_2|:=|(\int^1_0 3u^2\sqrt{d_{\sigma_0}}du)^{2/3}-(\int^1_0 3u^2\sqrt{d_{1}}du)^{2/3}|\leq \frac{\eps}{|\ln|\sigma||^{1/3}}.
\end{align}
Here and below $d_{\sigma_0}$ is evaluated at $\sigma_0\uw$, and $d_1$ is evaluated at $\uw$.   Since $\sigma-\sigma_0\sim \sigma_0\sim 1$ we have
$|B_1-B_2|\leq$
\begin{align}\label{f18b}
 \begin{split}
&C \left|\frac{1}{\sigma_0^2}\left(\int^1_0 u^2\sqrt{d_{1}}du\right)^{-\frac{1}{3}}\left(\int^1_0 u^4\frac{d'_{1}}{\sqrt{d_{1}}}du\right)-\left(\int^1_0 u^2\sqrt{d_{\sigma_0}}du\right)^{-\frac{1}{3}}\left(\int^1_0 u^4\frac{d'_{\sigma_0}}{\sqrt{\sigma_{z_0}}}du\right)\right|\lesssim\\
 &\quad\left|\sqrt{\sigma_0}\left(\int^1_0 u^2\sqrt{d_{1}}du\right)^{-\frac{1}{3}}-\left(\int^1_0 u^2\sqrt{d_{\sigma_0}}du\right)^{-\frac{1}{3}}\right|\left|\left(\frac{1}{\sigma_0^{5/2}}\int^1_0 u^4\frac{d'_{1}}{\sqrt{d_{1}}}du\right)\right|+\\
 &\quad\quad\left|\left(\int^1_0 u^2\sqrt{d_{\sigma_0}}du\right)^{-\frac{1}{3}}\right|\left|\left(\frac{1}{\sigma_0^{5/2}}\int^1_0 u^4\frac{d'_{1}}{\sqrt{d_{1}}}du\right)-\left(\int^1_0 u^4\frac{d'_{\sigma_0}}{\sqrt{d_{\sigma_0}}}du\right)\right|
\end{split}
\end{align}
By \eqref{f11c} and the computation that produced \eqref{f18a} we see that the second line of \eqref{f18b} is dominated by the right side of \eqref{f18ak}.

Next we show that the third line of \eqref{f18b} is dominated by the right side of \eqref{f18ak}.  We write
\begin{align}\label{f18c}
\frac{d_{\sigma_0}'}{\sqrt{d_{\sigma_0}}}-\frac{d_{1}'}{\sqrt{d_1}\sigma_0^{5/2}}=\left(d_{\sigma_0}'-\frac{d_1'}{\sigma_0^4}\right)\frac{1}{\sqrt{d_{\sigma_0}}}+\frac{d_1'}{\sigma_0^4}\left(\frac{1}{\sqrt{d_{\sigma_0}}}-\frac{\sigma_0^{3/2}}{\sqrt{d_1}}\right):=C+D.
\end{align}
Using \eqref{f11c} we obtain by the computation that produced \eqref{f18a} that the contribution from the term involving $D$ to the third line of \eqref{f18b} is dominated by the right side of \eqref{f18ak}.  To estimate the contribution from $C$ we write after observing some cancellations
\begin{align}\label{f18p}
C=\frac{(1-\sigma_0^2-\sigma_0^2f_p(\sigma_0\uw))+(\uw-1)\sigma_0^3f_p'(\sigma_0\uw)}{(\uw-1)^2\sigma_0^4}\cdot \frac{\sigma_0^{3/2}\sqrt{\uw-1}\;\uw}{\sqrt{\sigma_0^2\uw^2-1+\sigma_0^2\uw^2f_p(\sigma_0\uw)}}
\end{align}
We claim $|C|\leq \eps$ for $(u,\sigma)\in [0,1]\times A$ when $N_p$ is small, and thus  the contribution from the term involving $C$ to the third line of \eqref{f18b} is dominated by the right side of \eqref{f18aa}.  To estimate $C$ we note that when $\uw$ is bounded away from $0$,  $(\uw-1)^2$ can be factored out of the first factor in \eqref{f18p} yielding
\begin{align}\label{f19}
|d_{\sigma_0}'(\sigma_0\uw)-\frac{1}{\sigma_0^4}d'_1(\uw)|=\left|\frac{1}{2}\int^1_0(1-s)f_p''(\sigma_0+s(\sigma_0\uw-\sigma_0))ds\right|
\leq \eps.
\end{align}
The second factor is treated similarly.  Here we use that $f_p'(\sigma_0\uw)$ and $f_p''(\sigma_0\uw)$ are both small for $\uw$ bounded away from $0$ when $N_p$ is small (recall Proposition \ref{estimates}).    For $\uw$  near $0$ the smallness of $C$ follows from the smallness of $f_p'(\sigma_0\uw)\uw$.

\textbf{6. Region C. }For $\sigma\in \cZ_{-i\tilde\alpha}$ with $|\sigma|$  large the correspondence $\Xi_{f}\leftrightarrow \xi_{f}$ is one-to-one, so it suffices to show $\Xi_{f}$ is one-to-one on region $C$.

Choose $0<\kappa<1$ and for $\sigma,a\in C$ with $|\sigma-a|\leq \kappa |\sigma|$ write
\begin{align}\label{f22}
\Xi_{f}(\sigma)-\Xi_{f}(a)=(\sigma-a)\left[ \Xi'_{f}(a)+\frac{(\sigma-a)}{2}\int^1_0(1-s)\Xi''_{f}(a+s(\sigma-a))ds\right].
\end{align}
For $|\sigma|$ large we claim
\begin{align}\label{f23}
|\Xi''_{f}(\sigma)|\leq \eps/|\sigma|,
\end{align}
where $\eps\to 0$ as $N_p\to 0$.  Since $|\Xi'_{f}(a)|\sim 1$ for $|a|$ large, the modulus of the right side of \eqref{f23} is $\geq \frac{1}{2}|\sigma-a|$ for $\eps$ small enough.   The estimate \eqref{f23} follows by direct computation after noting $|f_p'(\sigma)|\leq \eps/|\sigma|$ for $|\sigma|$ large.  For example, using Proposition \ref{estimates} we estimate the term
\begin{align}\label{f24}
|-\tilde\alpha^2 \sigma h_2'(-i\tilde\alpha\sigma,\zeta)|\leq C|\tilde \alpha| |\tilde\alpha \sigma|\leq \eps/|\sigma| \text{ since }|\tilde\alpha \sigma|\leq 2\eps_2.
\end{align}

     The formula \eqref{f3} shows that there exists $m>0$ such that
     \begin{align}\label{f20}
     \left| \Xi_{f_0}(\frac{\sigma}{\sigma_0})-\Xi_{f_0}(\frac{a}{\sigma_0})\right|\geq m |\sigma-a| \text{ for }\sigma,a\in C.
     \end{align}
For $|\sigma-a|> \kappa|\sigma|$ ($\sigma,a\in C$) use estimate \eqref{f11b}(c) to write
\begin{align}\label{f21}
|\Xi_{f}(\sigma)-\Xi_{f}(a)|\geq \left| \Xi_{f_0}(\frac{\sigma}{\sigma_0})-\Xi_{f_0}(\frac{a}{\sigma_0})\right|-\eps_C\left| \Xi_{f_0}(\frac{\sigma}{\sigma_0})\right|-\eps_C\left|\Xi_{f_0}(\frac{a}{\sigma_0})\right|\geq \frac{m}{2}|\sigma-a|
\end{align}
for $\eps_C$ small enough.

\textbf{7. Adjacent regions. }Recall that the regions $A,B,C$ are determined by the choice of parameters $\delta$ and $K$. The above arguments show that there exist $\delta_0$, $K_0$ such that for $\delta<\delta_0$ and $K>K_0$, $\xi_{f}$ is injective (for $N_p$ small) on each of the regions $A,B,C$ determined by the choice $(\delta,K)$. It is immediate from \eqref{f3} and the estimates \eqref{f11b} that $\xi_{f}(\sigma_1)\neq\xi_{f}(\sigma_2)$ for $\sigma_1\in A$, $\sigma_2\in C$, so it remains to consider $\sigma_j$ in adjacent regions.

   Choose positive constants $\delta_j$ and $K_j$, $j=a,b$, such that $\delta_b<\delta_a<\delta_0$ and $K_b>K_a>K_0$, and which have the following additional property.    There exists $M>0$ such that if $A_j,B_j,C_j$ are the regions determined by the choice $(\delta_j,K_j)$, we have
   \begin{align}\label{f25}
|\xi_{f}(\sigma)|\leq M\text{ for }\sigma\in B_a \text{ and } |\xi_{f}(\sigma)|> M\text{ for }\sigma\in A_b\cup C_b.
\end{align}

Suppose now that $\sigma_1\in A_a$, $\sigma_2\in B_a$.  Considering the two cases $\sigma_1\in A_b$, $\sigma_1\in B_b$ and using  \eqref{f25}
and the above results for single regions, we conclude $\xi_{f}(\sigma_1)\neq \xi_{f}(\sigma_2)$. The case $\sigma_1\in B_a$, $\sigma_2\in C_a$ is treated similarly.
\end{proof}

\begin{proof}[Proof of Proposition \ref{f6}.]
In order to apply Theorem 9.1 of \cite{O}, Chapter 11, we must choose a suitable open subdomain of the $\xi-$plane on which to solve the equation \eqref{f26} below.  There are three requirements:

a)\;The domain should include the image of an interval $[M,\infty)$ under the map $x\to \xi$ (here, $x\in T_{M,R}$ as in \eqref{a1}), where $M$ can chosen independent of the parameters $(\zeta,h)$.

b)\;It must be possible to choose ``progressive paths" (defined below) for all points in the domain.

c)\;The integrals \eqref{f5} should all be finite; more precisely, there should be a finite upper bound independent of the choice of path and of relevant parameters such as $\zeta$ and $h$.

\textbf{1. Choice of $\xi-$domain. }Let $v(\sigma)$ be a solution to the perturbed Bessel problem \eqref{d3} on the dilated wedge $\cZ_{-i\tilde\alpha}$.  In the new variables $\xi=\xi_{f}(\sigma)$, $v=\left(\frac{d\xi}{d\sigma}\right)^{-1/2}W$
the problem \eqref{d3} becomes
\begin{align}\label{f26}
W_{\xi\xi}=(\tilde\gamma^2\xi+\psi(\xi))W.
\end{align}
We are not able to solve \eqref{f26} on the full open set $\xi_{f}(\cZ_{-i\tilde\alpha})$, because of problems choosing progressive paths created by the perturbation $f_p$.  Instead, we explain how to choose a subdomain where such paths can be chosen, and which also contains the image of the segment of the $x-$axis, $[M,\infty)$, under the map $x\to \xi$.   At first we ignore the right boundary segment of $\cZ_{-i\tilde\alpha}$ and treat the wedge as if it were infinite.

  We begin by specifying  a domain in the $\Xi-$plane. For small positive constants $\kappa$, $\eps$ both less than $1$, let $\Delta_{\Xi}(\kappa,\eps)=A_\Xi\cup B_\Xi$ where $A_\Xi$ and $B_\Xi$ are the open subsets of $\bC$ defined as follows.  $A_\Xi$ is the connected open set bounded by the parametrized segments
  \begin{align}
  \{it:t\geq 0\}, \{t:t\geq 0\}, \{t-i(\kappa t+\eps):t\geq -\eps\}, \{-\eps+it:t\geq \kappa\eps-\eps\},
  \end{align}
while $B_\Xi$ is the connected open set bounded by the segments
 \begin{align}
  \{it:t\geq 0\}, \{t:t\geq 0\}, \{t+i(\kappa t+\eps):t\geq -\eps\}, \{-\eps+it:t\leq -\kappa\eps+\eps\},
  \end{align}
Next let $\Delta_\xi(\kappa,\eps)=A_\xi\cup B_\xi\cup \bR$, where $A_\xi$ is the image of $A_{\Xi}$ under the map $\Xi\to \xi$, where the branch of the $2/3$ root is defined by taking $-\frac{3\pi}{2}<\arg \Xi<0$ for $\Xi\in A_\Xi$, and
$0<\arg \Xi<\frac{3\pi}{2}$ for $\Xi\in B_\Xi$. Observe that $\Delta_\xi(\kappa,\eps)$ is an open neighborhood of the real axis whose intersection with $\Re\xi=t$ has width $\sim t^{2/3}$ for $t>0$ large, and width $\sim |t|^{-1/2}$ for $t<0$, $|t|$ large.

 It follows from the formula \eqref{f3}  that the image of $\cZ_{-i\tilde\alpha}$ under $\sigma\to\xi_{f_0}(\frac{\sigma}{\sigma_0})$ contains a subdomain of the form $\Delta_\xi(\kappa,\eps)$ for some choice of $\kappa$, $\eps$.  This is because the image of $\cZ_{cut}(\sigma_0)$ under $\sigma\to \Xi_{f_0}(\frac{\sigma}{\sigma_0})$ contains a set of the form $\Delta_\Xi(\kappa,\eps)$.\footnote{We have $\Xi_{f_0}(\sigma)\sim \sigma$ for $\sigma>0$, $|\sigma|$ large.}
 By Proposition \ref{f12} and the estimates of Lemma \ref{f11a} we deduce, after further reduction of $N_p$ if necessary,  that $\xi_{f}(\cZ_{-i\tilde\alpha})$ must also contain a subdomain of the form $\Delta_\xi(\kappa,\eps)$, for some  smaller $\kappa$ and $\eps$.

Finally, we recall that the dilated wedge $\mathcal{Z}_{-i\tilde\alpha}=\mathcal{W}/(-i\tilde\alpha)$ has a right boundary arc of radius $\eps_2/|\tilde\alpha|>>1$, where $\eps_2$ is the radius of the right boundary arc of $\mathcal{W}$ (Definition \ref{defW}).  Thus, we define $\Delta_{\Xi}(\kappa,\eps,\eps_2)$ to be the bounded open set obtained by cutting off $\Delta_{\Xi}(\kappa,\eps)$ with this boundary arc.  With $\Delta_\xi(\kappa,\eps,\eps_2)$ the corresponding $\xi$ domain, we can repeat the procedure of the previous paragraph to deduce that $\xi_{f}(\mathcal{Z}_{-i\tilde\alpha})$ contains a subdomain of the form $\Delta_\xi(\kappa,\eps,\eps_2')$ for some $\eps'_2<\eps_2$ \footnote{Using estimate \eqref{f11b}(c) we can take $\eps'_2=(1-\eps_C)\eps_2$.}.

We may now define the subdomain $\cZ_{-i\tilde\alpha,s}$ appearing in the statement of Proposition \ref{f6} as
\begin{align}\label{f27}
\cZ_{-i\tilde\alpha,s}:=\xi_{f}^{-1}\left(\Delta_\xi(\kappa,\eps,\eps_2')\right).
\end{align}
This domain contains the image of $[M',\infty)$ under the map $x\to \sigma$, where $M'$ is slightly greater than $M$ (we have
$M'=M+O(|\ln(1-\eps_C)|)).$

\textbf{2. Choice of progressive paths. }Define the sectors $\mathbf{S_0}$, $\mathbf{S_1}$, and $\mathbf{S_{-1}}$ by
$|\arg \sigma|\leq \frac{\pi}{3}$, $\frac{\pi}{3}\leq \arg \sigma\leq \pi$, and $-\pi \leq \arg \sigma\leq -\frac{\pi}{3}$, respectively.
Let $\omega=\arg \tilde\gamma$ (recall $\tilde\gamma=-i\tilde \beta$ for $\tilde \beta$ as in \eqref{hr3}).
From the definition of Regime II we have for some small $\delta>0$,
\begin{align}\label{f27a}
-\delta\leq \arg\tilde\gamma\leq 0.
\end{align}


\begin{defn}\label{f27aa}
 Let $\Delta\subset\bC$
be a connected open set,
let $\partial\Delta$ denote its boundary, and take $j\in\{0,1,-1\}$. We say that progressive $j-$paths can be chosen in $\Delta$ provided there exists a point $\alpha_j\in \partial\Delta\cap e^{-2i\omega/3}\mathbf{S_j}$, possibly at infinity, such that for all $\xi\in\Delta$ there is a path $\mathcal P_j$ from $\xi$ to $\alpha_j$ in $\Delta$ with the properties:

a)\;As $v$ traverses $\mathcal P_j$ from $\xi$ to $\alpha_j$, the real part of $(\tilde\gamma^{2/3} v)^{3/2}$ is nondecreasing. The branch of $(\tilde\gamma^{2/3} v)^{3/2}$ is chosen so that $\Re(\tilde\gamma^{2/3} v)^{3/2}\geq 0$ in $e^{-2i\omega/3}\mathbf{S_j}$, and so that this real part is $\leq 0$ in $e^{-2i\omega/3}\mathbf{S_k}$, $k\neq j$.

b)\;The path $\mathcal P_j$ has a parametrization $v(\tau)$ such that $v''$ is continuous and $v'$ always nonvanishing, or consists of a finite chain of such paths.\footnote{This definition corrects an ambiguity in the definition given in section 9.1 of \cite{O}, Chapter 11.}

 \end{defn}

\begin{rem}\label{f27b}
For example, in the case $j=1$ the correct branch of $(\tilde\gamma^{2/3} v)^{3/2}$ is the one for which
\begin{align}
-\frac{\pi}{3}+2\pi-\frac{2\omega}{3}\leq \arg v\leq \frac{5\pi}{3}+2\pi-\frac{2\omega}{3}.
\end{align}
The condition in part (a) of the definition is linked to the choice of weight functions $E_j(z)$ defined in section 8.3 of \cite{O}, Chapter 11.   The definition of $E_j$ and $\mathbf{S_j}$ reflects the fact that $Ai_j(z)$ is recessive in $\mathbf{S_j}$ and dominant in $\mathbf{S_k}$, $k\neq j$.
\end{rem}

For a given $j$ it is easy to draw level curves of the correct branch of $\Re(\tilde\gamma^{2/3} \xi)^{3/2}$. A picture in the case $j=0$, $\arg\tilde  \gamma=0$ is given in Figure 9.1 of \cite{O}, Chapter 11.  Aided by such a picture together with
the explicit description of the (drawable) region $\Delta=\Delta_{\xi}(\kappa,\eps,\eps'_2)$ given in step \textbf{1}, ones sees  that progressive $j-$paths can be chosen in $\Delta$ for $j=0,1,-1$.
The point $\alpha_j\in \partial\Delta\cap e^{-2i\omega/3}\mathbf{S_j}$ is chosen to be a point where $\Re(\tilde\gamma^{2/3} \xi)^{3/2}> 0$ is maximized on $\partial\Delta\cap e^{-2i\omega/3}\mathbf{S_j}$.   Depending on the value of $\omega=\arg\tilde\gamma$, the point $\alpha_1$ or $\alpha_{-1}$ may need to be taken at infinity.

\textbf{3. Finiteness of the integrals $
\int_{\alpha_j}^\xi |\psi(s)s^{-1/2}|\;d|s|$. }
By ``finiteness" we mean here a finite bound that can be taken independent of the choice of $j-$progressive path and of the parameters $\eps_1$, $\eps_2$, $\zeta$ and $h$ appearing in the definitions of $\tilde\gamma=\tilde\gamma(\zeta,h)$ and $f_p(\sigma)=f_p(\sigma,\eps_1,\eps_2,\zeta,h)$ (recall Prop. \ref{d7}).  Here $\zeta\in\omega_\infty$, a neighborhood of $\zeta_\infty$, and $0\leq h\leq h_0$, where $\eps_2$, $\omega_\infty$ and $h_0$ were chosen in step \textbf{1} above and in the proof of Proposition \ref{f12} to make $f_p$ sufficiently small; moreover, $|\tilde\gamma|\geq K_1$  (Regime II).

  Clearly, for a given fixed $N>0$ we need only check the finiteness when at least one of $|\alpha_j|$, $|\xi|$ is $\geq N$.
Observe that $\Delta$ is unbounded on the left ($\Re\xi<0$) and, although $\Delta$ is bounded on the right for fixed $h$, there are choices of $(\zeta,h)$ in regime II for which the right boundary moves to infinity as $h\to 0$. Thus, $\Delta$ is effectively unbounded in both directions.

With  $f=f_0+f_p$ for $f_0(\sigma)=\frac{\sigma^2-1}{\sigma^2}$ and $f_p$ as in \eqref{d2}, function $\psi(\xi)$ in \eqref{c10} may be rewritten
\begin{align}\label{f28}
\psi(\xi)=\frac{5}{16\xi^2}+[4f(\sigma)f''(\sigma)-5f'(\sigma)^2]\frac{\xi}{16f^3(\sigma)}+\frac{\xi g(\sigma)}{f(\sigma)},\text{ where }g(\sigma)=-\frac{1}{4\sigma^2}.
\end{align}
Here and in the remainder of this step $\xi=\xi_f$.
Letting $\psi_0(\xi)$ denote the function obtained  by replacing $f$ by $f_0$ in \eqref{f28}, but leaving $\xi=\xi_f$,\footnote{The definition of $\xi_f$ involves $f_p$.}we compute
\begin{align}\label{f29}
\psi_0(\xi)=\frac{5}{16\xi^2}-\frac{\xi \sigma^2 (\sigma^2+4)}{4(\sigma^2-1)^3},
\end{align}
 observing an important cancellation.
We can write
\begin{align}\label{f31}
\psi(\xi)=\frac{5}{16\xi^2}+[4(f_0+f_p)(f_0''+f_p'')-5(f_0'+f_p')^2]\frac{\xi}{16(f_0+f_p)^3}+\frac{\xi g(z)}{f_0+f_p}=\psi_0(\xi)+\psi_1(\xi),
\end{align}
which defines $\psi_1$.  First we check the finiteness of the integral
\begin{align}\label{f32}
\int_{\alpha_j}^\xi |\psi_k(s)s^{-1/2}|\;d|s|
\end{align}
when $k=0$.

Observe that when $|\sigma|$ is small or large, we have $|\xi|$ large with $\Re\xi<0$ or $>0$ respectively, and $|f_p|/|f_0|<<1$.
For $|\sigma|$ large by \eqref{f3} we have $\sigma^2\sim \frac{4}{9}\xi^3$, so \eqref{f29} implies
$\psi_0(\xi)\sim \frac{-1}{4\xi^2}$.   For $|\sigma|$ small we have $|\xi|$ large   and \eqref{f3} implies
\begin{align}\label{f33}
\sigma\sim 2 \exp(-\frac{2}{3}|\xi|^{3/2}-1).
\end{align}
In this case \eqref{f29} implies $\psi_0(\xi)\sim \frac{5}{16\xi^2}$, so the finiteness is again clear.

Next consider \eqref{f32} when $k=1$.   First we write
\begin{align}\label{f34}
\frac{\xi}{16(f_0+f_p)^3}\sim \frac{\xi}{16f_0^3}(1-3\frac{f_p}{f_0})\text{ and }\frac{\xi g}{f_0+f_p}\sim \frac{\xi g}{f_0}(1-\frac{f_p}{f_0}).
\end{align}
Now we can read off the (largest) terms appearing in $\psi_1$ and estimate them one by one.  For example, the terms involving second derivatives are (ignoring some constant factors)
\begin{align}\label{f35}
f_p(f_0''+f_p'')(\frac{\xi}{f_0^3}-3\frac{\xi f_p}{f_0^4}),\;\;\;\;
f_0f''_p(\frac{\xi}{f_0^3}-3\frac{\xi f_p}{f_0^4}),\;\;\;\;
f_0f_0''\frac{\xi f_p}{f_0^4}
\end{align}
Now $f_p$ is given by \eqref{d4}, so we can list the terms appearing in $f_p''$ (ignoring some constant factors):
\begin{align}\label{f36}
\begin{split}
&(\alpha^2-\tilde\alpha^2\sigma^2)\tilde\alpha^2 h_1'',\;\;\sigma\tilde\alpha^3h_1',\;\;\tilde\alpha^2h_1\\
&\sigma\tilde\alpha^3 h_2'',\;\;\tilde\alpha^2 h_2'\\
&\tilde\alpha^2 h h_3'',
\end{split}
\end{align}
where the $h_j$ derivatives are $d/dt$ derivatives ($t$ as in Prop. \ref{estimates}).
Using Proposition \ref{estimates} we obtain
\begin{align}\label{f37}
\begin{split}
&h_1'=O(|\tilde\alpha \sigma|),\; h_1''=O(1)\\
&h_2'=O(1),\; h_2''=O(|\tilde\alpha \sigma|^{-1})\\
&h_3'=O(|\tilde\alpha \sigma|),\; h_3''=O(\frac{1}{h})
\end{split}
\end{align}
and  recall that we have for $\sigma\in \mathcal Z_{-i\tilde\alpha}\supset \xi_{f}^{-1}(\Delta)$:
\begin{align}\label{f38}
|\tilde\alpha \sigma|\leq \eps_2 \text{ (for }\eps_2 \text{ as in Definition }\ref{defW}).
\end{align}

We now estimate two typical terms from \eqref{f35}:
\begin{align}\label{f39}
|f_0f''\xi\frac{f_p}{f_0^4}|=6|\frac{\xi \sigma^2 f_p}{(\sigma^2-1)^3}|
\end{align}
When $|\sigma|$ is large, the right side of \eqref{f39} is $\leq C|\xi|/|\sigma|^4$, so the finiteness of the corresponding terms in \eqref{f32} is clear from $|\sigma|^2\sim \frac{4}{9}|\xi|^3$.   When $|\sigma|$ is small the finiteness follows from \eqref{f33}.

Next consider one of the ``worst" terms appearing in $\frac{\xi f_p''}{f_0^2}$, namely the one corresponding to the term
$\sigma\tilde\alpha^3h_2''$ from \eqref{f36}.   When $|\sigma|$ is large we have, using \eqref{f37},
\begin{align}\label{f40}
|\xi\frac{\sigma^4}{(\sigma^2-1)^2} \sigma\tilde\alpha^3 h_2''|\leq C|\xi \sigma \tilde \alpha^3 (\tilde \alpha \sigma)^{-1}|=C|\xi||\tilde\alpha|^{2}\leq C\frac{|\xi|}{|\sigma|^2}\leq \frac{C}{|\xi|^2},
\end{align}
since $|\tilde\alpha|\leq \eps_2/|\sigma|$.  This gives the finiteness of the corresponding term in \eqref{f32} at right infinity.
When $|\sigma|$ is small, we write
\begin{align}\label{f41}
|\xi\frac{\sigma^4}{(\sigma^2-1)^2} \sigma\tilde\alpha^3 h_2''|\leq C|\xi||\sigma^5||\tilde\alpha^3||\tilde\alpha \sigma|^{-1}= C|\xi||\tilde\alpha|^{2}|\sigma|^{4}
\end{align}
so the finiteness at left infinity follows from \eqref{f33}.

Next consider the term in $\frac{\xi f_p''}{f_0^2}$ corresponding to the term $\tilde\alpha^2 h h_3''$ in \eqref{f36}.  When $|\sigma|$ is large we have
\begin{align}
|\xi\frac{\sigma^4}{(\sigma^2-1)^2}\tilde\alpha^2 hh_3''|\leq C|\xi \tilde\alpha^2 h \frac{1}{h}|\leq C|\xi|/|\sigma|^2\leq C/|\xi|^2.
\end{align}
When $|\sigma|$ is small,
\begin{align}
|\xi\frac{\sigma^4}{(\sigma^2-1)^2}\tilde\alpha^2 hh_3''|\leq C|\xi \sigma^4\tilde\alpha^2 h \frac{1}{h}|,
\end{align}
so finiteness at left infinity follows again from \eqref{f33}.

The estimates corresponding to the remaining terms in $\psi_1$ are entirely similar to those above.

\textbf{4. Conclusion. }We have now checked that all the requirements for an application of Theorem 9.1 of \cite{O}, Chapter 11 are satisfied, so this concludes the proof of Proposition \ref{f6}.

\end{proof}

Proposition \ref{f9} describes the decaying solutions of \eqref{e4yy} on $[M,\infty)$.  In the proof we will of course use the fact that $Ai_{\pm 1}(z)$ is recessive in the sector $\mathbf{S}_{\pm 1}$.

 \begin{proof}[Proof of Proposition \ref{f9}]
 The explicit formulas for $\tilde\beta$ and $\tilde\gamma=-i\tilde\beta$ show that for $\tilde \beta$ in Regime II, we have
 \begin{align}
 \arg\tilde\gamma\leq 0 \text{ and }\arg\tilde \gamma =0\Leftrightarrow \Re\zeta =0.
 \end{align}
  The image of $[M,\infty)$ under the map $x\to \xi$ is a curve that approaches left infinity in $\Delta_\xi$ \eqref{f27} as $x\to \infty$.  The image of $[M,\infty)$ under $x\to \tilde\gamma ^{2/3}\xi$ thus lies for $x$ large enough in the interior of $\mathbf{S}_{1}$ when $\arg\tilde\gamma<0$.  So Proposition \ref{f6}  implies that
  $w(x)=z(x)^{-1/2}v_1(\sigma(x))$
  gives a decaying solution of \eqref{e4yy} on $[M,\infty)$, and thus
  \begin{align}\label{f50}
  \theta(x,\zeta,h)=Y(x,\zeta,h)\begin{pmatrix}K(x,\zeta,h)\begin{pmatrix}w\\hw_x\end{pmatrix}\\0\end{pmatrix}
\end{align}
is a decaying solution of \eqref{f43}.  Here we use the fact that the explicit estimates of $\eta_1$ and $\partial_\xi\eta_1$ given in
Theorem 9.1 of Chapter 11 of \cite{O} imply their contributions to $w(x)$ and $w_x(x)$ decay as well.
The matrix $K$ involves a factor of $e^{\varphi_0/h}$, so here we have used Remark \ref{f47a}.  When $\Re\zeta=0$, Remark \ref{f47a} and the formula for $w$ imply that $\theta$ is the desired bounded and  oscillating, but not decaying, solution of \eqref{f43}.

 \end{proof}

In the proof of Proposition \ref{f9z} we will sometimes speak of ``relative errors of size $O(p)$" defined as follows.
\begin{defn}[Relative error]\label{f6y}
When a term $\eta(p)$ depending on a small parameter $p$ (and possibly other variables) in an expression $A=B+\eta$ satisfies for some positive constant $C$,
\begin{align}\label{f6x}
|\eta|\leq C|p| |B|,
\end{align}
uniformly with respect to all the variables on which $A$, $B$, and $\eta$ depend,
we say that $\eta$ is a \emph{relative error of size $O(p)$}. When $\eta=\eta_1+\dots+\eta_N$ and
$\eta_j$ satisfies $|\eta_j|\leq C|p| |B|$, we say that  $\eta_j$ \emph{contributes  a relative error of size $O(p)$}.
\end{defn}

\begin{proof}[Proof of Proposition \ref{f9z}]
\textbf{1. }The proof is based on the formula \eqref{f50}, the expression for $w(x)$ given by Proposition \ref{f6}, and a standard expansion of the Airy function.

  Recall the definitions of the variables
\begin{align}\label{f51}
t=\frac{2}{\mu}\sqrt{aD(\infty,\zeta)}e^{-\mu x/2},\;z=\frac{t}{h},\;\tilde\gamma=-i\tilde\beta,\;\sigma=\frac{z}{\tilde\gamma}.
\end{align}
The variable $z$ occurs in equation \eqref{hr4}(b), but now instead of $W(z)$ we write $w(z)$ and we will abuse notation by writing, for example, $w(z)=w(x)$ to mean $W(z(x))=w(x)$.  The variable $\sigma$ occurs in equation \eqref{d3} and \eqref{d4}.  Recalling the tranformations that relate the dependent variables $w(z)$ of \eqref{hr4}(b) and $v(\sigma)$ of \eqref{d3}, we have
\begin{align}\label{f52}
\begin{split}
&v(\sigma)=\hat w(\tilde \gamma \sigma)=w(\tilde\gamma\sigma)(\tilde\gamma\sigma)^{1/2}=w(z)z^{1/2}, \;z\in\mathcal W/h, \text{ so }\\
&w(z)=z^{-\frac{1}{2}}v(\sigma)=z^{-\frac{1}{2}}\xi_\sigma^{-1/2}(\sigma)\left(Ai_1(\tilde\gamma^{2/3}\xi(\sigma))+\eta_1(\tilde\gamma,\xi(\sigma))\right),\;\sigma\in\mathcal{Z}_{-i\tilde\alpha}.
\end{split}
\end{align}

\textbf{2. }We first express the factor $C(x,\zeta)+hr(x,\zeta,h)$ appearing in the equation for $w(x)$ in terms of $\xi=\xi_{f}(\sigma)$, where $f=f_0+f_p$.  Using Remark \ref{hr4z} and \eqref{f15z}, we obtain
\begin{align}\label{f53}
\begin{split}
&\frac{4}{\mu^2}(C(x,\zeta)+hr(x,\zeta,h))=h^2z^2\left[(1-\frac{\tilde\gamma^2}{z^2})+(h^2z^2+\alpha^2)h_1(hz,\zeta)+hzh_2+hh_3\right]=\\
&\quad-\tilde\alpha^2\sigma^2\left[(1-\frac{1}{\sigma^2})+(\alpha^2-\tilde\alpha^2\sigma^2)h_1(-i\tilde\alpha \sigma,\zeta)-i\tilde\alpha \sigma h_2+hh_3\right]=\\
&\quad\quad-\tilde\alpha^2\sigma^2 f(\sigma)=-\tilde\alpha^2\sigma^2\xi(\xi_\sigma)^2.
\end{split}
\end{align}
The function $\sqrt{\xi}$ was defined on $\mathcal{Z}_{cut}(z_0)$ just below \eqref{f15}, so we can use the equation
\begin{align}\label{f54}
-\frac{\mu}{2}i\tilde\alpha\sigma\sqrt{\xi}\xi_\sigma=\sqrt{C(x,\zeta)+hr(x,\zeta,h)}
\end{align}
to define a branch of $\sqrt{C+hr}$ on the corresponding $x-$domain.  Since the argument of $-i\tilde\alpha\sigma\sqrt{\xi}\xi'$ is close to zero for $x$ near $M$,\footnote{This is because $z$ is large with $\arg z\sim 0$ for $x$ near $M$.} we have
\begin{align}\label{f55}
\sqrt{C+hr}=-\sqrt{\zeta^2+c_0^2\eta(x)}\;\ub(x)+O(h)=-s(x,\zeta)\ub(x)+O(h)\text{ for }x\text{ near }M
\end{align}
and thus
\begin{align}\label{f56}
-\frac{\mu}{2}i\tilde\alpha\sigma\sqrt{\xi}\xi_\sigma=\frac{\mu}{2}hz\sqrt{\xi}\xi_\sigma=-s(x,\zeta)\ub(x)+O(h)\text{ for }x\text{ near }M.
\end{align}

\textbf{3. Preliminaries. }We will use the standard asymptotic expansions valid for $|z|$ large on $|\arg z|\leq \pi-\delta$:
\begin{align}\label{f57}
\begin{split}
&Ai(z)\sim \frac{e^{-\chi}}{2\sqrt{\pi}z^{1/4}}\sum^\infty_0(-1)^s\frac{u_s}{\chi^s}, \\
&Ai'(z)\sim -\frac{z^{1/4}e^{-\chi}}{2\sqrt{\pi}}\sum^\infty_0(-1)^s\frac{v_s}{\chi^s},\text{ where } \;\chi=\frac{2}{3}z^{3/2},\;\;u_0=v_0=1.
\end{split}
\end{align}

In the expression for the approximate solution $\theta_1$,
\begin{align}\label{f58z}
\theta_1(x,\zeta,h)=e^{\frac{1}{h} h_1(x,\zeta)+k_1(x,\zeta)}T_1(x,\zeta),
\end{align}
we have
\begin{align}\label{f58}
\begin{split}
&T_1=P_0+sQ_0\text{ and, with }\mu_1(x,\zeta)=\ua+s\ub, \text{ where }\ua=-\frac{\kappa^2\zeta}{\eta u},\;\ub=-\frac{\kappa}{\eta u},\\
&h_1(x,\zeta)=\int^x_0\mu_1(x',\zeta)dx'=\int^x_0 \ua(x',\zeta)dx'+\int^x_0 s(x',\zeta)\ub(x') dx':= h_{1a}+h_{1b}.
\end{split}
\end{align}
Since $d_x\varphi_0(x,\zeta,h)=\frac{a+d}{2}=\frac{\ua+\ud}{2}+O(h)=\ua+O(h)$, we obtain
\begin{align}\label{f59}
\varphi_0-h_{1a}=O(h)+C_a(\zeta,h) \text{ near } x=M, \text{ where }C_a(\zeta,h)=O(1).
\end{align}

\textbf{4. Approximations. }Using the formula \eqref{f52} for $w(z)$ and the expansions \eqref{f57}, and setting $\psi=e^{-\frac{2\pi i}{3}}\tilde\gamma^{2/3}\xi$, for $x$ near $M$ we approximate\footnote{Here the roots of $\psi$ are defined for $|\arg\psi|\leq \pi-\delta$.}
\begin{align}\label{f60}
\begin{split}
&(a)\; w(z)\sim z^{-1/2}\xi_\sigma^{-1/2}Ai_1(\tilde\gamma^{2/3}\xi)\sim\frac{1}{2\sqrt\pi}\;z^{-1/2}\xi_\sigma^{-1/2}e^{-\frac{2}{3}\psi^{3/2}}\psi^{-1/4}\\
&(b)\; w_z(z)\sim z^{-1/2}\xi_\sigma^{-1/2}Ai_1'(\tilde\gamma^{2/3}\xi)\tilde\gamma^{2/3}\xi_\sigma\frac{1}{\tilde\gamma}\sim -\frac{1}{2\sqrt\pi}\;z^{-1/2}\xi_\sigma^{-1/2}e^{-\frac{2\pi i}{3}}e^{-\frac{2}{3}\psi^{3/2}}\psi^{1/4}\tilde\gamma^{-1/3}\xi_\sigma.
\end{split}
\end{align}
In the first ``$\sim$" of \eqref{f60}(a) we have ignored the $\eta_1$ contribution to $w$, while in the second ``$\sim$"  we have ignored contributions from terms in the expansion of $Ai(z)$ corresponding to $s\geq 1$.  The computations below will make it clear that these approximations contribute relative errors of size $O(1/\tilde\beta)$  in our approximation of $\theta(x,\zeta,h)$.  In the approximation \eqref{f60}(b) we have ignored similar terms contributing relative errors of the same size.  In addition, we have ignored the term $d_z\left(z^{-1/2}\xi_\sigma^{-1/2}\right)Ai_1(\tilde\gamma^{2/3}\xi)$, which contributes a relative error of size $O(h)$. Thus, we obtain,
\begin{align}\label{f61}
hw_x=-\frac{\mu}{2}hzw_z \sim \frac{\mu}{2}h\frac{1}{2\sqrt\pi}\;z^{1/2}\xi_\sigma^{-1/2}e^{-\frac{2\pi i}{3}}e^{-\frac{2}{3}\psi^{3/2}}\psi^{1/4}\tilde\gamma^{-1/3}\xi_\sigma
\end{align}
for $x$ near $M$.

\textbf{5. }Using the formula \eqref{f50} for the exact decaying solution $\theta$, we find
\begin{align}\label{f62}
\theta(x,\zeta,h)\sim e^{\frac{\varphi_0}{h}}[b^{1/2}wP_0+b^{-1/2}(hw_x)Q_0].
\end{align}
Here we have ignored relative errors of size $O(h)$ by ignoring the $O(h)$ entries in $Y_2$ (recall $Y=Y_1Y_2$) and the (2,1) entry of $K$, which is of size $O(h)$.  Plugging in \eqref{f60}(a) and \eqref{f61} we obtain\footnote{Here we use $\sqrt{\psi}=e^{-\frac{\pi i}{3}}\tilde\gamma^{1/3}\sqrt{\xi}$.}
\begin{align}\label{f63}
\begin{split}
&\theta\sim e^{\frac{\varphi_0}{h}-\frac{2}{3}\psi^{3/2}}\left(\frac{1}{2\sqrt{\pi}}b^{1/2}z^{-1/2}\xi_\sigma^{-1/2}\psi^{-1/4}\right)\left[P_0+\frac{\mu}{2}b^{-1}hz\psi^{1/2}e^{-\frac{2\pi i}{3}}\tilde\gamma^{-1/3}\xi_\sigma Q_0\right]=\\
&\quad e^{\frac{\varphi_0}{h}+\frac{2}{3}\tilde\gamma\xi^{3/2}}\left(\frac{1}{2\sqrt{\pi}}b^{1/2}z^{-1/2}\xi_\sigma^{-1/2}\psi^{-1/4}\right)\left[P_0-\frac{\mu}{2}b^{-1}hz\sqrt{\xi}\xi_\sigma Q_0\right].
\end{split}
\end{align}
From \eqref{f56} and $b=\ub+O(h)$ we find
\begin{align}\label{f64}
\begin{split}
&-\frac{\mu}{2}b^{-1}hz\sqrt{\xi}\xi_\sigma=s(x,\zeta)+O(h),\\
&d_x(\frac{2}{3}\tilde\gamma\xi^{3/2})=-\sqrt{\xi}\xi_\sigma\frac{\mu}{2}z=\frac{s\ub}{h}+O(1)=\frac{1}{h}d_x h_{1b}+O(1)\Rightarrow \frac{2}{3}\tilde\gamma\xi^{3/2}=\frac{h_{1b}}{h}+\frac{C_b(\zeta,h)}{h}+O(1)
\end{split}
\end{align}
near $x=M$.  With \eqref{f59} we obtain
\begin{align}\label{f65}
\frac{\varphi_0}{h}+\frac{2}{3}\tilde\gamma\xi^{3/2}=\frac{h_1(x,\zeta)}{h}+g(x,\zeta,h);
\end{align}
here $g=\frac{g_1(\zeta,h)}{h}+g_2(x,\zeta,h)$ with $g_1=O(1)$ and $g_2=O(1)$ near $x=M$.
Using \eqref{f64} and ignoring another $O(h)$ relative error, we can now rewrite \eqref{f63}
\begin{align}\label{f66}
\theta\sim e^{\frac{h_1(x,\zeta)}{h}+g}\left(\frac{1}{2\sqrt{\pi}}b^{1/2}z^{-1/2}\xi_\sigma^{-1/2}\psi^{-1/4}\right)T_1=G(x,\zeta,h)\theta_1(x,\zeta,h)\text{ near }x=M,
\end{align}
where the nonvanishing scalar function
\begin{align}\label{f67}
G(x,\zeta,h)=e^ge^{-k_1}\left(\frac{1}{2\sqrt{\pi}}b^{1/2}z^{-1/2}\xi_\sigma^{-1/2}\psi^{-1/4}\right).
\end{align}
Setting
\begin{align}\label{f68}
H(x,\zeta,h)=G^{-1}(x,\zeta,h),
\end{align}
we obtain the estimate of Proposition \ref{f9z}.

\end{proof}


\subsection{Proofs for Regime I}
This subsection gives the proofs of Propositions \ref{e4z}, \ref{e3}, \ref{e4y}, and \eqref{e4w}.
We begin by examining the change of variable $\sigma\to\xi_f(\sigma)$.

\begin{proof}[Proof of Proposition \ref{e4z}]
The proof is parallel to that of Proposition \ref{f12} for Regime II, but simpler.

\textbf{1. }The analyticity of $\xi_f$ follows immediately from the fact that $f+f_p$ is nonvanishing on $\mathcal{Z}_{\tilde\alpha}$ for $N_p$ sufficiently small.  This nonvanishing makes Regime I much easier to treat than Regime II.

\textbf{2. Estimates of $\xi_f-\xi_{f_0}$. }Here we provide the analogue of Lemma \ref{f11a} for Regime I.  For $N_p$ small we have
\begin{align}\label{m1}
\sqrt{f_0+f_p}=\sqrt{f_0}+O(f_p/\sqrt{f_0}) \text{ on }\mathcal{Z}_{\tilde\alpha}.
\end{align}
Thus, given  $K>>1$, there exists a positive constant $\eps =\eps(N_p)$, which can be taken to approach $0$ as $N_p\to 0$, such that
\begin{align}\label{m2}
\begin{split}
&|\xi_f(\sigma)-\xi_{f_0}(\sigma)|\leq \eps\text{ for }|\sigma|\leq K\\
&|\xi_f(\sigma)-\xi_{f_0}(\sigma)|\leq \eps |\xi_{f_0}(\sigma)|\text{ for }|\sigma|\geq  K.
\end{split}
\end{align}

\textbf{3. Injectivity. }Parallel to the proof of Proposition \ref{f12}, we divide $\mathcal{Z}_{\tilde\alpha}$ into subregions $A$, $B$, and $C$ consisting of $\sigma$ with respectively small, medium, and large modulus, and first prove injectivity on each subregion.   The arguments used to treat regions $B$ and $C$ in the case of Regime II can be repeated (almost) verbatim here.  The treatment of Region A is much the same as before, but easier.  Again, one starts with \eqref{f17} and shows that the integral has large modulus.   The case of adjacent regions can be treated as in Regime II to finish the proof.

\end{proof}

\begin{proof}[Proof of Proposition \ref{e3}]
In order to apply Theorem 3.1 in \cite{O}, Chapter 10, there are three requirements:

a)\;We must choose a suitable subdomain $\Delta_\xi$ of the $\xi$ plane on which to solve \eqref{e3ww}.  The domain should include the image of an interval $[M,\infty)$ under the map $x\to \xi$ (here, $x\in T_{M,R}$ as in \eqref{a1}), where $M$ can be chosen independent of the parameters $(\zeta,h)$.

b)\;It must be possible to choose ``progressive paths" (defined below) for all points in the domain.

c)\;The integrals \eqref{e2} should all be finite, with bounds independent of the choice of path and the parameters  $\zeta$ and $h$.

\textbf{1.  Definition of progressive paths. }Let $\Delta\subset \bC$ be an open, connected set  and let $\partial\Delta$ denote its boundary.

a)\;We say that progressive $1-$paths can be chosen in $\Delta$ provided there exists a point $\alpha_1\in \partial\Delta$, possibly at infinity, such that any point $\xi\in\Delta$ can be linked to $\alpha_1$ by a path $\mathcal{P}_1$ in $\Delta$ such that as $v$ traverses $\mathcal{P}_1$ from $\alpha_1$ to $\xi$,  the quantity $\Re (\tilde\beta v)$ is nondecreasing.

b)\;We say that progressive $2-$paths can be chosen in $\Delta$ provided there exists a point $\alpha_2\in \partial\Delta$, possibly at infinity, such that any point $\xi\in\Delta$ can be linked to $\alpha_2$ by a path $\mathcal{P}_2$ in $\Delta$ such that as $v$ traverses $\mathcal{P}_2$ from $\alpha_2$ to $\xi$,  the quantity $\Re (\tilde\beta v)$ is nonincreasing.

The paths are assumed to have a parametrization with the same regularity as described in Definition \ref{f27aa}(b).

\textbf{2. Choice of the domain $\Delta_\xi$. }At first we ignore the right boundary segment of $\mathcal{Z}_{\tilde\alpha}$ and treat this wedge as if it were infinite.

For small positive constants $\kappa$, $\eps$ define a domain $\Delta_\xi(\kappa,\eps)$ to be the open set whose boundary consists of the segments
\begin{align}\label{m4}
\{t+i\eps:t\leq 0\},\;\{t-i\eps:t\leq 0\}, \;\{t+i(\kappa t+\eps):t\geq 0\},\;\{t-i(\kappa t+\eps):t\geq 0\}.
\end{align}
Recall that we have
\begin{align}\label{m3}
\xi_{f_0}(\sigma)=\begin{cases}\log(\frac{\sigma}{2})+1+o(1)\text{ for }|\sigma|\text{ small}\\\sigma+o(1)\text{ for }|\sigma|\text{ large}\end{cases}.
\end{align}
Together with the formula \eqref{e1} for $\xi_{f_0}$, this implies that when $\arg\tilde\alpha\sim 0$,  the open set
$\xi_{f_0}(\mathcal{Z}_{\tilde\alpha})$ contains a set of the form $\Delta_\xi(\kappa,\eps)$ for some choice of $\kappa$, $\eps$.  Proposition \ref{e4z} and the estimates \eqref{m2} then imply, after further reduction of $N_p$ if necessary, that the perturbed domain
$\xi_{f}(\mathcal{Z}_{\tilde\alpha})$ also contains a subdomain of the form $\Delta_\xi(\kappa,\eps)$ for some smaller $\kappa$ and $\eps$.\footnote{Helpful drawings of the range of $\xi_{f_0}$ are given in figures 7.1 and 7.2 of Chapter 10 of \cite{O}.}

Recall that the dilated wedge $\mathcal{Z}_{\tilde\alpha}$ has a right boundary arc of radius $\eps_2/|\tilde\alpha|>>1$ for $\eps_2$ as in Definition \ref{defW}.  We define $\Delta_\xi(\kappa,\eps,\eps_2)$ to be the bounded open set obtained by cutting off $\Delta_\xi(\kappa,\eps)$ with this boundary arc.  We then repeat the procedure above to conclude that $\xi_f(\mathcal{Z}_{\tilde\alpha})$ contains a subdomain of the form $\Delta_\xi=\Delta_\xi(\kappa,\eps,\eps_2')$ for some $\eps_2'<\eps_2$ (but close to $\eps_2$).   Finally, we define the subdomain $\mathcal{Z}_{\tilde\alpha,s}$ appearing in the statement of Proposition \ref{e3} as
\begin{align}
\mathcal{Z}_{\tilde\alpha,s}:=\xi_f^{-1}(\Delta_\xi(\kappa,\eps,\eps_2')).
\end{align}
Provided $N_p$ is small enough,
this domain contains the image of $[M',\infty)$ under the map $x\to \sigma$, where $M'$ is slightly greater than $M$.

Next consider the other extreme case where $\arg\tilde\alpha = \frac{\pi}{2}-\delta$. The wedge $\mathcal{Z}_{\tilde\alpha}=\mathcal {W}/\tilde\alpha$ then consists of points $\sigma$ with
\begin{align}\label{m5}
-\eps_1-\frac{\pi}{2}+\delta<\arg \sigma <\eps_1-\frac{\pi}{2}+\delta,\;0<|\sigma|<\eps_2/|\tilde\alpha|
\end{align}
for $\eps_1<\delta$ as in Definition \ref{defW}.   Using \eqref{m3} and the formula \eqref{e1} for $\xi_{f_0}$, we see that $\xi_{f_0}(\mathcal{Z}_{\tilde\alpha})$
contains a domain, call it $\Delta_\xi(\rho_1,\rho_2,\eps_2)$, similar to $\Delta_\xi(\kappa,\eps,\eps_2)$ above, \emph{except} that the part of $\Delta_\xi(\rho_1,\rho_2,\eps_2)$ corresponding to small (respectively, large) $|\sigma|$ consists of points satisfying\footnote{There is a sharp bend in the domain, downward and to the right, which occurs near points $\xi_{f_0}(\sigma)$ for $\sigma$ close to $-i$, since $f_0(-i)=0$. However, note that $-i\notin \mathcal{Z}_{\tilde\alpha}$.}
\begin{align}
\rho_1<\Im \xi<\rho_2,\;\;\mathrm{respectively,}\;\; \rho_1<\arg \xi<\rho_2,
\end{align}
for constants $\rho_j$ such that
\begin{align}
-\eps_1-\frac{\pi}{2}+\delta<\rho_1<\rho_2<\eps_1-\frac{\pi}{2}+\delta.
\end{align}
As above the estimates \eqref{m2} imply that for $N_p$ small the perturbed domain $\xi_{f}(\mathcal{Z}_{\tilde\alpha})$ contains a set $\Delta_\xi=\Delta_\xi(\rho_1,\rho_2,\eps_2)$ of the same form for a slightly different choice of $(\rho_1,\rho_2,\eps_2
)$, and we define
\begin{align}
\mathcal{Z}_{\tilde\alpha,s}:=\xi_f^{-1}(\Delta_\xi(\rho_1,\rho_2,\eps_2')).
\end{align}
As before this set can be chosen to include the  image of $[M',\infty)$ under the map $x\to \sigma$, where $M'$ is slightly greater than $M$.

Domains $\Delta_\xi$ corresponding to other choices of $\tilde\beta$ in Regime I are chosen by the method just described.  If we write $\tilde\alpha=(a_1+ia_2)$, a progressive $1-$path is characterized by the property that its tangent vector $v_1+iv_2$ at any given point satisfies $v_1a_1-v_2a_2\geq 0$; that is, the vector $(v_1,v_2)$ makes an angle $\leq \frac{\pi}{2}$ with $(a_1,-a_2)$.   A sketch of the range of admissible tangent vectors  shows that progressive $1-$paths can be chosen in the domain $\Delta_\xi$ described above, if we take $\alpha_1$ to be any point at left infinity in $\Delta_\xi$. Similar considerations show that progressive $2-$paths can be chosen if $\alpha_2$ is taken to be a point on the right boundary arc of $\Delta_\xi$ where $\Re(\tilde\beta\xi)$ is maximized.

\textbf{3. Finiteness of the integrals $\int^\xi_{\alpha_j}|\psi(r)|\;d|r|$. }The argument is much like that for Regime II, so here we focus on the main differences.  First observe that since $\xi(\sigma)=\int^\sigma_{\sigma_0}\sqrt{f(s)}ds$,
\begin{align}\label{m6a}
\int_\mathcal P |\psi(\xi)|\;d|\xi|=\int_{\xi^-1(\mathcal P)}|\psi(\xi(\sigma))\sqrt{f(\sigma)}|\;d|\sigma|,
\end{align}
for a given path $\mathcal P$ in $\Delta_\xi$.  So we must check the finiteness of the integral on the right at $0$ and $\infty$.

We have $f=f_0+f_p$ where $f_0$ and $f_p$ are now defined in \eqref{d2}, and
\begin{align}\label{m6}
\psi(\xi_f(\sigma))=\frac{g(\sigma)}{f(\sigma)}+\frac{4f(\sigma)f''-5f'^2}{16f^3},\;\text{ where }g(\sigma)=-\frac{1}{4\sigma^2}\text{ and }f'=d_\sigma f.
\end{align}
Observe that for $N_p$ small,
\begin{align}\label{m7a}
\sqrt{f}(\sigma)\sim\sqrt{f_0}(\sigma)\sim\begin{cases}\frac{1}{\sigma}\text{ for }|\sigma|\text{ small }\\1\text{ for }|\sigma|\text{ large }\end{cases}.
\end{align}
Letting $\psi_0(\sigma)$ denote the function obtained by setting $f_p=0$ on the right in \eqref{m6}, we have
\begin{align}\label{m7}
\psi_0(\sigma)=\frac{1}{4}\frac{\sigma^2(4-\sigma^2)}{(1+\sigma^2)^3},
\end{align}
so the integral on the right in  \eqref{m6a}, with $\psi(\xi(\sigma))$ replaced by $\psi_0(\sigma)$, is integrable at $0$ and at $\infty$. We note that in the computation of $\psi_0(\sigma)$, a bad term of order $O(1)$ near $\sigma=0$ cancels out.

Next define $\psi_1(\sigma)$ by
\begin{align}\label{m8}
\psi(\xi_f(\sigma))=\psi_0(\sigma)+\psi_1(\sigma).
\end{align}
Writing
\begin{align}\label{m9}
\frac{1}{f_0+f_p}= \frac{1}{f_0}(1-\frac{f_p}{f_0}+...), \; \frac{1}{(f_0+f_p)^3}=\frac{1}{f_0^3}(1-3\frac{f_p}{f_0}+...),
\end{align}
we see that the main contribution of $g/f$ to $\psi_1$ is
\begin{align}
-\frac{g f_p}{f_0^2}= \begin{cases}O(\sigma^2)\text{ near }\sigma =0\\O(\frac{1}{\sigma^2})\text{ near }\infty\end{cases},
\end{align}
so the corresponding contributions to \eqref{m6a} are finite.

It remains to consider the contribution of $(4ff''-5f'^2)/16f^3$ to $\psi_1$. The terms involving second derivatives have the same form as the terms in \eqref{f35} \emph{after} setting the factor of $\xi$ there equal to one.   The terms in $f_p''$ have the same form as \eqref{f36}, and the estimates \eqref{f37} still apply.  We estimate the contribution of one of the ``worst terms" appearing in $\frac{f_p''}{f_0^2}$, namely the one corresponding to the term $\tilde\alpha^2 h h_3''$ in \eqref{f36}.  When $|\sigma|$ is large we have
\begin{align}\label{m10}
|\frac{\sigma^4}{(\sigma^2+1)^2}\tilde\alpha^2 hh_3''|\leq C| \tilde\alpha^2 h \frac{1}{h}|\leq C/|\sigma|^2.
\end{align}
The corresponding contribution of \eqref{m10} to \eqref{m6a} is thus integrable near  infinity.  When $|\sigma|$ is small,
\begin{align}\label{m11}
|\frac{\sigma^4}{(\sigma^2+1)^2}\tilde\alpha^2 hh_3''|\leq C| \sigma^4\tilde\alpha^2 h \frac{1}{h}|=| \sigma^{4}\tilde\alpha^{2}|\leq |\sigma|^{4},
\end{align}
so the corresponding contribution to \eqref{m6a} is integrable near $\sigma=0$.

The estimates corresponding to the remaining terms in $\psi_1$ are similar to those above.

\textbf{4. Conclusion. }We have now checked that all the requirements for an application of Theorem 3.1 of \cite{O}, Chapter 10 are satisfied, so this concludes the proof of Proposition \ref{e3}.
\end{proof}

\begin{proof}[Proof of Proposition \ref{e4y}]
The image of $[M,+\infty)$ under the map $x\to \xi(\sigma(x)))$ is a curve that remains close to the real  axis and approaches left infinity in $\Delta_\xi$ as $x\to \infty$.  Thus, $\Re(\tilde\beta\xi(\sigma(x)))\to -\infty$ as $x\to\infty$ for $\tilde\beta$ in Regime I.   Since $\xi_\sigma(\sigma)=O(\frac{1}{\sigma})$ for $\sigma$ near $0$ and $z=\sigma\tilde\beta$, we have
\begin{align}\label{m12}
z^{-1/2}(x)\xi_\sigma^{-1/2}(\sigma(x))=O(1/|\tilde\beta|^{1/2})\text{ for large }|x|.
\end{align}
Together with the estimates for $\eta_1$ in Proposition \ref{e3}, the above statements imply that for $w(x)$ given by \eqref{e4x},  $(w,hw_x)$ is a decaying solution of \eqref{e4yy}.

\end{proof}

\begin{proof}[Proof of Proposition \ref{e4w}]
The proof is parallel to that of Proposition \ref{f9z}, so we focus on the main differences.
Recall the definitions of the variables
\begin{align}\label{n1}
t=\frac{2}{\mu}\sqrt{aD(\infty,\zeta)}e^{-\mu x/2},\;z=\frac{t}{h},\;\sigma=\frac{z}{\tilde\beta}.
\end{align}
With notation similar to \eqref{f52} we write $w(z)$ for the unknown function $W(z)$  in \eqref{hr4}(b) and
\begin{align}\label{n2}
w(z)=z^{-\frac{1}{2}}v_1(\sigma)=z^{-\frac{1}{2}}\xi_\sigma^{-1/2}(\sigma)\left(e^{\tilde\beta \xi(\sigma)}+\eta_1(\tilde\beta,\xi(\sigma))\right),\;\sigma\in\mathcal{Z}_{\tilde\alpha}.
\end{align}

\textbf{2. } Using Remark \ref{hr4z} and $\xi_\sigma=\sqrt{f}$, we obtain
\begin{align}\label{n3}
\begin{split}
&\frac{4}{\mu^2}(C(x,\zeta)+hr(x,\zeta,h))=
\tilde\alpha^2\sigma^2\left[(1+\frac{1}{\sigma^2})+(\alpha^2+\tilde\alpha^2\sigma^2)h_1(\tilde\alpha \sigma,\zeta)+\tilde\alpha \sigma h_2+hh_3\right]=\\
&\quad\quad\tilde\alpha^2\sigma^2 f(\sigma)=\tilde\alpha^2\sigma^2\xi_\sigma^2.
\end{split}
\end{align}
Thus,
\begin{align}\label{n4}
\frac{\mu}{2}\tilde\alpha\sigma\xi_\sigma=\frac{\mu}{2}hz\xi_\sigma=\sqrt{C(x,\zeta)+hr(x,\zeta,h)}=-s(x,\zeta)\ub(x)+O(h)\text{ for }x\text{ near }M.
\end{align}

\textbf{3. Approximations. }Using the formula \eqref{n2}, for $x$ near $M$ we approximate
\begin{align}\label{n5}
\begin{split}
&(a)\; w(z)\sim z^{-1/2}\xi_\sigma^{-1/2}e^{\tilde\beta\xi}\\
&(b)\; w_z(z)\sim z^{-1/2}\xi_\sigma^{-1/2}e^{\tilde\beta\xi}\tilde\beta\xi_\sigma\frac{1}{\tilde\beta}=z^{-1/2}\xi_\sigma^{1/2}e^{\tilde\beta\xi}
\end{split}
\end{align}
In \eqref{n5}(a) we have ignored an $O(1/|\tilde\beta|)$ relative error coming from the $\eta_1$ contribution to $w$.  In the approximation \eqref{n5}(b) we have ignored a similar term contributing a relative error of the same size.  In addition, we have ignored the term $d_z\left(z^{-1/2}\xi_\sigma^{-1/2}\right)e^{\tilde\beta\xi}$, which contributes a relative error of size $O(h)$. Thus, we obtain,
\begin{align}\label{n6}
hw_x=-\frac{\mu}{2}hzw_z \sim -\frac{\mu}{2}h z^{1/2}\xi_\sigma^{1/2}e^{\tilde\beta\xi}
\end{align}
for $x$ near $M$.

\textbf{4. }Using the formula \eqref{f50} for the exact decaying solution $\theta$, we find as before
\begin{align}\label{n7}
\theta(x,\zeta,h)\sim e^{\frac{\varphi_0}{h}}[b^{1/2}wP_0+b^{-1/2}(hw_x)Q_0].
\end{align}
Plugging in \eqref{n5}(a) and \eqref{n6} we obtain
\begin{align}\label{n8}
\theta\sim e^{\frac{\varphi_0}{h}+\tilde\beta\xi}\left(b^{1/2}z^{-1/2}\xi_\sigma^{-1/2}\right)\left[P_0-\frac{\mu}{2}b^{-1}hz\xi_\sigma Q_0\right].
\end{align}
From \eqref{n4} and $b=\ub+O(h)$ we find
\begin{align}\label{n9}
\begin{split}
&-\frac{\mu}{2}b^{-1}hz\xi_\sigma=s(x,\zeta)+O(h),\\
&d_x(\tilde\beta\xi)=-\xi_\sigma\frac{\mu}{2}z=\frac{s\ub}{h}+O(1)=\frac{1}{h}d_x h_{1b}+O(1)\text{ near }x=M,
\end{split}
\end{align}
for $h_{1b}$ as in \eqref{f58}.
As in \eqref{f65} we obtain
\begin{align}\label{n10}
\frac{\varphi_0}{h}+\tilde\beta\xi=\frac{h_1(x,\zeta)}{h}+g(x,\zeta,h),\text{ near }x=M,
\end{align}
for a function $g$ as in \eqref{f65}.
Using \eqref{n9} and ignoring another $O(h)$ relative error, we can now rewrite \eqref{n8}
\begin{align}\label{n11}
\theta\sim e^{\frac{h_1(x,\zeta)}{h}+g}\left(b^{1/2}z^{-1/2}\xi_\sigma^{-1/2}\right)T_1=G(x,\zeta,h)\theta_1(x,\zeta,h)\text{ near }x=M,
\end{align}
where the nonvanishing scalar function
\begin{align}\label{n12}
G(x,\zeta,h)=e^ge^{-k_1}\left(b^{1/2}z^{-1/2}\xi_\sigma^{-1/2}\right).
\end{align}
Setting
\begin{align}\label{n13}
H(x,\zeta,h)=G^{-1}(x,\zeta,h),
\end{align}
we obtain the estimate of Proposition \ref{e4w}.

\end{proof}

\section{Regime III}\label{Three}

In this section we prove Propositions \ref{g8t}, \ref{g12}, \ref{g13z}, and \ref{g17z}.
Recall that  $f=f_0+f_p$, where
\begin{align}\label{j0}
\begin{split}
&f_0(s)=\frac{1}{s}\text{ and }f_p(s)=\frac{1}{s}\left[(4s+\alpha^2)h_1(2s^{1/2},\zeta)+2s^{1/2}h_2(2s^{1/2},\zeta)+hh_3(2s^{1/2},\zeta,h)\right].
\end{split}
\end{align}
First we prove Proposition \ref{g8t}, which concerns the  change of variable defined by
\begin{align}\label{j1}
2\xi^{1/2}(s)=\int^s_0f^{1/2}(r)dr \text{ for }s\in\mathcal{W}^2/4.
\end{align}

\begin{proof}[Proof of Proposition \ref{g8t}]
For $N_p$ small we have
\begin{align}\label{j2}
\sqrt{f}=\frac{1}{\sqrt{s}}(1+\eps_1(s)) \text{ where }|\eps_1(s)|<<1;
\end{align}
thus, $\xi(s)$ is analytic on $\mathcal{W}^2/4$.   From \eqref{j2} and \eqref{j1} we obtain
\begin{align}\label{j3}
\sqrt{\xi(s)}=\sqrt{s}(1+\eps_2(s)),\;\;|\eps_2(s)|<<1,
\end{align}
and thus, since $\xi^{-1/2}\xi_s=\sqrt{f}$, we have
\begin{align}\label{j4}
\xi_s(s)=1+\eps_3(s), \text{ where }|\eps_3(s)|<<1.
\end{align}
This implies injectivity on $\mathcal{W}^2/4$ since
\begin{align}
|\xi(s_1)-\xi(s_2)|=\left|(s_1-s_2)\int^1_0\xi_s(s_2+r(s_1-s_2))dr\right|\geq \frac{1}{2}|s_1-s_2|.
\end{align}

\end{proof}

The proof of Proposition \ref{g12} can be based on Theorem 9.1 of Chapter 12 of \cite{O} in the case where $\tilde\beta\geq 0$. However, the latter theorem  does not treat the case of $\tilde \beta$ nonreal needed here, and the  proof given in \cite{O} fails in that case.\footnote{For example, the properties of the weight function $\mathfrak{E}_\nu(z)$ defined in (8.08) of Chapter 12 of \cite{O} are derived using the fact that when $\nu\geq 0$, the modified Bessel function $K_\nu(z)$ does not vanish in $|\arg z|\leq \pi/2$.  But when $\nu=i|\nu|\neq 0$, for example, $K_\nu$ has infinitely many zeros on the positive real axis \cite{FS}.}
We show next how the proof of this theorem can be modified to treat the  case $\Re\tilde\beta\geq 0$.

\begin{proof}[Proof of Theorem 9.1 of Chapter 12 of \cite{O} for $\Re\tilde\beta\geq 0$.]

\textbf{1. }The modified argument uses the following estimates for the Bessel functions $I_\nu$, $K_\nu$ proved in section 16 of \cite{O2}.  Let $\textbf{M}$ denote a bounded subset of the half-plane $\Re\nu\geq 0$. For $\nu\in \textbf{M}$ and $|\arg z|\leq \pi/2$ we have
\begin{align}\label{g14}
|I_\nu(z)|\leq k V_\nu(z)\quad  |K_\nu(z)|\leq k X_\nu(z),
\end{align}
where
\begin{align}\label{g15}
\begin{split}
&V_\nu(z)=\frac{|z^\alpha e^z|}{1+|z|^{\alpha+\frac{1}{2}}},\quad X_\nu(z)=\ell_\nu(z) \frac{1+|z|^\alpha}{1+|z|^{\frac{1}{2}}}\frac{e^{-z}}{|z|^\alpha},\\
&\ell_\nu(z)=\ln\frac{1+2|z|}{|z|}\;(|\nu|<\delta),\quad \ell_\nu(z)=1 \;(|\nu|\geq \delta),
\end{split}
\end{align}
where $\alpha=\Re \nu\geq 0$ and $\delta$ is an arbitrary number in the range $0<\delta<\frac{1}{2}$.  The constant $k$ is independent of $\mu$ and $z$, but depends on $\delta$.

\textbf{2. }Next, in place of the weight function $\mathfrak{E}_\nu$ defined in (8.08) of \cite{O}, Chapter 12, we redefine $\mathfrak{E}_\nu$ as
\begin{align}\label{g16}
\mathfrak{E}_\nu(z):=\left(\frac{V_\nu(z)}{X_\nu(z)}\right)^{1/2}, \text{ for }\nu\in M, \;|\arg z|\leq \pi/2.
\end{align}
It is easy to check that for $\nu\in\mathbf{M}$
\begin{align}\label{g17}
\begin{split}
&\mathfrak{E}_\nu(z)\sim\begin{cases}|e^z|,\;|z| \text{ large }\\|z|^\alpha,\;|z|\text{ small }\end{cases}\text{ for }|\nu|\geq \delta,\\
&\mathfrak{E}_\nu(z)\sim\begin{cases}\ln 2 \;|e^z|,\;|z| \text{ large }\\(\ln \frac{1}{|z|})^{-\frac{1}{2}}|z|^\alpha,\;|z|\text{ small }\end{cases}\text{ for }|\nu|< \delta.
\end{split}
\end{align}
For $|z|$ of intermediate size $\mathfrak{E}_\nu(z)$ is continuous and bounded away from $0$ for each $\nu\in\textbf{M}$;  positive upper and lower bounds can be chosen independently of $\nu\in\textbf{M}$, $|\arg z|\leq \pi/2$.

Following \cite{O} we next define functions $\mathfrak{M}_\nu(z)$ and $\vartheta(z)$ by the equations
\begin{align}
|I_\nu(z)|=\mathfrak{E}_\nu(z)\mathfrak{M}_\nu(z)\cos\vartheta(z),\;\;|K_\nu(z)|=\mathfrak{E}^{-1}_\nu(z)\mathfrak{M}_\nu(z)\sin\vartheta(z),\text{ for }|\arg z|\leq \pi/2.
\end{align}
Thus,
\begin{align}
\mathfrak{M}_\nu(z)=[\mathfrak{E}_\nu^{-2}(z)|I_\nu(z)|^2+\mathfrak{E}_\nu^2(z)|K_\nu(z)|^2]^{1/2}.
\end{align}
Using \eqref{g14} and \eqref{g17} one readily verifies
\begin{align}
\mathfrak{M}_\nu(z)\leq C\begin{cases}\frac{1}{|z|^{1/2}},\;|z|\text{ large }\\1, \;|z|\text{ small, }|\nu|\geq \delta\\(\ln\frac{1}{|z|})^{1/2}\;,|z|\text{ small, }|\nu|< \delta\end{cases},
\end{align}
where $C$ can be chosen independent of $\nu\in M$. One can now define bounded constants $\mu_j$, $j=1,\dots 4$ as in (8.26), (8.27) of \cite{O}, Chapter 12; they can now be chosen independent of $\nu\in\textbf{M}$.

\textbf{3. }With these definitions the remainder of the proof of Theorem 9.1 in \cite{O}, Chapter 12  goes  essentially as before.  For example, in the error estimate for the solution expressed in terms of  $I_\nu$, progressive paths are those along which both $\Re t^{1/2}$ and $|t|$ are nondecreasing as $t$ passes from $0$ to $\xi$.  It follows from this and the properties of $\mathfrak{E}_\nu$ given in and below \eqref{g17} that  $\mathfrak{E}_\nu^{-1}(u\xi^{1/2})\mathfrak{E}_\nu(u t^{1/2})\leq N$, for some $N$ that can be chosen independently of $t$, $\zeta$ and the particular progressive path being considered.   Here $u>0$ is a large parameter, taken to be $\frac{2}{h}$ in our application to Proposition \ref{g12}. Thus, the key estimate (9.08) of \cite{O}, Chapter 12 of the kernel $K(\xi,v)$ in the integral equation for the error term still holds, but with $2$ replaced by a larger constant.\footnote{The estimate of $K(\xi,v)$ just above (9.08) in \cite{O}, Chapter 12 ($\zeta$ is used in place of $\xi$ there) is incorrect, but a slightly modified estimate of $|K(\xi,v)|$ leading to (9.08) is easily given.}

\end{proof}

\begin{proof}[Proof of Proposition \ref{g12}]In order to apply this version of Theorem 9.1 in \cite{O}, Chapter 12, there are three requirements:

a)\;We must choose a suitable subdomain $\Delta_\xi$ of the $\xi$ plane on which to solve \eqref{g10}.  The domain should include the image of an interval $[M,\infty)$ under the map $x\to \xi$ (here, $x\in T_{M,R}$ as in \eqref{a1}), where $M$ can be chosen independent of the parameters $(\zeta,h)$.

b)\;It must be possible to choose ``progressive paths" (defined below) for all points in the domain.

c)\;The integrals \eqref{g11} should all be finite, with bounds independent of the choice of path and the parameters  $\zeta$ and $h$.

\textbf{1.  Definition of progressive paths. }Let $\Delta$ be an open, connected subset of $\{\xi:|\arg\xi|<\pi/2\}$ and let $\partial\Delta$ denote its boundary.   We suppose $0\in\partial\Delta$.

a)We say that progressive $1-$paths can be chosen in $\Delta$ provided that any point $\xi\in\Delta$ can be linked to the origin by a path $\mathcal{P}_1$ in $\Delta$ such that as $v$ traverses $\mathcal{P}_1$ from $0$ to $\xi$, both $\Re{v^{1/2}}$ and $|v|$ are nondecreasing.

b)We say that  progressive $2-$paths can be chosen in $\Delta$ provided there exists a point $\alpha\in\partial\Delta$ with the following property:   any point $\xi\in\Delta$ can be linked to $\alpha$ by a path $\mathcal{P}_2$ in $\Delta$ such that as $v$ traverses $\mathcal{P}_2$ from $\alpha$ to $\xi$, both $\Re{v^{1/2}}$ and $|v|$ are nonincreasing.

The paths are assumed to have a parametrization with the same regularity as described in Definition \ref{f27aa}(b).

\textbf{2. Choice of the domain $\Delta_\xi$. }
Recall the definition of $\mathcal{W}$ from Definition \ref{defW}, we see that
\begin{align}
\mathcal{W}^2/4=\{s\in\bC:|\arg s|<2\eps_1,\;|s|<\eps^2_2/4\}.
\end{align}
The estimate \eqref{j3} implies
\begin{align}\label{j5}
|\xi(s)-s|\leq \eps_0|s|, \text{ where }\eps_0<<1,
\end{align}
and therefore the image of $\mathcal{W}^2/4$ under the map $s\to\xi(s)$ will contain
\begin{align}\label{j6}
\Delta_\xi:=\{\xi\in\bC: |\arg\xi|<\frac{3}{2}\eps_1,\;\;|\xi|<(1-\eps_0)\frac{\eps_1^2}{4}\}.
\end{align}
If we take $\alpha$ to be the point on the right boundary arc of $\Delta_\xi$ where $\Re\xi^{1/2}$ is maximized, it is obvious that progressive $1-$ and $2-$paths can be chosen in $\Delta_\xi$.  For example, in the $\xi^{1/2}$ plane one can choose these paths to be line segments.   Moreover, the domain $\Delta_\xi$ contains the image of $[M',\infty)$ under the map $x\to\xi$, where $M'$ is slightly greater than $M$ (we have $M'=M+O(|\ln(1-\eps_0)|)$.   We  define the domain $\mathcal{W}_s$ appearing in the statement of Proposition \ref{g12} to be
\begin{align}\label{j7}
\mathcal{W}_s:=\xi^{-1}(\Delta_\xi).
\end{align}

\textbf{3. Finiteness of the integrals $\int^\xi_0|\phi(r)r^{-1/2}|d|r|$. }Since $\Delta_\xi$ is bounded independent of $h$ (and $\zeta$), we need only consider behavior of the integrals near the origin.   Recall that
\begin{align}\label{j8}
\phi(\xi)=\frac{1-4\tilde\beta^2}{16\xi}+\frac{g(s)}{f(s)}+\frac{4f(s)f''(s)-5f^{'2}(s)}{16 f^3(s)}.
\end{align}
where $f=f_0+f_p$ as in \eqref{j0}.  Clearly, we must look for some cancellation of the singularity of $\phi$ due to the vanishing of $\xi$ at $s=0$ and the singularity of $f$ at $s=0$.

Let us first rewrite $f$ as $f(s)=\frac{a}{s}+f_2(s)$, where
\begin{align}\label{j9}
\begin{split}
&a:=1+\alpha^2h_1(0,\zeta)+hh_3(0,\zeta,h)\\
&f_2(s)= \frac{\alpha^2\left(h_1(2s^{1/2},\zeta)-h_1(0,\zeta)\right)}{s}+
\frac{h\left(h_3(2s^{1/2},\zeta,h)-h_3(0,\zeta,h)\right)}{s}+\left(4h_1+\frac{2h_2(2s^{1/2},\zeta)}{s^{1/2}}\right).
\end{split}
\end{align}
The estimates of Proposition \ref{estimates} for the $h_j$ imply that
$f_2(s)=\frac{O(s)}{s}$, and thus
\begin{align}\label{j10}
f(s)=\frac{a}{s}(1+O(s))\Rightarrow\sqrt{f}=\sqrt{\frac{a}{s}}\;(1+O(s))\Rightarrow \xi^{1/2}=\sqrt{as}+O(s^{\frac{3}{2}}).
\end{align}
This gives $\xi(s)=as+O(s^{2})$, and thus
\begin{align}\label{j11}
\frac{1-4\tilde\beta^2}{16\xi}=\frac{1-4\tilde\beta^2}{as}\;(1+O(s))=\frac{1-4\tilde\beta^2}{as}+O(1):=A(s)+B(s).
\end{align}
Set $\tilde f_0(s)=\frac{a}{s}$. A short computation shows
\begin{align}\label{j12}
A(s)+\frac{g(s)}{\tilde f_0(s)}+\frac{4\tilde f_0(s)\tilde f_0''(s)-5\tilde f_0^{'2}(s)}{16 \tilde f_0^3(s)}=0.
\end{align}
Since the contribution of $B(s)$ to
\begin{align}\label{j13}
\int^\xi_0|\phi(r)r^{-1/2}|d|r|
\end{align}
is finite\footnote{Here we use $d\xi=\xi_s ds$ and \eqref{j4}.}, it just remains to examine the contribution of
\begin{align}\label{j14}
\left(\frac{g(s)}{f(s)}+\frac{4f(s)f''(s)-5f^{'2}(s)}{16 f^3(s)}\right)-\left(\frac{g(s)}{\tilde f_0(s)}+\frac{4\tilde f_0(s)\tilde f_0''(s)-5\tilde f_0^{'2}(s)}{16 \tilde f_0^3(s)}\right).
\end{align}
Recall $f=\tilde f_0+f_2$. Thus, the terms in \eqref{j14} involving second derivatives are (ignoring some constant factors)\footnote{Compare \eqref{f35}.}
\begin{align}\label{j15}
f_2(\tilde f_0''+f_2'')(\frac{1}{\tilde f_0^3}-3\frac{ f_2}{\tilde f_0^4}),\;\;\;\;
\tilde f_0f''_2(\frac{1}{\tilde f_0^3}-3\frac{ f_2}{\tilde f_0^4}),\;\;\;\;
\tilde f_0\tilde f_0''\frac{ f_2}{\tilde f_0^4}
\end{align}
Setting $q(2s^{1/2},\zeta,h)=h_3(2s^{1/2},\zeta,h)-h_3(0,\zeta,h)$, we consider for example the contribution of $\tilde f_2:=hq/s$ to
$f_2''/\tilde f_0^2$ (one of the ``worst" terms in \eqref{j15}).  We compute
\begin{align}\label{j16}
\frac{\tilde f_2''}{\tilde f_0^2}=\frac{s^2}{a^2} \tilde f_2''=\frac{s^2}{a^2}h[q_{tt}s^{-2}-\frac{5}{2}q_ts^{-5/2}+2qs^{-3}]=\frac{1}{a^2}h[q_{tt}-\frac{5}{2}q_ts^{-1/2}+2qs^{-1}].
\end{align}
The estimates of Proposition \ref{estimates} show that the right side of \eqref{j16} is $O(1)$, so its contribution to the integrand of \eqref{j13} is $O(s^{-\frac{1}{2}})$, which is integrable near $0$.   The terms in \eqref{j14} involving first derivatives are estimated similarly.

\textbf{4. Conclusion. }We have now checked that all the requirements for an application of Theorem 9.1 of \cite{O}, Chapter 12 are satisfied, so this concludes the proof of Proposition \ref{g12}.

\end{proof}

Next we show that for $\Re\zeta>0$ the decaying solution of \eqref{e4yy} is given by
\begin{align}\label{j17}
w(x)=\frac{\sqrt{2}}{t(x)}\hat v_1(s(x))=\frac{\sqrt{2}}{2s^{1/2}}\xi_s^{-1/2}\left(\xi^{1/2}I_{\tilde\beta}(2\xi^{1/2}/h)+\eta_1(\tilde\beta,\xi)\right).
\end{align}

\begin{proof}[Proof of Proposition \ref{g13z}]
As $x\to \infty$ we have $s\to 0$ and $\xi(s)\to 0$.
Recall from \eqref{j3} and \eqref{j4} that
\begin{align}
\xi(s)=s+\eps_a(s),\;\xi_s(s)=1+\eps_b(s),\;\text{ where }|\eps_j(s)|<<1,
\end{align}
so in estimating $w(x)$ we can ignore the factors multiplying $I_{\tilde\beta}$ and $\eta_1$.  For $|z|$ small with $|\arg z|\leq \frac{\pi}{2}$ we have
\begin{align}
|I_{\tilde\beta}(z)|\leq k|z|^{\Re\tilde\beta}.
\end{align}
Since $\Re\tilde\beta>0$ for $\Re\zeta>0$, this implies decay of the term involving $I_{\tilde\beta}$ as $x\to \infty$.  The estimate of $\eta_1$  in Theorem 9.1, Chapter 12 of \cite{O} implies that this contribution decays to zero as well.
Differentiating \eqref{j17} and arguing as above we obtain that $w_x$ also decays to $0$ as $x\to\infty$.

\end{proof}

We now show that the exact decaying solution $\theta$ of Erpenbeck's system \eqref{f43} identified in Proposition \ref{g13z} is of type $\theta_1$ at $x=M$.

\begin{proof}[Proof of Proposition \ref{g17z}]
The  proof runs parallel to that of Proposition \ref{f9z}.  The variables are
\begin{align}\label{j18}
t=\frac{2}{\mu}\sqrt{aD(\infty,\zeta)}e^{-\mu x/2},\; t=2s^{1/2},\; 2\xi^{1/2}=\int^s_0\sqrt{f(r)}\;dr.
\end{align}
\textbf{1. }We recall from Remark \ref{hr4z} that
\begin{align}
\frac{4}{\mu^2}(C(x,\zeta)+hr(x,\zeta,h))=t^2\left[(1+\frac{\tilde\alpha^2}{t^2})+(t^2+\alpha^2)h_1(t,\zeta)+th_2+hh_2\right].
\end{align}
Using \eqref{j18} and   recalling the definition of $f(s)$ \eqref{j0}, we rewrite this as
\begin{align}\label{j19}
\frac{4}{\mu^2}(C(x,\zeta)+hr(x,\zeta,h))=\tilde\alpha^2+4s^2f(s)=\tilde\alpha^2+4s^2\xi^{-1}\xi_s^2.
\end{align}
For $x$ near $M$ and $\zeta\in\omega$ we have $|C(x,\zeta)|>k>0$, $\arg C(x,\zeta)\sim 0$.  We have $\tilde\alpha=O(h)$ in Regime III, so \eqref{j19} implies
\begin{align}\label{j20}
\mu s\xi^{-1/2}\xi_s=\sqrt{C+hr}+O(h)=-s(x,\zeta)\ub(x)+O(h)\text{ for }x\text{ near }M.
\end{align}

\textbf{2. }For $|z|$ large with $|\arg z|<\frac{\pi}{2}$ we have asymptotic expansions (\cite{AS}, Chapter 9)
\begin{align}\label{j21}
\begin{split}
&I_{\tilde\beta}(z)\sim\frac{e^z}{\sqrt{2\pi z}}(1+O(1/z)),\\
&I'_{\tilde\beta}(z)\sim\frac{e^z}{\sqrt{2\pi z}}(1+O(1/z)),
\end{split}
\end{align}
where $(1+O(1/z))$ can be expanded explicitly in powers of $z^{-1}$.

\textbf{3. Approximations. }Using the formula \eqref{j17} for $w$, the expansions \eqref{j21}, and the fact that
\begin{align}
d_t(\xi^{1/2}(s(t)))=\xi^{-1/2}\xi_s\frac{t}{4},
\end{align}
 we approximate for $x$ near $M$
\begin{align}\label{j22}
\begin{split}
&w\sim \frac{\sqrt{2}}{t}\xi_s^{-1/2}\xi^{1/2}I_{\tilde\beta}(2\xi^{1/2}/h)\sim\frac{1}{\sqrt{2\pi}}\;t^{-1}\xi_s^{-1/2}\xi^{1/4}h^{1/2}e^{2\xi^{1/2}/h}\\
&hw_t\sim h\frac{\sqrt{2}}{t}\xi_s^{-1/2}\xi^{1/2}I_{\tilde\beta}'(2\xi^{1/2}/h)\frac{1}{h}\xi^{-1/2}\xi_s\frac{t}{2}\sim\sqrt{\frac{2}{\pi}}\frac{1}{4}\xi_s^{1/2}h^{1/2}\xi^{-1/4}e^{2\xi^{1/2}/h}.
\end{split}
\end{align}
Here we have ignored relative errors of size $O(h)$ associated with higher order terms in the expansions \eqref{j21}, with $\eta_1$, and with other terms in the expression for $hw_t$.  This gives
\begin{align}\label{j23}
hw_x=-\frac{\mu}{2}htw_t\sim-\sqrt{\frac{2}{\pi}}\frac{\mu}{8}\;t\xi_s^{1/2}h^{1/2}\xi^{-1/4}e^{2\xi^{1/2}/h}.
\end{align}

\textbf{4. }Using the formula \eqref{f50} for $\theta$ and ignoring $O(h)$ relative errors as in \eqref{f62}, we obtain
\begin{align}\label{j24}
\theta(x,\zeta,h)\sim e^{\frac{\varphi_0}{h}}[b^{1/2}wP_0+b^{-1/2}(hw_x)Q_0].
\end{align}
Substituting in the expressions for $w$ and $hw_x$ and using \eqref{j20}, we find
\begin{align}\label{j25}
\begin{split}
&\theta\sim e^{\frac{\varphi_0}{h}+\frac{2\xi^{1/2}}{h}}\left(b^{1/2}\sqrt{\frac{h}{2\pi}}\frac{1}{t}\xi_s^{-1/2}\xi^{1/4}\right)\left[P_0-\frac{\mu}{4}b^{-1}t^2\xi_s\xi^{-1/2}Q_0\right]\sim\\
&\quad\quad e^{\frac{\varphi_0}{h}+\frac{2\xi^{1/2}}{h}}\left(b^{1/2}\sqrt{\frac{h}{2\pi}}\frac{1}{t}\xi_s^{-1/2}\xi^{1/4}\right)\left[P_0+s(x,\zeta)Q_0\right].
\end{split}
\end{align}
We have
\begin{align}
d_x(2\xi^{1/2})=-\xi^{-1/2}\xi_s\mu s=bs(x,\zeta)=\ub s(x,\zeta)+O(h)\text{ near }x=M
\end{align}
for $h_{1b}$ as in \eqref{f58}.   As in \eqref{f65} we obtain
\begin{align}\label{j26}
\frac{\varphi_0}{h}+\frac{2\xi^{1/2}}{h}=\frac{h_1(x,\zeta)}{h}+g(x,\zeta,h) \text{ near }x=M,
\end{align}
for a function $g$ as in \eqref{f65}.
Thus, we can now rewrite \eqref{j25}
\begin{align}\label{j27}
\theta\sim e^{\frac{h_1(x,\zeta)}{h}+g}\left(b^{1/2}\sqrt{\frac{h}{2\pi}}\frac{1}{t}\xi_s^{-1/2}\xi^{1/4}\right)T_1=G(x,\zeta,h)\theta_1(x,\zeta,h)\text{ near }x=M,
\end{align}
where the nonvanishing scalar function
\begin{align}\label{j28}
G(x,\zeta,h)=e^ge^{-k_1}\left(b^{1/2}\sqrt{\frac{h}{2\pi}}\frac{1}{t}\xi_s^{-1/2}\xi^{1/4}\right).
\end{align}
Setting
\begin{align}\label{j29}
H(x,\zeta,h)=G^{-1}(x,\zeta,h),
\end{align}
we obtain the estimate of Proposition \ref{g17z}.

\end{proof}

\part{Proofs for Part \ref{finite}.}\label{pffinite}

\section{Turning points in $(0,\infty)$.}

Here we prove Propositions \ref{q16} and \ref{q30}.

\begin{proof}[Proof of Proposition \ref{q16}]
For $\Re\zeta=0$ and $x<x(\zeta)$ we take
\begin{align}
\rho^{3/2}(x,\zeta)=\frac{3}{2}\int^x_{x(\zeta)}\sqrt{x(\zeta)-y}\sqrt{-d(y,\zeta)} \;dy,
\end{align}
where the square roots are taken to be positive.\footnote{Here we use the fact that $d(y,\zeta  )$ is negative for real $y$ near $x(\zeta)\in\bR$.}   Making the changes of variable $t=\sqrt{x(\zeta)-y}$ and then $t=u\sqrt{x(\zeta)-x} $, we obtain
\begin{align}
\rho^{3/2}(x,\zeta)=-(x(\zeta)-x)^{3/2}\int^1_0 3u^2\sqrt{-d(x(\zeta)+(x-x(\zeta))u^2,\zeta)}\;du,
\end{align}
which implies \eqref{q18}.  The analyticity of $\rho$ in $x$ and $\zeta$ and the properties \eqref{q19}(a)-(c) are evident from the formula \eqref{q18}.  Property \eqref{q19}(d) is proved by differentiating $\rho_x^2\rho=C(x,\zeta)$ with respect to $\zeta$ and evaluating at $x=x(\zeta)$.
The analyticity of both sides of the equation \eqref{q17} implies that $\rho$ is a solution on $\mathcal{O}\times\omega$.

\end{proof}

\begin{proof}[Proof of Proposition \ref{q30}]
\textbf{1. }First we show that appropriate multiples of $\theta_-$ and $\theta_+$ are, respectively, of type $\theta_1$ and $\theta_2$ at $x_R$.   For $\zeta\in\omega_1$ and $x$ near $x_R$, $\rho(x,\zeta)$ takes values near the negative real axis.   Noting that we must take
\begin{align}\label{r1}
-\pi<\arg(h^{-2/3}\rho e^{\pm 2\pi i/3})<\pi
\end{align}
in order to use the expansions to rewrite the expressions in \eqref{q28}, we obtain for $x$ near $x_R$
\begin{align}\label{r2}
\theta_-\sim e^{\frac{\varphi_0}{h}-\frac{2}{3}h^{-1}i(-\rho)^{3/2}} \left(\frac{1}{2\sqrt{\pi}}b^{1/2}\rho_x^{-1/2}(h^{-2/3}\rho e^{-2\pi i/3})^{-1/4}\right) [P_0+s(x,\zeta)Q_0].
\end{align}
Here we have used\footnote{Recall that $\rho_x^2\rho=C(x,\zeta)=s^2\ub^2$ and that for $\zeta=i|\zeta|\in\omega_1$ and real $x$ near $x_R$, we have $s=i|s|=i\sqrt{|\zeta|^2-c_0^2\eta(x)}$.}
\begin{align}\label{r2a}
-\frac{2}{3}(h^{-2/3}\rho e^{-2\pi i/3})^{3/2}=-\frac{2}{3}ih^{-1}(-\rho)^{3/2}\text{ and }i\rho_x b^{-1}(-\rho)^{1/2}=s(x,\zeta)+O(h).
\end{align}
Similarly, we obtain for $x$ near $x_R$
\begin{align}\label{r3}
\theta_+\sim e^{\frac{\varphi_0}{h}+\frac{2}{3}h^{-1}i(-\rho)^{3/2}} \left(\frac{1}{2\sqrt{\pi}}b^{1/2}\rho_x^{-1/2}(h^{-2/3}\rho e^{2\pi i/3})^{-1/4}\right) [P_0-s(x,\zeta)Q_0].
\end{align}

For $x$ near $x_R$ and $\zeta\in\omega_1$ we have
\begin{align}\label{r4}
\begin{split}
&-\frac{2}{3}i(-\rho)^{3/2}(x,\zeta)=\int^x_{x_R-\delta}s\ub(y,\zeta)\;dy-\frac{2}{3}i(-\rho)^{3/2}(x_R-\delta,\zeta)\text{ and so }\\
&\int^x_0 s\ub(y,\zeta)\;dy=-\frac{2}{3}i(-\rho)^{3/2}(x,\zeta)+\int^{x_R-\delta}_0s\ub(y,\zeta)\;dy+\frac{2}{3}i(-\rho)^{3/2}(x_R-\delta,\zeta).
\end{split}
\end{align}
Since $\mu_1(x,\zeta)=\ua+s(x,\zeta)\ub$ and $T_1(x,\zeta)=P_0+sQ_0$, \eqref{r2} implies
\begin{align}\label{r5}
\begin{split}
&\theta_-(x,\zeta,h)\sim\theta_1(x,\zeta,h)G_-(x,\zeta,h)\text{ for $x$ near $x_R$ where }G_-(x,\zeta,h)=\\
&\left(\frac{1}{2\sqrt{\pi}}b^{1/2}\rho_x^{-1/2}(h^{-2/3}\rho e^{-2\pi i/3})^{-1/4}\right)\exp\left[-\frac{1}{h}\left(\int^{x_R-\delta}_0s\ub(y,\zeta)\;dy+\frac{2}{3}i(-\rho)^{3/2}(x_R-\delta,\zeta)\right)\right]e^{k_-(x,\zeta,h)},
\end{split}
\end{align}
for a function $k_-=O(1)$.   Similarly, we obtain from \eqref{r3}
\begin{align}\label{r6}
\begin{split}
&\theta_+(x,\zeta,h)\sim\theta_2(x,\zeta,h)G_+(x,\zeta,h)\text{ for $x$ near $x_R$ where }G_+(x,\zeta,h)=\\
&\left(\frac{1}{2\sqrt{\pi}}b^{1/2}\rho_x^{-1/2}(h^{-2/3}\rho e^{2\pi i/3})^{-1/4}\right)\exp\left[\frac{1}{h}\left(\int^{x_R-\delta}_0s\ub(y,\zeta)\;dy+\frac{2}{3}i(-\rho)^{3/2}(x_R-\delta,\zeta)\right)\right]e^{k_+(x,\zeta,h)},
\end{split}
\end{align}
for a function $k_+=O(1)$.

From \eqref{r5} and \eqref{r6} we see that the functions
\begin{align}\label{r7}
\overline\theta_1:=G_-^{-1}(x_R,\zeta,h)\theta_-(x,\zeta,h)\text{ and }\overline\theta_2:=G_+^{-1}(x_R,\zeta,h)\theta_+(x,\zeta,h)
\end{align}
are exact solutions of \eqref{q1} on $\mathcal{O}$, which are respectively of type $\theta_1$ and $\theta_2$ at $x_R$.\footnote{Caution: It is not necessarily true that $\overline\theta_1$, for example, is of type $\theta_1$ for $x\neq x_R$.}  For later use we note that the growth rates in $h$ of the factors $G_{\mp}^{-1}(x_R,\zeta,h)$ are
\begin{align}\label{r8}
R_\mp(\zeta,h):=h^{-1/6}\exp\left[\pm\frac{1}{h}\Re \left(\int^{x_R-\delta}_0s\ub(y,\zeta)\;dy+\frac{2}{3}i(-\rho)^{3/2}(x_R-\delta,\zeta)\right)\right].
\end{align}

\textbf{2. }Computations like those that produced \eqref{r2} and \eqref{r3} show that for $x$ near $x_L$ we have
\begin{align}\label{r9}
\begin{split}
&\theta_-\sim e^{\frac{\varphi_0}{h}+\frac{2}{3}h^{-1}\rho^{3/2}} \left(\frac{1}{2\sqrt{\pi}}b^{1/2}\rho_x^{-1/2}(h^{-2/3}\rho e^{-2\pi i/3})^{-1/4}\right) [P_0+s(x,\zeta)Q_0]\\
&\theta_+\sim e^{\frac{\varphi_0}{h}+\frac{2}{3}h^{-1}\rho^{3/2}} \left(\frac{1}{2\sqrt{\pi}}b^{1/2}\rho_x^{-1/2}(h^{-2/3}\rho e^{2\pi i/3})^{-1/4}\right) [P_0+s(x,\zeta)Q_0],
\end{split}
\end{align}
since $b^{-1}\rho_x\sqrt{\rho}=s(x,\zeta)+O(h)$ for $x$ near $x_L$ and $\zeta\in\omega_1$.\footnote{Recall that $\rho(x,\zeta)>0$ for real $x$ near $x_L$ and for $\zeta\in\omega_1$ such that $\zeta=i|\zeta|$.}  From \eqref{r9} and the fact that
$T_1=P_0+sQ_)$ it is evident that
\begin{align}\label{r9a}
\theta_-(x,\zeta,h)\sim\theta_1(x,\zeta,h)K_-(x,\zeta,h)\text{ for $x$ near $x_L$}
\end{align}
for a nonvanishing scalar function $K_-$.

\textbf{3. Exact solutions $\overline\theta_i$, $i=3,4,5$. }
 After shrinking the neighborhoods $\mathcal O$ and $\omega_1$ and reducing $\delta>0$  if necessary, we choose an open ball $B(\uzeta,R)$ centered at $x(\uzeta)$ such that
\begin{align}\label{r10}
x(\omega_1)\cup [x_L,x_R] \subset\mathcal{O}\subset B(x(\uzeta),R/2),
\end{align}
and such that the profile $p(x)$ has an analytic extension to $B(\uzeta,R)$.
The exact solutions $\overline\theta_i(x,\zeta,h)$  are constructed for $\zeta\in\omega_1$ from approximate solutions $\theta_i$ of the form
\eqref{t10}, which are defined initially on $[0,x_L]$, and then extended to a simply connected neighborhood of $\{x_L,x_R\}$ by analytic continuation in
\begin{align}\label{r11}
\cS:=B(x(\uzeta),R)\cap \{x:\Im x\geq 0\}\setminus x(\omega_2), \text{ where }\omega_1\subset\subset \omega_2
\end{align}
and $\omega_2$ is a slight enlargement of $\omega_1$.  As explained in section 4.2 of \cite{LWZ1}, the $\overline\theta_i$
are exact solutions of \eqref{q1} and satisfy\footnote{The proof by a contraction argument is based on being able to choose ``progressive paths" in $\cS$; see Theorem 3.1 of \cite{LWZ1}.}
\begin{align}\label{r12}
|\overline\theta_i(x,\zeta,h)-\theta_i(x,\zeta,h)|\leq Ch|\theta_i(x,\zeta,h)| \text{ in }\cS\text{ for }\zeta\in\omega_1.
\end{align}
Like $\overline\theta_i$, $i=1,2$, the functions $\overline\theta_i$, $i=3,4,5$ are  solutions of \eqref{q1} in a full neighborhood of $x(\zeta)$ for $\zeta\in\omega_1$; however, the asymptotic behavior \eqref{r12} is known only in $\mathcal{S}$.

\textbf{4. Growth rates. }From the expressions \eqref{t10} for the $\theta_i$ we can read off the growth rates with respect to $h$ of the $\overline\theta_i(x,\zeta,h)$, $i=1,...,5$ at $x_R$ for $\zeta\in\omega_1$:
\begin{align}\label{r13}
\begin{split}
&\overline\theta_1(x_R,\zeta,h): \;e^{\frac{1}{h}\Re \int^{x_R}_0[\ua(y,\zeta)+s(y,\zeta)\ub(y)]\;dy}:=e^{A(\zeta)/h}\\
&\overline\theta_2(x_R,\zeta,h): \;e^{\frac{1}{h}\Re \int^{x_R}_0[\ua(y,\zeta)-s(y,\zeta)\ub(y)]\;dy}:=e^{B(\zeta)/h}\\
&\overline\theta_i(x_R,\zeta,h), i\geq 3: \;e^{\frac{1}{h}\Re \int^{x_R}_0 \frac{\zeta}{u(y)}\;dy}:=e^{C(\zeta)/h}
\end{split}
\end{align}

\textbf{5. Expand $H(x_R,\zeta,h)\theta(x,\zeta,h)$. }The exact bounded (or decaying) solution $H(x_R,\zeta,h)\theta$ on $[x_R,\infty)$ extends analytically to a complex neighborhood of $[0,\infty]$.  On $\cS$ we can expand it as
\begin{align}\label{r14}
 H(x_R,\zeta,h)\theta(x,\zeta,h)=c_1(\zeta,h)\overline\theta_1+\dots+c_5(\zeta,h)\overline\theta_5, \text{ for }\zeta\in\omega_1.
 \end{align}
 Corollary \ref{q5a} implies that  $H(x_R,\zeta,h)\theta(x,\zeta,h)$ is of type $\theta_1$ at $x_R$.  Evaluating \eqref{r14} at $x_R$ and using Cramer's rule and \eqref{r13}, we determine the growth rates of the coefficients in \eqref{r14}:
\begin{align}\label{r15}
c_1(\zeta,h)=1+O(h),\;c_2=O(he^{(A(\zeta)-B(\zeta))/h}),\;c_i=O(he^{(A(\zeta)-C(\zeta))/h}), i\geq 3,
\end{align}
where
\begin{align}\label{r16}
A(\zeta)-B(\zeta)=2\Re\int^{x_R}_0 s(y,\zeta)\ub(y)\;dy\text{ and }A(\zeta)-C(\zeta)=\Re\int^{x_R}_0\left(s(y,\zeta)\ub(y)-\frac{\zeta}{\eta u(y)}\right)dy.
\end{align}

\textbf{6. Conclusion. }Using \eqref{r7}, \eqref{r8}, \eqref{r9} and \eqref{r15}, we can now read off the growth rates at $x_L$ of the individual terms in the expansion \eqref{r14}:
\begin{align}\label{r17}
\begin{split}
&(a)\;c_1(\zeta,h)\overline\theta_1(x_L,\zeta,h):\; (1+O(h))\cdot R_-(\zeta,h)\cdot h^{1/6}\exp\left[\frac{1}{h}\Re \left(\int^{x_L}_0\ua(y,\zeta)\;dy+\frac{2}{3}\rho^{3/2}(x_L,\zeta)\right)\right],\\
&(b)\;c_2(\zeta,h)\overline\theta_2(x_L,\zeta,h): \;he^{(A(\zeta)-B(\zeta))/h}\cdot R_+(\zeta,h)\cdot h^{1/6}\exp\left[\frac{1}{h}\Re \left(\int^{x_L}_0\ua(y,\zeta)\;dy+\frac{2}{3}\rho^{3/2}(x_L,\zeta)\right)\right],\\
&(c)\; c_i(\zeta,h)\overline\theta_i(x_L,\zeta,h),\;i\geq 3:\; he^{(A(\zeta)-C(\zeta))/h}\cdot e^{\Re\frac{1}{h}\int^{x_L}_0\frac{\zeta}{u(y)}\;dy}.
\end{split}
\end{align}

First we compare the rates in \eqref{r17}(a),(b).  Recalling the expressions \eqref{r8} for $R_\pm$, and noting  from \eqref{r16} that
\begin{align}
e^{(A(\zeta)-B(\zeta))/h}\cdot \exp\left(-\frac{1}{h}\Re\int^{x_R-\delta}_0 s(y,\zeta)\ub\;dy\right)\leq \exp\left(\frac{1}{h}\Re\int^{x_R-\delta}_0 s(y,\zeta)\ub\;dy\right)
\end{align}
and  from  Remark \ref{q19a} that
\begin{align}\label{r18a}
\Im (-\rho)^{3/2}(x_R-\delta,\zeta)\leq 0 \text{ for }\zeta\in\omega_1,
\end{align}
we obtain
\begin{align}\label{r18}
|c_2(\zeta,h)\overline\theta_2(x_L,\zeta,h)|/|c_1(\zeta,h)\overline\theta_1(x_L,\zeta,h)|\leq Ch.
\end{align}

Next we compare the rates in \eqref{r17}(a),(c). From \eqref{r18a}, the fact that $\Re\rho^{3/2}(x_L,\zeta)>0$, and
\begin{align}
e^{\frac{1}{h}\left(A(\zeta)-\Re\int^{x_R}_0\frac{\zeta}{u(y)}\;dy+\Re \int^{x_L}_0\frac{\zeta}{u(y)}\;dy\right)}\leq e^{\frac{1}{h}\Re\int^{x_R-\delta}_0 s(y,\zeta)\ub\;dy}\cdot e^{\frac{1}{h}\Re\int^{x_L}_0 \ua(y,\zeta)\;dy},
\end{align}
we see that
\begin{align}\label{r19}
|c_3(\zeta,h)\overline\theta_3(x_L,\zeta,h)|/|c_1(\zeta,h)\overline\theta_1(x_L,\zeta,h)|\leq Ch.
\end{align}
Thus, $c_1(\zeta,h)\overline\theta_1(x_L,\zeta,h)$ is, for small $h$,  the dominant term in the expansion \eqref{r14} evaluated at $x_L$.  Since $c_1(\zeta,h)=1+O(h)$, we see from \eqref{r9a} and \eqref{r7} that the estimate  of Proposition \ref{q30} holds with $\alpha(\zeta,h):=G_-(x_R,\zeta,h)K_-^{-1}(x_L,\zeta,h)$.

\end{proof}

\section{The turning point at $0$.}

This section gives the proof of Proposition \ref{q34}.

\begin{proof}[Proof of Proposition \ref{q34}]
\textbf{1. Basis of exact solutions near $0$. }As noted before the statement of Proposition \ref{q34}, we have exact solutions $\theta_\pm$ on $\mathcal{O}\ni 0$ satisfying $\theta_\pm(x,\zeta,h)\sim$
\begin{align}\label{s1}
\begin{split}
\quad e^{\varphi_0/h}\left[b^{1/2}(\rho_x)^{-1/2} Ai(h^{-2/3}\rho e^{\pm2\pi i/3})P_0+b^{-1/2}h^{1/3}(\rho_x)^{1/2}e^{\pm2\pi i/3} Ai'(h^{-2/3}\rho e^{\pm 2\pi i/3})Q_0\right]
\end{split}
\end{align}
modulo $O(h)$ errors.  Exact solutions $\overline\theta_1$ and $\overline\theta_2$, which are respectively  of type $\theta_1$ and $\theta_2$ at $x_R=2\delta$, are again given by the formulas \eqref{r7}.  Here the functions $G^{-1}_\mp(2\delta,\zeta,h)$ have growth rates in $h$, $R_\mp(\zeta,h)$, given by \eqref{r8}.

To construct exact solutions $\overline\theta_j$, $j=3,4,5$, near $x=0$, we use the block diagonal form provided by our extension of Proposition \ref{q7} to a neighborhood of $x=0$.   The $3\times 3$ block $A_{22}(x,\zeta,h)$ in \eqref{q6} has semisimple eigenvalues
\begin{align}
\mu_j^*(x,\zeta,h)=\mu_j(x,\zeta)+O(h)=\frac{\zeta}{h}+O(h), \;j=3,4,5.
\end{align}
Since this block has no turning points, we can apply standard results (for example, Theorem 3.1 of \cite{LWZ1}) to construct exact solutions $\phi_{2,j}(x,\zeta,h)$ of $d_x\phi_{2,j}=A_{22}(x,\zeta,h)\phi_{2,j}$ on $[0,3\delta]$ satisfying
\begin{align}\label{s2}
|\phi_{2,j}(x,\zeta,h)-e^{\frac{1}{h}\int^x_0\mu_j(s,\zeta)ds}a_j(x,\zeta,h)|\leq Ch|e^{\frac{1}{h}\int^x_0\mu_j(s,\zeta)ds}|,\;j=3,4,5
\end{align}
for appropriate $a_j=O(1)$. We then obtain exact solutions $\overline\theta_j$ of type $\theta_j$ on $[0,3\delta]$ by setting
\begin{align}
\overline\theta_j=Y(x,\zeta,h)\begin{pmatrix}0\\ \phi_{2,j}\end{pmatrix}, \;j=3,4,5,
\end{align}
where $Y$ is the conjugator of Proposition \ref{q7}.

We note that the elements of the basis $\mathcal B=\{\overline\theta_1\dots,\overline\theta_5\}$ have the growth rates at $x_R=2\delta$ given by \eqref{r13}.

\textbf{2. Expand $H(2\delta,\zeta,h)\theta$. }As in \eqref{r14} we expand the exact solution  $H(2\delta,\zeta,h)\theta$ in the basis $\mathcal{B}$ and, after evaluating at $x_R=2\delta$, we again obtain the growth rates \eqref{r15} for the coefficients $c_j(\zeta,h)$, $j=1,\dots,5$.

\textbf{3. Regime A. }We  show that for $(\zeta,h)$ in Regime A, the term $c_1\overline\theta$ is the dominant term in the expansion \eqref{r14} at $x=0$. Observe that  for \emph{all} $\zeta\in\omega_1$ we have
\begin{align}\label{s3}
\begin{split}
&(a)\;\arg\rho(0,\zeta)\in[0,\pi], \text{ and thus }\\
&(b)\;\arg(e^{-2\pi i/3}\rho(0,\zeta))\in [-2\pi /3,\pi/3], \text{ while }\\
&(c)\;\arg(e^{2\pi i/3}\rho(0,\zeta))\in [2\pi /3,5\pi/3].
\end{split}
\end{align}
In case (b) the zeroes of $Ai(z)$, which all lie on the negative real axis, are avoided; thus, there exist positive constants $A_i$ such that
\begin{align}\label{s4}
A_1\leq |Ai(h^{-2/3}\rho(0,\zeta)e^{-2\pi i/3})|\leq A_2\text{ for }(\zeta,h)\text{ in Regime A.}
\end{align}
Ignoring an error of size $h^{-1/3}$, we have
\begin{align}\label{s4a}
\theta_-(0,\zeta,h)\sim b^{1/2}(\rho_x)^{-1/2} Ai(h^{-2/3}\rho e^{-2\pi i/3})P_0:=q(0,\zeta,h)P_0\text{ in Regime A}.
\end{align}
With \eqref{r7}, \eqref{r8}, and \eqref{r15} this gives for some positive constant $C$,
\begin{align}\label{s5}
|c_1(\zeta,h)\overline \theta_1(0,\zeta,h)|\geq CR_-(\zeta,h)=Ch^{-1/6}\exp\left[\frac{1}{h}\Re \left(\int^{\delta}_0s\ub(y,\zeta)\;dy+\frac{2}{3}i(-\rho)^{3/2}(\delta,\zeta)\right)\right].
\end{align}
Similarly, we obtain
\begin{align}\label{s6}
\begin{split}
&|c_2(\zeta,h)\overline \theta_2(0,\zeta,h)|\leq Che^{\frac{1}{h}2\Re \left(\int^{2\delta}_0s\ub(y,\zeta)\;dy\right)}R_+(\zeta,h), \text{ where }\\
&\qquad
R_+(\zeta,h)=h^{-1/6}\exp\left[-\frac{1}{h}\Re \left(\int^{\delta}_0s\ub(y,\zeta)\;dy+\frac{2}{3}i(-\rho)^{3/2}(\delta,\zeta)\right)\right],\\
&|c_j(\zeta,h)\overline \theta_j(0,\zeta,h)|\leq Che^{\frac{1}{h}\Re\int^{2\delta}_0\left(s(y,\zeta)\ub(y)-\frac{\zeta}{\eta u(y)}\right)dy},\;j=3,4,5,
\end{split}
\end{align}
where in the last estimate we have used $|\overline\theta_j(0,\zeta,h)|=O(1)$, $j=3,4,5$.

Recalling \eqref{r18a}, from \eqref{s5} and \eqref{s6} we see that in Regime A at $x=0$
\begin{align}
|c_2\overline\theta_2/c_1\overline\theta_1|\leq Ch\text{ and }|c_j\overline\theta_j/c_1\overline\theta_1|\leq Ch^{7/6}, \;j=3,4,5,
\end{align}
and thus $c_1\overline\theta_1$ is the dominant term in the expansion \eqref{r14} at $x=0$.  Using \eqref{r7} and \eqref{s4a} we obtain
\begin{align}
\overline\theta_1(0,\zeta,h)=G_-^{-1}(2\delta,\zeta,h)q(0,\zeta,h)P_0.
\end{align}
Since $\theta_1(0,\zeta,h)=P_0+s(0,\zeta)Q_0+O(h)$, for fixed $\kappa>0$ we therefore obtain \eqref{q35} with $\alpha(\zeta,h):=G_-(2\delta,\zeta,h)q^{-1}(0,\zeta,h)$ provided $(\zeta,h)$ lies in Regime A, $\zeta\in\omega_2$, and $0<h\leq h_0$ for small enough $\omega_2\ni\zeta_0$ and $h_0$.

\textbf{4. Regime B with $\arg\rho(0,\zeta)$ away from $\pi/3$. }First we determine the size of $c_1\overline\theta_1(0,\zeta,h)$ in Regime $B$.  By \eqref{s3}(b) we can use the expansions \eqref{f57} of $Ai(z)$ and $Ai'(z)$ for all $(\zeta,h)$ in Regime B to obtain, modulo $O(h)$ relative errors,
\begin{align}\label{s7}
\begin{split}
&c_1\overline\theta_1(0,\zeta,h)\sim k(\zeta,h)R_-(\zeta,h)e^{-\frac{2}{3}(h^{-2/3}\rho(0,\zeta)e^{-2\pi i/3})^{3/2}}\cdot\\
&b^{1/2}\rho_x^{-1/2}(h^{-2/3}\rho(0,\zeta)e^{-2\pi i/3})^{-1/4}\left[P_0-b^{-1}\rho_xe^{-2\pi i/3}h^{1/3}(h^{-2/3}\rho(0,\zeta)e^{-2\pi i/3})^{1/2}Q_0\right],
\end{split}
\end{align}
where $k(\zeta,h)=O(1)$ and is bounded away from $0$.   Letting $\arg\rho(0,\zeta)=\beta\in [0,\pi]$  and noting that for $\zeta$ near $\zeta_0$ the second term inside the brackets is small compared to the first, we obtain
for some positive constant $K$
\begin{align}\label{s8}
c_1\overline\theta_1(0,\zeta,h)\geq Ke^{\frac{1}{h}\Re \left(\int^{\delta}_0s\ub(y,\zeta)\;dy+\frac{2}{3}i(-\rho)^{3/2}(\delta,\zeta)\right)}|\rho(0,\zeta)|^{-1/4}e^{\frac{1}{h}\frac{2}{3}|\rho(0,\zeta)|^{3/2}\cos(\frac{3\beta}{2})}
\end{align}

For a small positive $\eps_0$ we note that for $\beta\in [0,\frac{\pi}{3}-\eps_0]$ (respectively, $\beta\in [\frac{\pi}{3}+\eps_0,\pi]$), $c_2\overline\theta_2$ has an expansion similar to \eqref{s7}, except that $c_1$ is replaced by $c_2$, $R_-$ by $R_+$, and all factors of $e^{-2\pi i/3}$ are replaced by $e^{2\pi i/3}$ (respectively,    $e^{-4\pi i/3}$).  Thus, for $\beta\in [0,\frac{\pi}{3}-\eps_0]$ we obtain
\begin{align}\label{s9}
c_2\overline\theta_2(0,\zeta,h)\leq Che^{\frac{1}{h}2\Re \left(\int^{2\delta}_0s\ub(y,\zeta)\;dy\right)}e^{-\frac{1}{h}\Re \left(\int^{\delta}_0s\ub(y,\zeta)\;dy+\frac{2}{3}i(-\rho)^{3/2}(\delta,\zeta)\right)}|\rho(0,\zeta)|^{-1/4}e^{\frac{1}{h}\frac{2}{3}|\rho(0,\zeta)|^{3/2}\cos(\frac{3\beta}{2})},
\end{align}
while for $\beta\in [\frac{\pi}{3}+\eps_0,\pi]$ we obtain
\begin{align}\label{s10}
c_2\overline\theta_2(0,\zeta,h)\leq Che^{\frac{1}{h}2\Re \left(\int^{2\delta}_0s\ub(y,\zeta)\;dy\right)}e^{-\frac{1}{h}\Re \left(\int^{\delta}_0s\ub(y,\zeta)\;dy+\frac{2}{3}i(-\rho)^{3/2}(\delta,\zeta)\right)}|\rho(0,\zeta)|^{-1/4}e^{-\frac{1}{h}\frac{2}{3}|\rho(0,\zeta)|^{3/2}\cos(\frac{3\beta}{2})},
\end{align}
From \eqref{s9} and \eqref{s8} we see that at $x=0$
\begin{align}\label{s11}
|c_2\overline\theta_2/c_1\overline\theta_1|\leq Ch\text{ for }\beta\in [0,\frac{\pi}{3}-\eps_0].
\end{align}

The same estimate holds for $\beta\in [\frac{\pi}{3}+\eps_0,\pi]$, but this is much less clear since now $\cos(3\beta/2)\leq 0$!   Inspection of \eqref{s8} and \eqref{s10} shows that  the estimate holds for this range of $\beta$ provided
\begin{align}\label{s12}
\Re\left( i(-\rho)^{3/2}(\delta,\zeta)\right)\geq |\rho(0,\zeta)|^{3/2}|\cos(3\beta/2)|\text{ for }\zeta\in\omega_2.
\end{align}
where $\omega_2\subset\omega_1$ is a neighborhood of $\zeta_0$.
Writing $\zeta=\zeta_r+i\zeta_i$ and setting $\gamma=\arg(-\rho(\delta,\zeta))$, we have
\begin{align}\label{s13}
\Im(-\rho^{3/2}(\delta,\zeta))=|\rho(\delta,\zeta)|^{3/2}\sin(3\gamma/2),
\end{align}
and \eqref{q19} implies $\gamma\leq 0$.   Now $\gamma$ is close to zero and, by \eqref{q19}(d), we have  $|\Im\rho(\delta,\zeta)|\geq C|\zeta_r|$, so
\begin{align}\label{s14}
|\sin{3\gamma/2}|\sim |3\gamma/2|\sim \frac{3}{2}|\tan\gamma|=\frac{3}{2}\left|\frac{\Im\rho(\delta,\zeta)}{\Re\rho(\delta,\zeta)}\right|\geq C_1|\zeta_r|/|\rho(\delta,\zeta)|.
\end{align}
With \eqref{s13} this implies
\begin{align}\label{s15}
|\Im(-\rho^{3/2}(\delta,\zeta))|\geq C_1|\zeta_r||\rho(\delta,\zeta)|^{1/2}.
\end{align}
To have \eqref{s12} it now suffices to choose $\omega_2$ so that
\begin{align}\label{s15a}
|\cos(3\beta/2)|\leq C_1\frac{|\zeta_r||\rho(\delta,\zeta)|^{1/2}}{|\rho(0,\zeta)|^{3/2}}\text{ for }\zeta\in\omega_2.
\end{align}

Using \eqref{q19} again, we have
\begin{align}
\rho(0,\zeta)=\rho(0,\zeta_0)+\rho_\zeta(\zeta_0)(\zeta-\zeta_0)+O(|\zeta-\zeta_0|^2)\sim C|\zeta-\zeta_0|=C|\zeta_r,\Im(\zeta-\zeta_0)|.
\end{align}
When $|\rho(0,\zeta)|\sim|\zeta-\zeta_0|\sim |\zeta_r|$, we can choose $\omega_2$ so that $|\rho(\delta,\zeta)|/|\rho(0,\zeta)|$ is large and thereby arrange to have \eqref{s15a}.    When $|\zeta_r|\leq \kappa |\Im(\zeta-\zeta_0)|$ for $\kappa$ small, we must have $\beta$ close to $\pi$. Setting $\alpha=\pi-\beta$ we have
\begin{align}
|\cos(3\beta/2)|=|\sin(3\alpha/2)|\sim|\tan\alpha|\sim\frac{|\Im\rho(0,\zeta)|}{|\Re\rho(0,\zeta)|}\sim\frac{|\zeta_r|}{|\rho(0,\zeta)|}.
\end{align}
We can now shrink $\omega_2$ if necessary, so that $|\rho(\delta,\zeta)|/|\rho(0,\zeta)|$ is large for $\zeta\in\omega_2$, thereby arranging to have \eqref{s15a}.   This establishes \eqref{s12},\footnote{The argument shows that \eqref{s12} holds for $\zeta\in\omega_2$ when $\beta\in [\eps_0,\pi]$.} and thus for $(\zeta,h)$ in Regime $B$ we have at $x=0$,
\begin{align}\label{s16}
|c_2\overline\theta_2/c_1\overline\theta_1|\leq Ch\text{ for }\beta\in [\frac{\pi}{3}+\eps_0,\pi], \;\zeta\in\omega_2.
\end{align}

The estimate \eqref{s6} for $j=3,4,5$ still holds for Regime B, so \eqref{s8} and \eqref{s12} imply that at $x=0$
\begin{align}\label{s16a}
|c_j\overline\theta_j/c_1\overline\theta_1|\leq Ch\text{ for }\;\zeta\in\omega_2
\end{align}
and $(\zeta,h)$ in Regime B, when $\beta\in [0,\frac{\pi}{3}-\eps_0]\cup [\frac{\pi}{3}+\eps_0,\pi]$. With \eqref{s11} and \eqref{s16},
 we obtain \eqref{q35} as before for these $(\zeta,h)$.

\textbf{4. Regime B with $\arg\rho(0,\zeta)$ near $\pi/3$. }To treat $\overline\theta_2$ we now use the fact that  for large $|z|$ with $|\arg z|\leq 2\pi/3$
\begin{align}\label{s17}
\begin{split}
&Ai(-z)\sim\pi^{-1/2}z^{-1/4}\left[\sin(\gamma+\frac{\pi}{4})\sum^\infty_0 a_k\gamma^{-2k}-\cos(\gamma+\frac{\pi}{4})\sum^\infty_0 b_k\gamma^{-2k-1}\right]\\
&Ai'(-z)\sim-\pi^{-1/2}z^{1/4}\left[\cos(\gamma+\frac{\pi}{4})\sum^\infty_0 c_k\gamma^{-2k}+\sin(\gamma+\frac{\pi}{4})\sum^\infty_0 d_k\gamma^{-2k-1}\right],
\end{split}
\end{align}
where $\gamma:=\frac{2}{3}z^{3/2}$ (\cite{AS}, 10.4.60, 10.4.62).   We write, for example,
\begin{align}\label{s18}
Ai(h^{-2/3}\rho(0,\zeta)e^{2\pi i/3})=Ai(-h^{-2/3}e^{-i\pi/3}\rho(0,\zeta)),
\end{align}
where now  $\arg(e^{-i\pi/3}\rho(0,\zeta)):=\theta$ is close to $0$.  We have
\begin{align}\label{s19}
|\cos\left(\frac{2}{3}(h^{-2/3}e^{-\pi i/3}\rho(0,\zeta))^{3/2}+\frac{\pi}{4}\right)|\leq Ce^{\frac{1}{h}\frac{2}{3}|\rho(0,\zeta)|^{3/2}|\sin(\frac{3\theta}{2}|},
\end{align}
and \eqref{s10} now holds with the exponential on the right in \eqref{s19} replacing that on the far right in \eqref{s10}.   Since $\arg(\rho(0,\zeta))$ is near $\pi/3$, we have $|\zeta_r|\sim |\zeta-\zeta_0|\sim |\rho(0,\zeta)|$, so we can arrange to have \eqref{s15a}, with $|\sin(3\theta/2)|$ now in place of
$|\cos(3\beta/2)|$, by choosing $\omega_2$ so that $|\rho(\delta,\zeta)|/|\rho(0,\zeta)|$ is large for $\zeta\in\omega_2$.   Thus, we can obtain the estimates \eqref{s16} and \eqref{s16a} for this range of $\beta$, and the estimate \eqref{q35} is a consequence of these as before.

\end{proof}

\part{Appendices}

\section{Coefficients appearing in the linearized systems}\label{coefficients}

\emph{\quad}The matrix coefficients appearing in the reduced system \eqref{t1} are
\begin{align}\label{A1}
\begin{split}
&A_x=\begin{pmatrix}u&-v&0&0&0\\vp_v&u&0&vp_S&vp_\lambda\\0&0&u&0&0\\0&0&0&u&0\\0&0&0&0&u\end{pmatrix},\;A_y=\begin{pmatrix}0&0&-v&0&0\\0&0&0&0&0\\vp_v&0&0&vp_S&vp_\lambda\\0&0&0&0&0\\0&0&0&0&0\end{pmatrix}\\
&B=\begin{pmatrix}-u'&v'&0&0&0\\p'-v(c_0^2/v^2)'&u'&0&vp_S'&vp_\lambda'\\0&0&0&0&0\\-\Phi_v&S'&0&-\Phi_S&-\Phi_\lambda\\-r_v&\lambda'&0&-r_S&-r_\lambda\end{pmatrix},
\end{split}
\end{align}
where $(')$ denotes differentiation with respect to $x$ and $c_0^2=-v^2p_v(v,S,\lambda)$.

The matrix $\Phi_0(x,\zeta)$ in the system \eqref{t3} is computed in \cite{E3}, p.112 to be
\begin{align}\label{phi0}
\Phi_0(x,\zeta)=\begin{pmatrix}-\frac{(1-\eta)\zeta}{\eta u}&-\frac{m\zeta}{\eta u}&-\frac{im}{1-\eta}&0&0\\-\frac{(1-\eta)\zeta}{\eta mu}&-\frac{(1-\eta)\zeta}{\eta u}&0&0&0\\\frac{i(1-\eta)}{\eta m}&\frac{i}{\eta}&\frac{\zeta}{u}&0&0\\\frac{(1-\eta)p_S\zeta}{\eta m^2 u}&\frac{(1-\eta)p_S\zeta}{\eta m u}&\frac{ip_S}{m}&\frac{\zeta}{u}&0\\\frac{(1-\eta)p_\lambda\zeta}{\eta m^2 u}&\frac{(1-\eta)p_\lambda\zeta}{\eta m u}&\frac{ip_\lambda}{m}&0&\frac{\zeta}{u}\end{pmatrix}.
\end{align}
This computation can be done using \eqref{t3} and \eqref{A1}.

\section{The stability function $V(\zeta,h)$.}\label{stabilityfn}

\emph{\quad}The stability function $V(\zeta,h)$ is given by
\begin{align}\label{u1}
V(\zeta,h)=\theta(0,\zeta,h)\cdot P(0+)-\theta(0,\zeta,h)\cdot\frac{1}{h}(\zeta h_t+i h_y).
\end{align}
Here $m=u/v$, the mass flux, is a constant independent of $x$,
\begin{align}\label{u2}
h_t=\frac{v_--v_+}{v_-T_+\eta_+}\begin{pmatrix}2(1-\eta_+)g_+/m\\T_+\eta_++2(1-\eta_+)g_+\\0\\-m(v_--v_+)\eta_+\\0\end{pmatrix},
\end{align}
and $h_y$ has the single nonzero component $(h_y)_3=m(v_--v_+)$.   By $v_\pm$, for example, we denote  components of the profile states $P_\pm:=P(0\pm)$ just to the right and left of the von Neumann shock, and
\begin{align}\label{V3}
g_+=T_+-\frac{1}{2}(v_--v_+)p_{S+}.
\end{align}
The expression \eqref{u1} for $V(\zeta,h)$, found in \cite{CJLW}, is simpler than the expression derived in \cite{E1} and used in \cite{E2,E3}.  The equality of the two forms of $V$ was proved in section 4 of \cite{CJLW}.

The stability function for the von Neumann shock, $L_1(\zeta)$,
which appears in Assumption \ref{vN}, is given explicitly in \cite{E2} as:
\begin{align}\label{u3}
\begin{split}
&L_1(\zeta)=-\frac{u_-(1-\chi_v)}{\eta_+}\left[\frac{\ell_+\zeta(\zeta+\kappa_+s_+)}{u_+u_-}+\eta_+\left(1-\frac{\zeta^2}{u_+u_-}\right)\right]\\
&\ell=2-(1-\eta)(1-\chi_v)v_-p_S/T,\;\chi_v=v_+/v_-.
\end{split}
\end{align}
From the expression \eqref{u1} and the fact that
\begin{align}
L_1(\zeta)=-T_1(0,\zeta)\cdot (\zeta h_t+ih_y),
\end{align}
it is clear that Assumption \ref{vN} implies that $V(\zeta,h)$ is nonvanishing for small $h$ when $\theta(0,\zeta,h)$ is of type $\theta_1$.

\section{Classical asymptotic ODE results used.}
Here we state the theorems from \cite{O} that are used in this paper.  To keep this section brief, we state the results only in the simplified form that we actually use; also, we refer to earlier parts of this paper for definitions of some terms that appear below.   We note that Theorem \ref{third} below is an extension of Theorem 9.1 of \cite{O}, Chapter 12, to the case where the parameter $\nu$ satisfies $\Re \nu\geq 0$ instead of just $\nu\geq0$. The extension was proved in section \ref{Three}.

For a parameter $u\in\bC$ with $|u|$ large, we consider equations of the form 
\begin{align}\label{basic}
W_{\xi\xi}=(u^2\xi^m+\psi(\xi))W, \text{ where }m=0,1,-1,
\end{align}
on a simply connected, open subset $\Delta$, possibly unbounded,  of the complex $\xi$-plane.   The function $\psi$ is analytic on $\Delta$ but may have singularities at isolated points on its boundary.  The following three theorems deal respectively with the cases $m=0,1,-1$.  

\begin{thm}[Theorem 3.1 of \cite{O}, Chapter 10]\label{first}
Let $m=0$ in \eqref{basic} and suppose $|\arg u|<\pi/2$.  For $j=1,2$ let $\alpha_j\in \partial \Delta$ and suppose that for any $\xi\in\Delta$ a progressive $j-$path can be chosen in $\Delta$ from $\alpha_j$ to $\xi$.\footnote{Such paths are defined in step \textbf{1} of the proof of Proposition  \ref{e3}. }Suppose also that  there is an upper bound for the integrals
\begin{align}\label{ints}
\int^\xi_{\alpha_j}|\psi(s)|d|s| \text{ on progressive $j-$paths},
\end{align}
which is independent of $\xi\in\Delta$.   Then the equation \eqref{basic} has solutions $W_j$ on $\Delta$ satisfying 
\begin{align}
W_j(\xi)=e^{(-1)^{j-1}u\xi}+\eta_j(u,\xi),\;j=1,2,
\end{align}
where the errors $\eta_j$ satisfy the estimates \eqref{e3z}.

\end{thm}

With $Ai(z)$ the standard Airy function, we set 
\begin{align}
Ai_0(z)=Ai(z), \;Ai_1(z)=Ai(ze^{-2\pi i/3}),\;\;Ai_{-1}(z)=Ai(ze^{2\pi i/3}).
\end{align}

\begin{thm}[Theorem 9.1 of \cite{O}, Chapter 11]\label{second}
Let $m=1$ in \eqref{basic} and for small $\delta>0$, suppose $|\arg u|<\delta$.  For $j=0,1,-1$ let $\alpha_j\in \partial \Delta$ and suppose that for any $\xi\in\Delta$ a progressive $j-$path can be chosen in $\Delta$ from $\alpha_j$ to $\xi$.\footnote{Such paths are defined in step \textbf{2} of the proof of Proposition  \ref{f6}. }Suppose also that  there is an upper bound for the integrals
\begin{align}\label{ints2}
\int^\xi_{\alpha_j}|\psi(s)s^{-1/2}|d|s| \text{ on progressive $j-$paths},
\end{align}
which is independent of $\xi\in\Delta$.   Then the equation \eqref{basic} has solutions $W_j$ on $\Delta$ satisfying 
\begin{align}
W_j(\xi)=Ai_j(u^{2/3}\xi)+\eta_j(u,\xi),\;j=0,1,-1,
\end{align}
where the errors $\eta_j$ satisfy the estimates \eqref{f6z}.

\end{thm}

\begin{thm}[Theorem 9.1 of \cite{O}, Chapter 12]\label{third}
Let $m=-1$ in \eqref{basic} and suppose $u>0$.   We now assume $\Delta\subset \{\xi:|\arg\xi|<\pi/2\}$, $0\in\partial \Delta$, and that $\psi(\xi)$ has the form
\begin{align}
\psi(\xi)=\frac{\nu^2-1}{4\xi^2}+\frac{\phi(\xi)}{\xi},
\end{align}
where $\phi$ is analytic at $\xi=0$.
Let $\alpha_1=0$, $\alpha_2\in \partial \Delta$ and suppose that for $j=1,2$ and any $\xi\in\Delta$ a progressive $j-$path can be chosen in $\Delta$ from $\alpha_j$ to $\xi$.\footnote{Such paths are defined in step \textbf{1} of the proof of Proposition  \ref{g12}. }Suppose also that  there is an upper bound for the integrals
\begin{align}\label{ints3}
\int^\xi_{\alpha_j}|\phi(s)s^{-1/2}|d|s| \text{ on progressive $j-$paths},
\end{align}
which is independent of $\xi\in\Delta$.   Then the equation \eqref{basic} has solutions $W_j$ on $\Delta$ satisfying 
\begin{align}
\begin{split}
&(a) W_1(\xi)=\xi^{1/2}I_{\nu}(2u\xi^{1/2})+\eta_1(u,\xi)\\
&(b) W_2(\xi)=\xi^{1/2}K_{\nu}(2u\xi^{1/2})+\eta_2(u,\xi).
\end{split}
\end{align}
where the errors $\eta_j$ satisfy the estimates \eqref{g12b}.

\end{thm}

\end{document}